\documentclass{article}

\usepackage{subfig}
\usepackage{graphicx}
\usepackage{adjustbox}
\usepackage[dvipsnames]{xcolor}
\usepackage{tabularx}
\usepackage{amsmath,amssymb}
\usepackage{amsthm}
\usepackage{dsfont}
\usepackage{mathrsfs} 
\usepackage{gensymb}
\usepackage[left=2cm, right=2cm, top=2.5cm, bottom=2.5cm]{geometry}
\usepackage{cleveref}
\usepackage{cases}
\usepackage{mathtools} 
\usepackage{gensymb} 
\usepackage{caption} 
\usepackage{bm} 
\usepackage{cleveref}
\usepackage{numprint} 
\usepackage{marginnote}
\usepackage{multirow} 
\usepackage{multicol} 
\usepackage{numprint}
\usepackage{longtable}
\usepackage{yfonts}
\usepackage{eufrak}
\usepackage{ulem}
\usepackage{siunitx}
\usepackage{rotating}
\usepackage{algorithm}%
\usepackage{algorithmicx}%
\usepackage{algpseudocode}%
\usepackage{titlesec}
\usepackage{url}

\newcommand{\id}[0]{\text{Id}}
\newcommand{\partialder}[2]{\frac{\partial#1}{\partial#2}}

\newcommand{\secondpartialder}[2]{\frac{\partial^2#1}{\partial#2^2}}
\newcommand{\partialderlong}[2]{ \frac{\partial}{\partial#2}\left(#1\right) }

\newcommand{\myw}[0]{{\scriptstyle W}}
\newcommand{\loss}[0]{\mathscr{L}}

\newcommand{\st}{\text{ s.t. }}

\makeatletter
\newcommand{\slantedparallel}{\mathrel{\mathpalette\new@parallel\relax}}
\newcommand{\new@parallel}[2]{
	\begingroup
	\sbox\z@{$#1T$}
	\resizebox{!}{\ht\z@}{\raisebox{\depth}{$\m@th#1/\mkern-5mu/$}}%
	\endgroup
}
\makeatother

\newcommand{\uut}[0]{{\scriptstyle\mathfrak{U}}}
\newcommand{\vt}[0]{\mathrm{v}}
\newcommand{\ut}[0]{\textfrak{u}}
\newcommand{\wt}[0]{\mathrm{w}}
\newcommand{\m}[1]{\underline#1}

\newcommand{\eps}[0]{\varepsilon}

\newcommand{\node}[0]{$n$-ODE}
\newcommand{\Node}[0]{$N$-ODE}
\newcommand{\Varepsilon}{\scalebox{1.3}{$\varepsilon$}}

\newcommand{\appendixsectionformat}{
	\titleformat{\section}
	{\normalfont\Large\bfseries}
	{Appendix \thesection}{1em}{}
}	
\theoremstyle{definition} 
\newtheorem{obs}{Observation}
\newtheorem{theorem}{Theorem}

\newtheorem{example}{Example}
\newtheorem{definition}{Definition}

\newtheorem{proposition}{Proposition}

\theoremstyle{plain}
\theoremstyle{remark}
\newtheorem{remark}{Remark}

\definecolor{raspberry}{rgb}{1,0,0.2}
\definecolor{darkred}{rgb}{0.5,0.1,0.1}
\definecolor{darkgreen}{rgb}{0.1,0.6,0.2}

\newcommand{\mg}[1]{\color{darkgreen}#1 \color{black}}

\newcommand{\rev}[1]{{\color{black} #1}}

\title{\rev{Model reduction of parametric ordinary differential equations via autoencoders: representation properties and convergence analysis}}

\begin{document}

\author{
	Enrico Ballini$^1$, 
	Marco Gambarini$^{2,3}$, 
    Alessio Fumagalli$^1$,\\
    Luca Formaggia$^1$, 
	Anna Scotti$^1$, 
	Paolo Zunino$^1$
}

\date{\today}

\footnotetext[1]{MOX, Department of Mathematics, Politecnico di Milano, Piazza Leonardo da Vinci 32, 20133 Milano, Italy}
\footnotetext[2]{CMCC Foundation - Euro-Mediterranean Center on Climate Change, Via Marco Biagi 5, 73100 Lecce, Italy}
\footnotetext[3]{RFF-CMCC European Institute on Economics and the Environment, Via Solari 11, 20144 Milano, Italy}

\maketitle

\begin{abstract}
 We propose a reduced-order modeling approach for nonlinear, parameter-dependent ordinary differential equations (ODE). Dimensionality reduction is achieved using nonlinear maps represented by autoencoders. The resulting low-dimensional ODE is then solved using standard integration in time schemes, and the high-dimensional solution is reconstructed from the low-dimensional one. 
We investigate the \rev{architecture} of neural networks for constructing effective autoencoders \rev{that hold necessary properties to reconstruct the input manifold with exact representation capabilities}. We study the convergence of the reduced-order model to the high-fidelity one. Numerical experiments show the robustness and accuracy of our approach \rev{in different scenarios}, highlighting its \rev{effectiveness in highly complex and nonlinear settings} without sacrificing accuracy. Moreover, we examine how the reduction influences the stability properties of the reconstructed high-dimensional solution.   
\end{abstract}


\section{Introduction}
Ordinary differential equations and partial differential equations (PDEs) are ubiquitous in the modeling of time-dependent phenomena in a wide range of disciplines including, among others, fluid dynamics \cite{Thompson1988}, chemical kinetics \cite{Steefel1994}, geological simulations \cite{Nordbotten2012} and solid continuum mechanics \cite{Marsden2012}. A typical approach to solving PDEs involves two steps: the first is to discretize the spatial variables, resulting in a high-dimensional system of ODEs, while the second step regards the integration in time. For simplicity, in this work we refer only to ODEs, with the understanding that the methodologies developed are equally applicable to PDEs in the sense described above. 

The models that capture the phenomena of interest may depend on some parameters, denoted by $\mu$. This variability may reflect uncertainty in physical properties that are difficult to measure, or may arise from the need to explore different scenarios, see, e.g. \cite{Ballini2025}.
When the parameters vary, we obtain parametric ordinary differential equations. In many applications, such as uncertainty quantification, optimization, or inverse problems, it is often necessary to solve these ODEs several times for a wide range of parameter values. Direct computation can be prohibitively expensive, requiring the community to develop reduced‐order modeling techniques that alleviate this computational burden.


In the past few years, scientific machine learning has increasingly incorporated neural networks, like autoencoders for data compression, into reduced-order modeling and methods for improving the efficiency of ODE solvers.
Autoencoders have found widespread application in diverse fields and for multiple purposes. For example, in \cite{Gonzalez2018, Murata2019, Lee2020, Fresca2021, Franco2022, Kadeethum2022}, autoencoders have been employed to compress and subsequently reconstruct discrete solutions that approximate the governing physical equations on a computational grid, possibly using a linear interpolation of the low-dimensional solution \cite{SabziShahrebabaki2022}.


A key strength of autoencoders lies in their ability to achieve maximal information compression, effectively addressing the challenge of slow Kolmogorov $n$-width decay~\cite{Franco2022}. The Kolmogorov $n$-width measures how well a nonlinear solution manifold can be captured by an $n$-dimensional linear subspace; a slow decay indicates that many linear modes are needed to represent the manifold accurately. In contrast, nonlinear autoencoders can approximate the underlying nonlinear manifold with the minimum number of parameters required, thus overcoming a fundamental limitation of the traditional proper orthogonal decomposition (POD) approach~\cite{Franco2022}, where a linear combination of basis vectors is adopted to describe the manifold. This advantage comes from the ability of neural networks to learn nonlinear maps that describe high-dimensional structures using only a few latent variables.
Moreover, autoencoders handle discontinuities effectively \cite{Lee2020} and are robust to non-affine parameter dependencies, unlike the well-established POD.

Neural networks have emerged as a powerful tool for the solution of ODEs. A comprehensive survey of how machine learning techniques can improve ODE solvers is provided in \cite{Legaard2023}.
Further insight in this area is provided in \cite{Yiping2018}, which highlights the similarities between residual neural networks \cite{He2016} and the discrete time-stepping of ODE solvers.
Subsequently, the authors of~\cite{Chen2018} proposed the idea of learning the right-hand side of an ODE directly from the data. They also developed an efficient alternative to standard backpropagation for computing gradients of the loss function with respect to network weights using the adjoint sensitivity method.
The Ph.D. thesis by Kidger~\cite{Kidger2022} offers a detailed exploration of neural differential equations, and is a valuable resource not only for its own contributions but also for the extensive references it contains.
In a related work, \cite{Rubanova2019} demonstrated that recurrent neural networks can effectively model ODE systems even when training data are sampled at irregular time intervals.
Furthermore, \cite{Chen2020} proposed a framework for modeling discontinuous changes, called \textit{events}, in a time-dependent system. 
More recently, \cite{Peter2023} proposed a specific autoencoder architecture for time-dependent PDEs that learns a time-dependent vector field associated with input spatial points without the need for a spatial mesh.  
The work in \cite{Li2025} uses a combination of a convolutional autoencoder to identify the reduced solution and an additional neural network that approximates a one-step time marching procedure in the reduced space. This idea has been tested on many different scenarios in which physical equations are solved.

In the context of reduced-order modeling, the authors in \cite{Hasegawa2020} and \cite{Gupta2022} employ recurrent neural networks to advance in time the reduced solution obtained from an encoder. Rather than training the autoencoder to approximate the identity map, one can design it to predict the solution at the next time step based on previous states. This strategy is illustrated in \cite{Pant2021}, where the autoencoder takes the form of a UNet \cite{Ronneberger2015}. In \cite{Fresca2023} and similarly in \cite{Conti2024, Li2024, Kumar2021}, a feedforward network first computes the reduced representation, and then this state is propagated over a short time interval by a long-short-term memory (LSTM) network. More recently, the work of \cite{Fu2023} has explored the use of self-attention neural networks to perform time evolution.  
The study in \cite{Farenga2025} employs an unknown low-dimensional ODE to approximate the solution of a related high-dimensional ODE. In this framework, the encoder is used solely to specify the initial condition of the low-dimensional system.

To give a comprehensive overview, it is worth mentioning other approaches based on neural networks that can be used to construct surrogate models, such as physics‐informed neural networks \cite{Raissi2019}, the use of graph neural networks to reproduce the solution in parameterized meshes \cite{Franco2023b}, mesh-informed neural networks \cite{Franco2023a}, and latent dynamic networks to directly estimate a quantity of interest from latent dynamics \cite{Regazzoni2024}. 
A particularly active area of research focuses on the approximation of the operators that define the governing differential equations. Several well-established architectures fall into this category, such as the DeepONet \cite{Lu2021} and the Fourier neural operator \cite{Li2020b}, wavelet neural operator \cite{Tripura2023} and kernel operator learning \cite{Ziarelli2025}.

We propose the \textit{Dynamical Reduced Embedding} (DRE) method based on an autoencoder approximated by fully connected neural networks to define the dynamics of a low-dimensional ODE, which is subsequently learned from the data. Once the model is constructed, it provides an approximation of the solution to the high-dimensional ODE via the efficient solution of a low-dimensional ODE, followed by a reconstruction step to recover the solution in the high-dimensional space.
Unlike architectures that rely on temporal dependencies and require sequential training (such as backpropagation through time), the use of an autoencoder allows independent processing of each input instance. This enables a straightforward parallelization of the training, as there are no temporal causality constraints to enforce \cite{Goodfellow2017}.


\rev{Several recent approaches employ data-driven representations to approximate the solution manifold and its dynamics. On one hand, modal decomposition techniques such as dynamic mode decomposition (DMD) provide a spectral characterization of nonlinear dynamics by computing the eigendecomposition of an approximating linear operator and can be interpreted as finite-dimensional approximations of Koopman spectral analysis. While highly effective for extracting coherent dynamical structures from time-series data, these approaches rely fundamentally on linear representations and are not designed to approximate nonlinear solution manifolds or to provide convergence guarantees for reduced-order models. On the other hand, recent neural-network-based approaches construct nonlinear approximation manifolds or learn reduced dynamics directly from data. For instance, ANN-augmented projection-based reduced-order models aim at mitigating the Kolmogorov barrier by combining classical projection with nonlinear manifolds learned from data \cite{Barnett2023}, while operator-inference frameworks, although generally related to projection and regression-driven approaches, can be combined by learning low-dimensional dynamics through rollout-based training to improve robustness and stability properties \cite{Uy2023}. Our work is aligned with these contributions in its use of nonlinear representations and data-driven reduced dynamics, and is closely related to the framework proposed by Farenga et al.~\cite{Farenga2025}, where a high-dimensional ODE is linked to a low-dimensional one.


In a broad sense, the main originality of the present work lies in the joint analysis of the representation error and dynamical approximation error and the characterization of the architectural conditions ensuring injectivity of the encoder.
More precisely, we develop a mathematical framework for autoencoder-based model reduction that explicitly distinguishes representation and approximation errors, characterizes the topology of the solution manifold, and establishes convergence and stability results for the fully approximated reduced system. In particular, unlike approaches that treat the latent dynamics as a black-box model, our formulation introduces constrained reduced dynamics that enable rigorous analysis. 
The theoretical results are supported by representative numerical experiments inspired by applications.}

The paper is organized as follows.
In \Cref{sec:method}, we present the core of the reduced-order model strategy along with the main theoretical results. The remaining sections are dedicated to studying the methodology in detail. In particular, \Cref{sec:autoencoders} investigates the conditions under which a solution manifold can be compressed by an autoencoder with null representation error. Then \Cref{sec:ode_red} expands this foundation by formalizing the solution manifold of a parametric ODE as a low-dimensional topological manifold and deriving the minimal latent dimension required for lossless compression. Time discretization is then addressed in \Cref{sec:time_discretization}, where a study of stability and convergence in time is presented. The fully approximated problem based on the use of neural networks is presented in \Cref{sec:nn_approx}, along with an analysis of the global error. Finally, numerical experiments are reported in \Cref{sec:numerical_studies} to validate the proposed approach.

\section{Problem setup}\label{sec:method}	
We consider a parametric ODE of $N$ components, called \textit{\Node}, in the time interval $(0, T]$, with $T>0$, in the form
\begin{align}\label{eq:Node_mu}
	\begin{aligned}
		&\dot{u}_N = F_N(u_N; \mu), \\
		&u_N(0) = u_{N0}(\mu).
	\end{aligned}
\end{align} 
where $u_N : [0, T] \to \mathbb{R}^N$ is the \textit{N-solution}, $F_N: \mathbb{R}^N\to\mathbb{R}^N$ is a possibly nonlinear right-hand side, and $u_{N0}$ is the given initial condition. The parameters $\mu$ vary in a compact set with non-empty interior $\mathcal{P}\subset \mathbb{R}^{a-1}$, with $a-1\geq1$ being the number of parameters. 
Without loss of generality, we restrict our analysis to the autonomous case \cite{Lambert1999}. Unless it is important to remark it, we drop the dependence on $\mu$ in the notation and simply write the parametric ODE as
\begin{align}\label{eq:Node}
	\begin{aligned}
		&\dot{u}_N = F_N(u_N), \\
		&u_N(0) = u_{N0}.	
	\end{aligned}
\end{align} 

Since \eqref{eq:Node} can be computationally expensive due to $N$ being large, for instance if the system results from a prior space discretization of a PDE, we aim to solve a smaller ODE, called \textit{\node}, whose size is $n\ll N$
\begin{align}\label{eq:node_0}
\begin{aligned}
	&\dot{u}_n = F_n(u_n), \\
	&u_n(0) = \Psi'(u_{N0}),
\end{aligned}
\end{align}
where $\Psi'$ is a map that compresses the $N$-solution. It will be detailed in the following sections. 
Subsequently, from $u_n$, called the \textit{n-solution}, we reconstruct a high-dimensional solution, possibly approximated, through a map $\Psi$, i.e. $u_N = \Psi(u_n)$. The map $\Psi$ will be detailed in the following sections. We generally call \textit{reconstructed solution} the $N$-solution obtained from the $n$-solution. It can be equal to the original $N$-solution of \eqref{eq:Node} or it can differ due to approximations introduced, for example, by the time discretization or the use of neural networks.
It should be clear from the context whether it is the approximated solution or the exact one. 

Equation \eqref{eq:node_0} is solved discretizing the time axis with a constant timestep size, $\Delta t$, so $t_k = k\Delta t$, $k=0, \ldots, K$.
The numerical solution is obtained through a discrete-time integration scheme, which is characterized by its difference operator, $\mathscr{L}_n$ \cite{Lambert1999}. For a linear multistep method with $P$ steps (including also Forward Euler (FE) for $P=1$), the difference operator is
\begin{equation}\label{eq:lmm_diff_operator}
	\mathscr{L}_n(u_n^{k+1},u_n^k,\ldots,u_n^{k+1-P}, F_n, \Delta t) = -\sum_{p=0}^{P}\alpha_p u_n^{k+1-p} + \Delta t\sum_{p=0}^{P}\beta_p F_n(u_n^{k+1-p}), \quad k = P-1, P, \ldots
\end{equation}
where $\alpha_p$ and $\beta_p$ are coefficients of the time integration scheme and $u_n^k$ is the solution at time step $k$.
The maps $\Psi'$, $\Psi$, $F_n$ will be approximated by neural networks $N_{\Psi'}$, $N_{\Psi}$, $N_{F_n}$, respectively.
This leads to our reduced-order modeling strategy that consists of finding an approximation of \eqref{eq:Node} by solving the discrete problem $\mathfrak{P}(\uut)$ 
\begin{align}
	\mathfrak{P}(\uut) =
	\begin{cases}\label{eq:fully_approximated_0}
		\uut_N^{k+1} = N_\Psi(\uut_n^{k+1}), \\
		\mathscr{L}_n(\uut_n^{k+1},\uut_n^{k},\ldots,\uut_n^{k+1-P}, N_{F_n}, \Delta t) = 0, \quad k=P-1,P,\ldots \\
		\uut_n^0 = N_{\Psi'}(u_N^0).
	\end{cases}
\end{align}
where $\uut = (\uut_N, \uut_n)$ and $\uut_n$, $\uut_N$ represent, respectively, the discrete $n$-solution and the reconstructed one. 
We call the strategy of solving \eqref{eq:Node} through \eqref{eq:fully_approximated_0} DRE.

It is now important to study the properties of the proposed method. We are particularly interested in identifying the conditions under which $\uut_N^{k} \to u_N(t_{k})$, in understanding how to construct and train $N_{\Psi'}$, $N_{\Psi}$, and $N_{F_n}$ with the highest possible accuracy, and in analyzing the stability properties of $\mathfrak{P}$.

\subsection{\rev{Detailed outline}}
For the reader's convenience, we summarize \rev{in this section what we consider to be the most relevant topics of this work}. We do not focus here on precise notation or detailed assumptions and definitions of the framework; these will be introduced in the subsequent sections. Instead, \rev{given the large number of topics covered in this work,} we want to informally provide the \rev{main outline and key issues addressed} in the order of exposition in the text \rev{to facilitate the readability of this work}. 
\begin{itemize}
	\item  A linear--nonlinear autoencoder can reproduce with null representation error manifolds described by $f = f_V + f_{V^\perp}$, where $f_{V^\perp} \perp f_V$ (\textbf{\Cref{th:linear_nonlinear}}). This proposition justifies the use of a linear encoder. \\
    However, for a generic manifold, a nonlinear encoder is required and the reduction must be performed at the last layer of the encoder to ensure a null representation error (\textbf{\Cref{th:reduction}}).
	\item \textbf{\Cref{th:manifold}.}{
			\textit{$a$-manifold.}} Under the assumptions of well-posedness of the \Node, continuity of $F_N$ and $u_{N0}$ with respect to $\mu$, and local injectivity of the map $\mu \mapsto u_N$, the set $\{u_N(t;\mu)\}_{t \in [0,T],\ \mu \in \mathcal{P}}$ is an $a$-manifold. Moreover, the minimum latent dimension is $n \leq 2a+1$ (\textbf{\Cref{th:latent_dim}}).
	\item \textbf{\Cref{th:well_posedness_node}.} \textit{Well-posedness of the reduced problem.} Under proper continuity assumptions on the autoencoder, the reduced problem is well-posed.
	%
	%
	%
	\item We present two training strategies for $N_{F_n}$: one offers the possibility to increase the accuracy of the reduced-order model by setting $\Delta t < \Delta t_\text{train}$, the other allows for a trivial parallelization in time of the training. 
	\item \textbf{\Cref{th:global_convergence}.} \textit{Global convergence.} Assuming that enough training data are available, that the global minium of the loss function is reached, that convergent discretization schemes for both the \Node\ and the \node\ are used, and that proper neural network architectures are used, we have the following bound
		\begin{equation*}
			\max_{\mu} \|\m u_N - \m \uut_N \|_{\infty} \leq L_{\Psi}(\eps_{\Psi'}(\myw) + \eps_{F_n}(\Delta t_\text{train},\myw)T) e^{L_{F_n} T} + L_{\Psi}C_{1}\Delta t^P + \eps_{\Psi}(\myw),
		\end{equation*}
		so, for $|\myw|\to\infty$ and $\Delta t, \Delta t_{\text{train}} \to 0$, we have $\max_{\mu} \|\m u_N - \m \uut_N \|_{\infty} \to 0$. Here, $\myw$ is the total number of trainable weights, $\eps_{\Psi'}$, $\eps_\Psi$, $\eps_{F_n}$ are approximation errors, $L_\Psi$ and $L_{F_n}$ are Lipschitz constants, and $C_1$ a positive constant.
\end{itemize}

\section{Autoencoder-based dimensionality reduction}\label{sec:autoencoders}
In this section, we present the mathematical foundations for using autoencoders for nonlinear model reduction. Our focus is on understanding how the characteristics of the data manifold impose constraints on the design of the autoencoder architecture. To this end, we introduce the notions of representation error and approximation error, which allow us to distinguish between inaccuracies due to architectural limitations and those arising from the training of neural networks. Through concrete examples, we explore different encoder–decoder configurations and analyze the classes of manifolds they are capable of reconstructing exactly. These insights are essential to ensure that the autoencoder accurately captures the structure of the solution manifold, which is an important prerequisite for constructing a reliable reduced order method in the subsequent analysis.

\subsection{Neural networks}\label{sec:neural_network}  
The main goal of this section is to fix the notation. Taking into account a layer $\ell \in 1, \ldots, L$, we denote by $n_\ell$ its size, by \( y^{(\ell)} \in \mathbb{R}^{n_\ell} \) its output, by \( W^{(\ell)} \in \mathbb{R}^{n_\ell \times n_{\ell-1}} \) and \( b^{(\ell)} \in \mathbb{R}^{n_\ell} \) its trainable weights and biases, and by \( \sigma^{(\ell)} \) its nonlinear activation \rev{function, except for the last one, which is set to be the identity, $\sigma^{(L)} = \id$. The activation functions are possibly parameterized by $\myw_\sigma^{(\ell)}$. Throughout the paper, we assume the use of monotonic activation functions.} 
Each layer is constituted by an affine transformation followed by a nonlinear activation:
\begin{equation}\label{eq:nn_layer}
	y^{(\ell)} = \sigma^{(\ell)}\left(W^{(\ell)} y^{(\ell-1)} + b^{(\ell)}\right).
\end{equation}
To simplify notation, we define the layer-wise transformation \( z^{(\ell)} : \mathbb{R}^{n_{\ell-1}} \to \mathbb{R}^{n_\ell} \) by:
\begin{equation}
	z^{(\ell)}(y^{(\ell-1)}) := \sigma^{(\ell)}\left(W^{(\ell)} y^{(\ell-1)} + b^{(\ell)}\right),
\end{equation}
so that equation~\eqref{eq:nn_layer} becomes \( y^{(\ell)} = z^{(\ell)}(y^{(\ell-1)}) \).
The feed-forward fully connected neural network is a composition of \( L \) layers from $\mathbb{R}^{n_\text{in}}$ to $\mathbb{R}^{n_\text{out}}$, where $n_\text{in} = n_0$ and $n_\text{out}=n_L$:
\begin{equation}\label{eq:full_nn}
	y^{(L)} = z^{(L)} \circ z^{(L-1)} \circ \cdots \circ z^{(1)}(y^{(0)}),
\end{equation}
where \( y^{(0)} \in \mathbb{R}^{n_{\text{in}}} \) is the input to the network.
To simplify the notation, we define $\myw = \cup_\ell(W^{(\ell)}, b^{(\ell)}, \myw_\sigma^{(\ell)})$ as all trainable weights of a neural network, and we denote by $|\myw|$ the total number of trainable weights.

%
%
%

\subsection{Dimensionality reduction}\label{sec:reduction}
We consider a $m$-manifold, $\mathcal{M}$, of topological dimension $m$, in the ambient space $\mathbb{R}^N$. We now aim to relate the properties of $\mathcal{M}$ to the properties of the autoencoder. We call $n$ the latent size, $\Psi':\mathbb{R}^N\to\mathbb{R}^n$, $m\leq n<N$, the encoder, and $\Psi:\mathbb{R}^n\to\mathbb{R}^N$ the decoder. The autoencoder is defined as the composition $\Psi \circ \Psi': \mathbb{R}^N\to\mathbb{R}^N$, such that $\Psi \circ \Psi' = \id_{\mathcal{M}}$, where $\id_{\mathcal{M}}$ represents the identity map on $\mathcal{M}$. This requires that $\Psi\circ\Psi'$ is bijective and $\Psi'$ injective on $\mathcal{M}$, i.e. $\Psi':\mathcal{M}\to \mathbb{R}^n$ is injective. 

The maps $\Psi'$ and $\Psi$ will be approximated by neural networks (see \Cref{sec:nn_approx}), denoted by $N_{\Psi'}$ and $N_{\Psi}$, and still referred to as the encoder and decoder, respectively. The distinction between $\Psi'$, $\Psi$ and their neural network approximations $N_{\Psi'}$, $N_{\Psi}$ should be clear from the context.
\rev{
$N_{\Psi'}$ and $N_{\Psi}$ are defined as compositions of affine and nonlinear functions, as detailed in \Cref{sec:neural_network}. We argue that such compositions may compromise the properties that the encoder and the decoder are required to satisfy: the former should be injective, while the latter should properly reproduce the possible nonlinearities of $\mathcal{M}$.
Accordingly, we quantify the extent of the violation of these requirements by means of the following definitions:
\begin{definition}[Encoder inexactness index, $\Varepsilon_e$]
	Let $V = \{v \st v = N_{\Psi'}(x_{N,1}) = N_{\Psi'}(x_{N,2}), \ x_{N,1}, x_{N,2} \in \mathcal{M}, \ x_{N,1} \neq x_{N,2}\}$. Then, $\Varepsilon_e := |V|$, where $|V|$ is the cardinality of $V$.
\end{definition}
\begin{definition}[Decoder inexactness index, $\Varepsilon_d$]
	Let $W = \{w \st w = N_{\Psi}(x_n), \  x_n \in \mathbb{R}^n \}$. Then, $\Varepsilon_d := N-\dim(W)$.
\end{definition}
We focus on the cases where $\Varepsilon_e = 0$ and $\Varepsilon_d = 0$, corresponding to cases in which the encoder and decoder fully satisfy their intended requirements. We note that, according to the definitions above, $\Varepsilon_e$ could become unbounded, however, in this work we are interested in the cases where it is null.\\
For practical convenience, we introduce the following definition to combine the two aforementioned requirements:
}
\begin{definition}[Representation inexactness index, $\Varepsilon_1$]
	\rev{We call the representation inexactness index the following quantity: $\Varepsilon_1 = \Varepsilon_e + \Varepsilon_d$. This quantity represents an error which is intrinsically} introduced by an improper choice of $n_\ell$. This requirement is related to the topological properties of $\mathcal{M}$ and to the reduction/expansion at each step in the composition \eqref{eq:full_nn}.
    We are particularly interested in the latent size $n$ since, throughout the paper, we aim to achieve the maximum compression of the data $x_N$, i.e., to have $n$ as close as possible to $m$.
\end{definition}
\rev{The representation inexactness index is not the only source of error affecting the approximation of the identity, namely $N_{\Psi}\circ N_{\Psi'} \approx \Psi\circ\Psi' = \id_{\mathcal{M}}$. Accordingly, we introduce a second source of error. Let $\mathcal{A}$ be the search set for the trainable weights of the autoencoder, here considered as a single neural network. The layers in $\mathcal{A}$ have size $(n_\ell^{\mathcal{A}})_{\ell = 1}^L$. The training of the autoencoder can be cast as the following optimization problem
\begin{equation*}
    \myw_\text{opt} = \arg \min_{{\tiny W} \in \mathcal{A}} \loss(\myw),
\end{equation*}
where $\loss$ is the loss function.
We denote by $\overline{N}_{\Psi'}$ the network associated with $\myw_\text{opt}$ obtained from the optimization over $\overline{\mathcal{A}}$, which is constructed such that the layers have sizes $(\underbrace{n_\ell^{\mathcal{A}}, \ldots, n_\ell^{\mathcal{A}}}_{M\ \text{times}} )_{\ell=1}^L$ in the limit of $M \to \infty$.

}
\rev{
\begin{definition}[Approximation error, $\Varepsilon_2$]
    This source of error is specifically related to the use of neural networks and their associated training procedure. We define it as $\Varepsilon_2 \coloneq \|\overline{N}_{\Psi}\circ \overline{N}_{\Psi'} - N_{\Psi}\circ N_{\Psi'}\|_{L^\infty(\mathcal{M})}$.
    It arises because $\mathcal{A}$ may not be sufficiently large, for instance due to a finite number of layers and neurons, and because the optimization procedure may not reach the global minimum, so that $\myw_\text{train} \neq \myw_\text{opt}$, where $\myw_\text{train}$ denotes the weights obtained through training.
    
\end{definition}
}
\rev{Note that $\overline{N}_{\Psi}\circ \overline{N}_{\Psi'}$ may not represent the identity on the manifold due to the error associated to a possible $\Varepsilon_1 > 0$. Therefore, $\Varepsilon_2 = 0 \not\Rightarrow N_{\Psi}\circ N_{\Psi'} = \id_{\mathcal{M}}$.}
%
\rev{It follows that it is} relevant to study the effects of the functional composition in \eqref{eq:full_nn}, i.e. to study $\Varepsilon_1$, even before considering the approximation error, $\Varepsilon_2$. Note that this study is independent of the specific activation function used. 
In particular, we aim to have $\Varepsilon_1 = 0$, so that the only source of error in the approximation of $\id_\mathcal{M}$ is due to the practical and unavoidable limitations arising from the use of neural networks, represented by $\Varepsilon_2$.
In the following, after presenting two preliminary lemmas, we analyze three distinct scenarios with respect to types of manifolds and autoencoders. For clarity of exposition, we first make assumptions about the autoencoder that we subsequently use and then, from these assumptions, we then derive the properties of the underlying manifold for which, for the given autoencoder, we have $\Psi\circ\Psi' = \id_{\mathcal{M}}$.


\rev{\begin{proposition}[Image of a linear decoder]\label{sec:lem_dec}
Let $\Psi:\mathbb{R}^n\to\mathbb{R}^N$ be a linear decoder with $n<N$, and let $J_\Psi\in\mathbb{R}^{N\times n}$ denote its Jacobian. Then, $\Psi(\mathbb{R}^n)=\mathrm{col}(J_\Psi)$,
where $\mathrm{col}(J_\Psi)\subset\mathbb{R}^N$ denotes the column space of $J_\Psi$. In particular, for any subset $\mathcal{S}\subset\mathbb{R}^n$, the reconstructed set $\Psi(\mathcal{S})$ is contained in the linear subspace $\mathrm{col}(J_\Psi)$ of $\mathbb{R}^N$.
\end{proposition}}

\rev{Note that, to avoid any ambiguity in the notation, whenever a linear decoder is considered, its application to a vector $x_n$ is expressed in terms of its Jacobian matrix. Similarly for the encoder. In the following, let $V \subset \mathbb{R}^N$ with $\text{dim}(V)=n$ be a linear subspace.}

\rev{
\begin{proposition}[Injectivity of linear encoder]\label{sec:lem_enc}
	A linear encoder can only be injective on manifolds for which there exist both a linear subspace $V$ and a linear transformation $T:\mathcal{M}\to V$ such that $T$ is injective.
	Indeed, by the arbitrariness of the encoder, we can always define an encoder that is injective on a linear subspace of $\mathbb{R}^N$, and in particular $V$. It follows that, given the existence of a linear injection to $V$, the encoder is also injective to $\mathcal{M}$. 
\end{proposition}
}


\subsubsection{Linear encoder - Linear decoder}
We assume here to have a linear encoder and a linear decoder and we analyze, also with a numerical example, the type of manifold that is possible to reconstruct \rev{fulfilling the requirement $\Varepsilon_1 = 0$.} 

%
   

\rev{
\begin{proposition}[Exactness of autoencoders with a linear encoder and/or a linear decoder]\label{sec:lem_enc_dec}
An autoencoder with a linear encoder and/or a linear decoder can reproduce with $\Varepsilon_1 = 0$ only linear manifolds. 
\end{proposition}
}

Proposition \ref{sec:lem_enc_dec} is a direct consequence of Proposition \ref{sec:lem_dec} and Proposition \ref{sec:lem_enc}. In fact, the limiting condition is that the decoder is only capable of reproducing linear manifolds, $\mathcal{M}\in V$, while the encoder can be applied to the linear $\mathcal{M}$ preserving the injectivity.

Note that the previous propositions also hold for affine \rev{encoders}, but, for the sake of simplicity, we treat only the linear case.
\rev{We present now different cases} to provide constructive insight into the conditions under which an autoencoder can represent a manifold without loss, that is, when  $\Varepsilon_1$ vanishes. By analyzing specific configurations of the encoder and decoder, we clarify the limit of the expressivity of various architectural choices. 
See Fig.~\ref{fig:autoenc_case_0} for an example. 
In this section, we intentionally show simple examples such as Fig.~\ref{fig:autoenc_case_0}-\ref{fig:autoenc_case_2_dirty}, to enable a graphical interpretation of the topic. 
\begin{figure}[h!]
	\centering
	\includegraphics[width=0.4\linewidth]{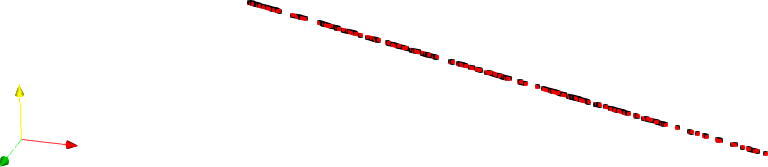}
	\caption{Example of a 1-manifold in $\mathbb{R}^3$ that can be reproduced exactly with linear transformations. Black point (overlapped by red points): original manifold. Red points: output of autoencoder. Here, $n=1$.}
	\label{fig:autoenc_case_0}
\end{figure}
%

\subsubsection{Linear encoder - Nonlinear decoder}
Here, we consider a linear encoder and a nonlinear decoder, with the same objective as before: having $\Varepsilon_1=0$. 
\rev{
Let $f : \mathbb{R}^n \to \mathcal{M}$ be the map that describes the manifold. Given a linear space $V \subset \mathcal{M}$, we denote by $f_V$ and $f_{V^\perp}$ the decomposition of $f$ into its components in $V$ and in its orthogonal complement, respectively.
}
\rev{
\begin{proposition}[Exactness of autoencoders with linear encoder and nonlinear decoder]\label{th:linear_nonlinear} 
    A linear--nonlinear autoencoder can reproduce a manifold $\mathcal{M}$ with $\Varepsilon_1 = 0$ if there exist a linear space $V$ and a map $f_V$ such that the manifold is described by $f = f_V + f_{V^\perp}$ with $f_V$ injective.
\end{proposition}

This is a direct consequence of \Cref{sec:lem_enc}: the linear encoder performs an injective projection onto a linear subspace, while the nonlinearity of the decoder is a necessary condition to reconstruct the lost orthogonal contribution, $f_{V^\perp}$.
}

In the following examples, we choose $f_V$ to be a linear map and thus consider manifolds of the form $x_N = Ax_n + g(x_n)$, with $A \in \mathbb{R}^{N\times n}$ full rank and $g(x_n) \perp Ax_n$. Despite its apparent simplicity, the latter expression can describe non-trivial geometries: see Fig.~\ref{fig:autoenc_case_1}
 
Examples are represented by the surface of a function or a parametric curve, Fig.~\ref{fig:autoenc_case_1}. In the latter, the line is a contracting coil described by $(x,y,z) = (r\cos(\theta),r\sin(\theta), 0.2\theta)$, where $r=\frac{2\pi}{2\pi+\theta}$, for which a possible representation is $A=(0,0,\rev{0.2})^T$ and $g=(r\cos(\theta),r\sin(\theta),0)^T$. For practical simplicity, in these numerical examples, we set $\Psi' = W^{(1)}$ (see \eqref{eq:nn_layer}) and $\Psi$ is approximated by neural networks with a finite and relatively small number of layers, trained using the common gradient-based method (as will also see in \Cref{sec:train}) resulting in nonzero, but small, values of $\Varepsilon_2$. Consequently, in Fig.~\ref{fig:autoenc_case_1}, the data points and the autoencoder output do not exactly coincide.
\begin{figure}[h!]
	\centering
	\subfloat[Graph of function. $m=2$, $n=2$.]{\includegraphics[width=0.6\linewidth]{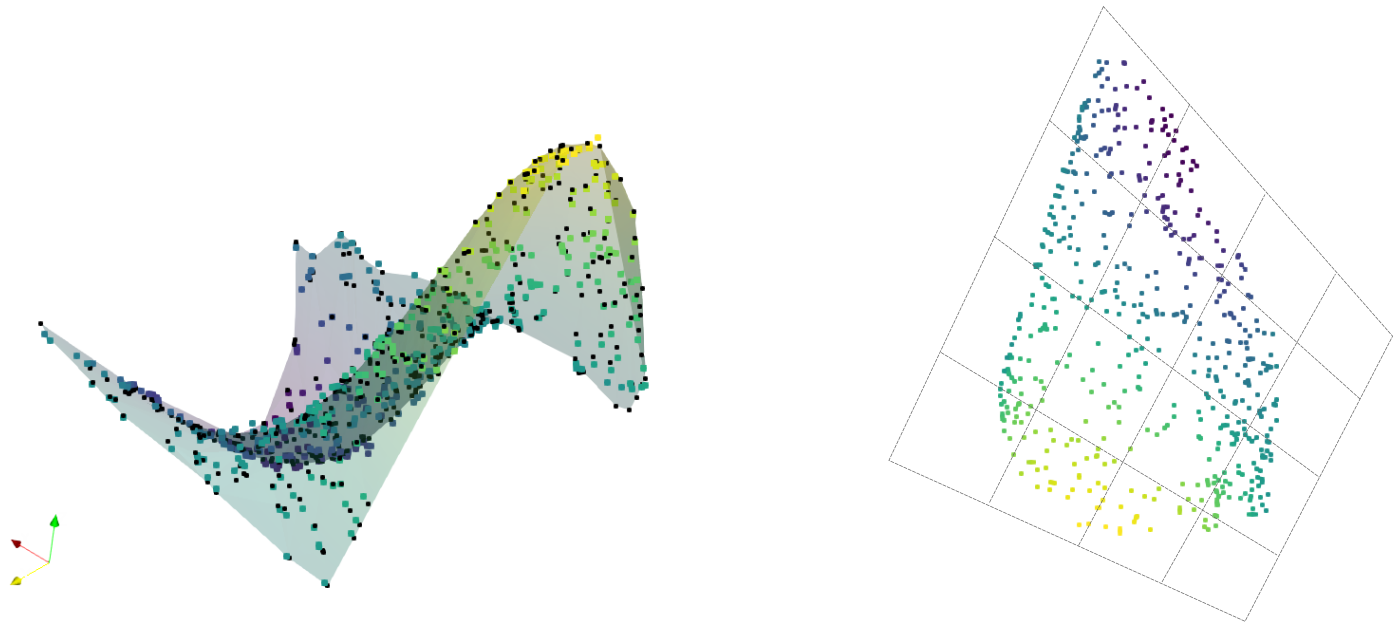}}\hfill
	\subfloat[Coil. $m=1$, $n=1$.]{\includegraphics[width=0.2\linewidth]{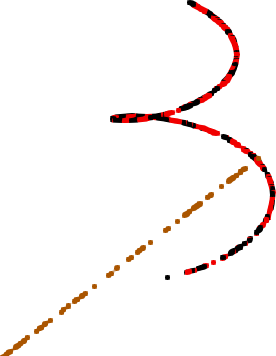}}
	\caption{\textbf{(a)} Manifold defined as graph of a function.  
    Left panel, black dots: encoder inputs; colored dots: decoder output. The shaded surface is reconstructed with a Delaunay triangulation. The color scale is related to the z-coordinate and it is used just to ease the graphical interpretation.
	Right panel: flat surface obtained from $\phi\circ\Psi'$.
	\textbf{(b)}
    Coil geometry. Black dots: encoder input. Red dots: decoder output. Ochre line: output of $\phi\circ\Psi'$.}
	\label{fig:autoenc_case_1}
\end{figure}

Incidentally, by constructing the decoder as $\Psi =  \psi\circ\phi$, $\phi : \mathbb{R}^n \to \mathbb{R}^N$ a linear function, $\psi:\mathbb{R}^N\to\mathbb{R}^N$ a nonlinear function, we can flatten the curved manifold, obtaining $\mathcal{M}_{\text{flat}} = \phi(\Psi'(\mathcal{M}))$, Fig.~\ref{fig:autoenc_case_1}.

\subsubsection{Nonlinear encoder - Nonlinear decoder}\label{sec:linear-nonlinear}
We consider here a nonlinear encoder combined with a nonlinear decoder. The nonlinearities in the encoder allow us to represent more complex manifolds. As already mentioned, in our approach, the encoder will be implemented as a neural network that combines activation functions with affine transformations. 

By the universal approximation theorem, see, e.g. \cite{Li2023, Guehring2021, Lu2021a, Hanin2017}, we see that, under appropriate continuity hypotheses, fully-connected neural networks are capable of approximating functions from $\mathbb{R}^N$ to $\mathbb{R}^n$, such as $\Psi'$.

For the reader's convenience, we summarize here the main result from \cite{Hanin2017} (Theorem 1), which serves both as a proof of the universal approximation property and as an upper bound on the width size:
\begin{theorem}[Upper bound for width \cite{Hanin2017}]\label{th:uap}
	\textit{Let $w_{\min}(n_{\text{in}}, n_{\text{out}})$
	be the minimal value of the width $w$ such that for every continuous function 
	$f : [0, 1]^{n_{\text{in}}} \to \mathbb{R}^{n_{\text{out}}}$
	and every $\varepsilon > 0$ there is a ReLU neural network $N_f$ with input dimension $n_{\text{in}}$, hidden layer widths at most $w$, and output dimension $n_{\text{out}}$ that $\varepsilon$-approximates $f$:
	\begin{equation*}
		\sup_{x \in [0,1]^{n_{\text{in}}}} \| f(x) - N_f(x) \| \leq \varepsilon.
	\end{equation*}
	Then, for every $n_{\text{in}}, n_{\text{out}} \geq 1$, we have $n_\text{in}+1\leq w_{\min}(n_\text{in},n_\text{out})\leq n_{\text{in}} + n_{\text{out}}$, proving that
	$w_{\min}(n_{\text{in}}, n_{\text{out}}) \leq n_{\text{in}} + n_{\text{out}}$.}
\end{theorem}
This theorem leads to the following proposition about fully nonlinear autoencoders.
\begin{proposition}
	A nonlinear--nonlinear autoencoder can reproduce with $\Varepsilon_1 = 0$ generic manifolds\textbf{.}
\end{proposition}
However, dimensionality reduction is performed layer by layer via a sequence of affine maps, which leads to the following theorem:
\begin{theorem}[Sufficient condition for $\Varepsilon_1=0$]\label{th:reduction} \textit{We assume $\mathcal{M}$ to be a compact $m$-manifold, \rev{and $n$ and $L$ to be sufficiently large, with $m \leq n <N $}. A sufficient condition to ensure \rev{the existence of an injective encoder} is that \rev{the width of the encoder layers} satisfies \rev{$n_\text{in} \leq n_\ell \leq n_\text{in} + n_\text{out}$}, for $\ell = 1,\ldots, L-1$.}  
\end{theorem} 
\begin{proof}
	To ensure the \rev{existence of an injective} encoder approximated by a neural network, \rev{we need to ensure the existence of injective layers}, condition which \rev{can only be satisfied in general} when $n_\ell \geq n_{\ell-1}$ for all hidden layers, so $n_\ell \geq n_\text{in}$ for $\ell = 1,\ldots,L-1$. By \Cref{th:uap}, we have that $n_\ell = n_\text{in}+n_\text{out}$ is sufficient to satisfy $\Varepsilon_1 = 0$. 
\end{proof}
\rev{Additionally, under the above conditions, in particular $n_\text{in} \leq n_\ell \leq n_\text{in} + n_\text{out}$, for $\ell = 1,\ldots, L-1$, and assuming ideal training procedure, we have that $\lim_{|\myw|\to\infty}\|\Psi' - N_{\Psi'}\| = 0$. This result also provides practical insights on the construction of the encoder.} \\
Therefore, the encoder can be written as: $\Psi' = \phi'\circ\psi'$, where $\phi' : \mathbb{R}^\rev{M}\to\mathbb{R}^n$ is a linear function, and $\psi':\mathbb{R}^N\to\mathbb{R}^M$, $M \geq N$, a nonlinear function.
As a result, the role of $\psi'$ is that of reshaping the manifold to ensure that the linear reduction in the last layer is injective. 
Similar conclusions can be drawn for the decoder: \rev{due to representation error and accuracy considerations, the expansion should not be performed solely in the last layer, as this approach would limit the network to representing an affine space to $\mathbb{R}^{n_{L-1}}$.}\\
\begin{remark}
	The message of \Cref{th:reduction} is that, without knowledge of the specific problem at hand, it is generally inappropriate to design a fully connected encoder with layers whose sizes reduce progressively (similarly for the decoder). Indeed, a reduction in dimension at an intermediate layer may compromise the injectivity of the encoder. 
\end{remark}
In the following examples and in the numerical cases in \Cref{sec:numerical_studies}, we use either PReLU or ELU activation functions. Empirically, we notice that $M=N$ is sufficient to have $\Varepsilon_1 = 0$ for the considered cases. \\
Fig.~\ref{fig:autoenc_case_2_dirty} shows a perturbed coil with an additional oscillation along its length. It is an example of a manifold for which a nonlinear encoder is necessary. In the same figure, the output of $\psi'$ is represented in green.  We see that the coil has been unwound, and can now be projected injectively onto a straight line, since in this case $n=1$. 
\begin{figure}[h!]
	\centering
	\includegraphics[width=0.5\linewidth]{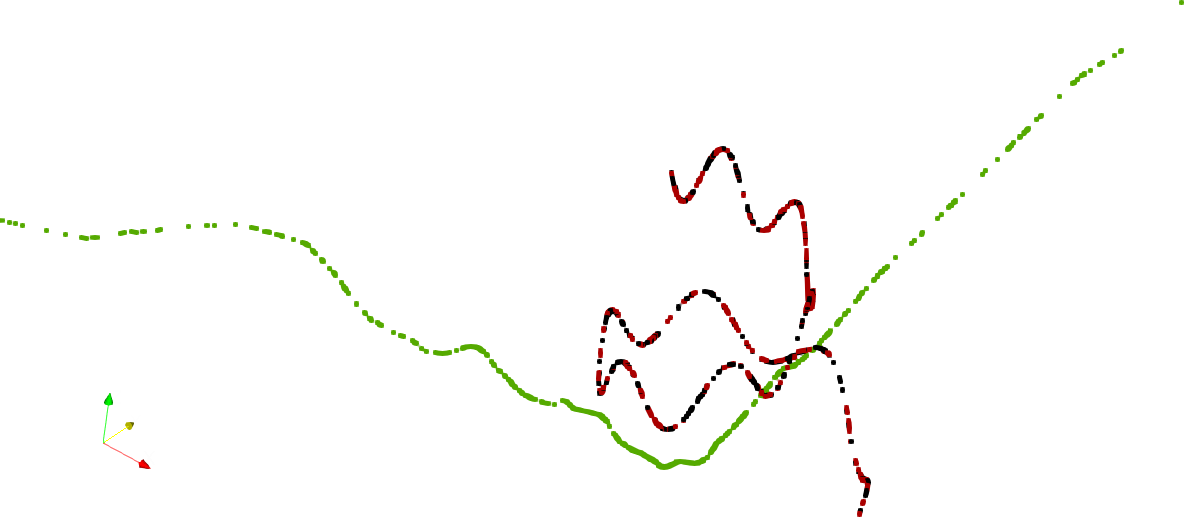}
	\caption{Noisy coil. A fully nonlinear autoencoder in necessary to reproduce $\id_{\mathcal{M}}$ with $\Varepsilon_1 = 0$. Black point: original manifold, red points: its reconstruction with $\Psi\circ\Psi'$, green points: output of $\psi'$. Here, $m=1$ and $n=1$.}
	\label{fig:autoenc_case_2_dirty}
\end{figure}
%

\begin{remark}Dimensionality reduction and similarity with POD-DL-ROM \cite{Fresca2022}:
	the reduction procedure presented in \cite{Fresca2022} relies on a first linear reduction based on the singular value decomposition of the snapshot matrix, followed by a non linear reduction made with an encoder. Similarly for the reconstruction step. The overall compression map can be interpreted as a single $N_{\Psi'}$, where $\sigma^{(1)} = \id$, the identity map, $b^{(1)} = 0$ and $W^{(1)}$, a rectangular matrix, computed from the singular value decomposition. For practical simplicity, in our method we just use neural networks in the conventional sense where $W^{(1)}$ (as all the other weights) is computed by an optimization procedure as it will be explained in \Cref{sec:train}.
\end{remark}

\section{ODE dimensionality reduction}\label{sec:ode_red}

In the following, our aim is to characterize the set of solutions of the \Node, and the possibility of describing it with a limited number of coordinates. 
We define $U$ as the subset of the solution set in a neighborhood of given $\mu^*$ and $t^*$: $U = \{u_N(t;\mu)\}_{B_r((t^*,\mu^*))}$, where $B_r(\cdot)$ is the closed ball of radius $r$.
\begin{theorem}[$a$-manifold]\label{th:manifold}
\textit{We assume that $F_N$ and $u_{N0}$ depend continuously on $\mu$ and \eqref{eq:Node_mu} is well-posed for each $\mu \in \mathcal{P}$, and the map $B_r((t^*,\mu^*)) \ni (t, \mu) \to u_N\rev{(t;\mu)}$ is injective $\forall (t^*, \mu^*) \in [0,T]\times\mathcal{P}$, with $r$ small enough. Then, the set $\{u_N(t;\mu)\}_{t\in[0,T], \mu \in \mathcal{P}}$ is \rev{an} $a$-manifold.}
\end{theorem}
\begin{proof}
    First, we note that the set $\{u_N(t;\mu)\}_{t\in[0,T], \mu \in \mathcal{P}}$ is a subset of a second countable Hausdorff space. 
    Then, by the \rev{well-posedness of \eqref{eq:Node_mu} and by} the continuity of $F_N$ and $u_{N0}$ in $\mu$, we see that $(t, \mu) \to u_N\rev{(t;\mu)}$ is continuous. By the continuity and by the compactness of $\mathcal{P}$, we have that $U$ is compact and $B_r((t,\mu))\to U$ has a continuous inverse, thus $B_r((t,\mu))\to U$ is a local homeomorphism. We conclude that $U$ is locally Euclidean, so $\{u_N(t;\mu)\}_{t\in[0,T], \mu \in \mathcal{P}}$ is \rev{an} $a$-manifold, that we denote by $\mathcal{M}$. 
\end{proof}
It follows that $\mathcal{M}$ can be described by a small number of coordinates, $n<N$, as presented in \Cref{sec:reduction}. With the following theorem, we aim to characterize the minimum value of $n$.
\begin{theorem}[Latent dimension]\label{th:latent_dim} 
\textit{Under the hypothesis of \Cref{th:manifold}, the minimum latent dimension is $n \leq 2a+1$.}
\end{theorem}
\begin{proof}
    According to Menger-N\"obeling embedding theorem (see, e.g., \cite{munkres2000topology}), any compact topological space of topological dimension $a$ admits a topological embedding in $\mathbb{R}^{2a+1}$. Hence, there exists a continuous injective map $\Psi' : \mathcal{M} \to \mathbb{R}^{2a+1}$.   
\end{proof}
This result provides an upper bound on the minimal latent dimension required to represent the manifold $\mathcal{M}$ without loss of information, and justifies the use of autoencoder architectures with a reduced dimension, $n$, not greater than $2a+1$ for the exact encoding of $\mathcal{M}$. \rev{We note that similar results, obtained under a different setting, can be found in \cite{Franco2022}.}

\begin{obs}{}
Throughout the paper, our aim is to describe the map $\mathcal{M} \to \mathbb{R}^n$ using a \textit{single} function $\Psi'$, and similarly for the inverse mapping $\mathbb{R}^n \to \mathcal{M}$. Although an atlas-based approach is possible, see \cite{Floryan2022}, the use of single global maps offers the advantage of requiring the approximation of only one encoder and one decoder. This comes with the negligible drawback of a potentially larger reduced dimension, at most equal to $2a + 1$.
\end{obs}
\begin{obs}{}\label{remark:augmentation}
Global injectivity can be obtained by augmenting the solution set with parameters and time, i.e. considering the set $\{(u_N(t;\mu), t, \mu)\}_{t \in [0,T],\, \mu \in \mathcal{P}}$; see Example 3.5 of \cite{Lee2000} and \cite{Franco2022}. \\
By the injectivity and continuity of $(t,\mu)\to(u_N(t;\mu),t,\mu)$, and by the compactness of $\mathcal{P}$, follows that there exist a global homeomorphism to $\mathbb{R}^a$. The minimum number of coordinates required to describe $\mathcal{M}$ is the same as that of $[0,T]\times\mathcal{P}$ (see \cite{Franco2022}\footnote{Note that in Theorem 1 of \cite{Franco2022}, Hilbert spaces are considered, although the result also holds for more general settings.}), which is equal to $a$. Thus, we can achieve the lowest latent size $n = a$.
\end{obs}

\begin{remark}
    In this paper, we assume a setting in which the number of parameters is much smaller than the size of \Node\ (which may arise from the spatial discretization of a PDE, so $N$ is expected to be very large), i.e., $a \ll N$. The upper bound for $n$, namely $n \leq 2a+1$, also satisfies the property of being much smaller than $N$, resulting in an effective reduction.
\end{remark}

\subsection{ODE of size $n$}
We now aim to utilize the information originating from \Node\ to better describe $\mathcal{M}$. To this end, we recall the definition of the $n$-solution $u_n = \Psi'(u_N)$, with $\Psi' \in \mathcal{C}^1$ as an encoder. Its dynamic is determined by $F_N$ and the autoencoder:
\begin{align}\label{eq:def_F_n}
\begin{aligned}
    &\dot{u}_n = \partialder{\Psi'(u_N)}{t} = \partialder{\Psi'}{u_N}\partialder{u_N}{t} = J_{\Psi'}\partialder{u_N}{t} = \\
    &= J_{\Psi'}(u_N)F_N(u_N) = J_{\Psi'}(\Psi(u_n))F_N(\Psi(u_n)).
\end{aligned}
\end{align}
We call $F_n(u_n) \coloneq J_{\Psi'}(\Psi(u_n))F_N(\Psi(u_n)) : \mathbb{R}^n \rightarrow \mathbb{R}^n$, and $u_{n0} \coloneq \Psi'(u_{N0}) $. Therefore, we have specified the right hand side, $F_n$, of the \node\ \eqref{eq:node_0}, that we report here for the reader's convenience
\begin{align*}
    &\dot{u}_n = F_n(u_n), \\
    &u_{n0} = \Psi'(u_{N0}).
\end{align*}
At this stage, we have linked the original \Node\ to the smaller \node.
This becomes useful in the context of multi-query scenarios where it is necessary to find $u_N(t)$ for many different values of $\mu$.
Instead of solving the \Node, where $N$ may be very large and it may derive from a discretization of a PDE, it may be convenient to learn a parameterization of $\mathcal{M}$, solve the \node, where $n$ can be much smaller than $N$, and then retrieve the $N$-solution as $u_N = \Psi(u_n)$. We will see in \Cref{sec:train} how to efficiently achieve this task from a practical standpoint. We now aim to better characterize $u_N$ in terms of $u_n$. However, before proceeding with the analysis, it is necessary to compute some useful preliminary quantities.

\subsection{Preliminary properties}\label{sec:ode_reduction}

In this section, we collect a set of auxiliary results that will be instrumental in the theoretical analysis of the proposed reduction method and its numerical discretization. These properties concern both geometric and functional aspects of the encoding and decoding maps, and provide the foundation for the error estimates and approximation results presented in the subsequent sections. For clarity and completeness, we state them independently before incorporating them into the main convergence analysis.

\begin{proposition}[Lipschitz constant of ${F_n}$]\label{th:L_{F_n}}
	We assume that \rev{ $\Psi$, $\Psi'$ and $F_N$} are Lipschitz continuous with Lipschitz constants \rev{$L_\Psi$, $L_{\Psi'}$ and $L_{F_N}$}, respectively. Furthermore, we assume that the encoder is differentiable and that its Jacobian, denoted by $J_{\Psi'}$, has a Lipschitz constant $L_{J_{\Psi'}}$ and is bounded with $M_{J_{\Psi'}} \coloneqq \sup_{u_N \in \mathcal{M}} \|J_{\Psi'}(u_N)\|$. Additionally, we assume that $F_N$ is bounded, with $M_{F_N} \coloneqq \sup_{u_N\in\mathcal{M}} \|F_N(u_N)\|$. From the definition of $F_n$, we compute an upper bound, $\overline{L}_{F_n,\mu}$, for its Lipschitz constant $L_{F_n}$ (see \Cref{sec:computations} for the detailed steps):
	\begin{align}\label{eq:L_F_n}
	\begin{aligned}
	    &\|F_n(u_{n1}) - F_n(u_{n2})\| = \|J_{\Psi'}(\Psi(u_{n1}))F_N(\Psi(u_{n1})) - J_{\Psi'}(\Psi(u_{n2}))F_N(\Psi(u_{n2}))\| \leq \\
	    & \leq \underbrace{(M_{J_{\Psi'}}L_{F_N} + M_{F_N}L_{J_{\Psi'}}) L_\Psi}_{= \overline{L}_{F_n,\mu}} \|u_{n1} - u_{n2}\|,
	\end{aligned}
	\end{align}
	%
	Note that $\overline{L}_{F_n,\mu}$ depends on $\mu$ since the quantities related to $F_N$ do. Whenever necessary, we take the maximum over the parameters: $\overline{L}_{F_n} = \max_{\mu} \overline{L}_{F_n,\mu}$, and we have $L_{F_n} \leq \overline{L}_{F_n}$. Lipschitz continuity ensures that the solution of the \node\ exists and is unique \cite{Butcher2016, Quarteroni2007}.
\end{proposition}
\rev{We anticipate here that $\Psi$, $\Psi'$, and $F_N$ will be approximated by neural networks. Similar conclusions to those of the above proposition can be drawn for neural networks. In this case, it is also possible to relate upper bounds of the Lipschitz constants to the architecture of the network, see \cite{Pintore2024}.
}
\begin{proposition}[Equivalence of autoencoders]\label{th:equivalence_autoenc}
	Let $g: \mathbb{R}^n \rightarrow \mathbb{R}^n$ be an invertible function. The composition $\Psi \circ g^{-1} \circ g \circ \Psi'$ describes the same autoencoder $\Psi \circ \Psi'$ but allows us to consider different \rev{$n$-solutions}. \rev{Indeed, let $u_n^g = g(u_n)$, $\Psi'^g = g \circ \Psi'$, and $\Psi^g = \Psi \circ g$, we have that $u_n \neq u_n^g$ while $\Psi \circ \Psi'$ and $\Psi^g \circ \Psi'^g$ represents the same autoencoder in the sense that they both coincide with $\id_\mathcal{M}$.}
\end{proposition}
\begin{proposition}[Lipschitz constants of scaled problem]
Let us consider the scaled mapping $g(u_n) = K_s u_n$ with $K_s \in \mathbb{R}^+$, which represents a linear scaling of the $n$-solution. We now proceed to compute the Lipschitz constants of the scaled encoder and decoder, defined respectively as $\tilde{\Psi}' = g \circ \Psi'$ and $\tilde{\Psi} = \Psi \circ g^{-1}$.
	\begin{align*}
		&\| \tilde{\Psi}(u_{n1}) - \tilde{\Psi}(u_{n2}) \| = \| \Psi(g^{-1}(u_{n1})) - \Psi(g^{-1}(u_{n2})) \| \leq L_\Psi \| g^{-1}(u_{n1}) - g^{-1}(u_{n2}) \| = 1/K_sL_\Psi \|u_{n1} - u_{n2}\|, \\
		&\| \tilde{\Psi}'(u_{N1}) - \tilde{\Psi}'(u_{N2}) \| = \| g(\Psi'(u_{N1})) - g(\Psi'(u_{N2})) \| = K_s \| \Psi'(u_{N1}) - \Psi'(u_{N2}) \| \leq K_sL_{\Psi'} \|u_{N1} - u_{N2}\|,
	\end{align*}
	so $L_{\tilde{\Psi}'} = K_sL_{\Psi'}$ and $L_{{\tilde{\Psi}}} = 1/K_sL_\Psi$.
	For the right-hand side $F_n$, we have that the Lipschitz constant remains invariant under the scaling, i.e., $L_{\Tilde{F}_n} = L_{F_n}$ (see \Cref{sec:computations} for the computation). In order to modify $L_{F_n}$, a nonlinear $g$ is required.
\end{proposition}
\begin{proposition}[Initial conditions for $u_n$]
	Let $g_{\mu}(u_n) = u_n - u_{n,0}$ be a $\mu$-dependent function. The use of $g_\mu$ allow us to modify the initial conditions. In particular, defining $\overline{u}_n \coloneq g_{\mu}(u_n) = u_n - u_{n,0}$, we obtain the following \node:
	%
	%
	%
	\begin{align*}
		&\dot{\overline{u}}_n = F_n(\overline{u}_n + u_{n,0}) \eqcolon \overline{F}_n(\overline{u}_n), \\
		&\overline{u}_n(0) = 0.
	\end{align*}
	Then, $u_N$ is the retrieved by re-adding the initial condition: $u_N = \Psi(\overline{u}_n + u_{n,0})$. It is therefore possible to set the initial condition to zero for any value of the parameters, assuming the knowledge of (or the possibility to approximate) $g_\mu$. Otherwise, the initial conditions are computed with the encoder: $u_{n,0} = \Psi'(u_{N,0})$.
\end{proposition}

\subsection{Stability after exact reduction}\label{sec:stability_exact}
We aim to establish the Lyapunov stability, i.e. an estimate for the discrepancy between the exact solution and that derived from a perturbed ODE. To emphasize the impacts of the reduction, we perturb both \Node\ and \node. Subsequently, we calculate the error between the exact solution and the reconstructed $u_N$. Additionally, establishing stability in the Lyapunov sense, paired with Lipschitz continuity, ensures the well-posedness of the \node\ \cite{Quarteroni2007}. Moreover, this study will be instrumental, in \Cref{sec:stability_exact_nn}, in studying the perturbations associated with the approximation errors of neural networks.\\
Let us consider a single trajectory. Given $\eps_0$, $\eps$, $\gamma_0$, $\gamma$ $>0$, we define the perturbations as follows.
\begin{itemize}
	\item \Node: $\Delta_0 \in \mathbb{R}^n$, and $\Delta(t): (0, T] \to\mathbb{R}^n$, with $\|\Delta_0\| < \eps_0$, and  $\|\Delta(t)\| < \eps, \forall t$; 
	\item \node: 
	$\delta_0 \in \mathbb{R}^n$, and $\delta(t): (0, T] \to\mathbb{R}^n$, with $\|\delta_0\| < \gamma_0$, and $\|\delta(t)\| < \gamma, \forall t$. 
\end{itemize} 
We call $z_N$ the solution of the perturbed \Node:
\begin{align*}
    &\dot{z}_N = F_N(z_N) + \Delta, \\
    &z_N(0) = u_{N0} + \Delta_0.
\end{align*}
%
%
%
In addition, we perturb \node\ and we call $w_n$ the solution of the perturbed \node, obtained from the reduction of the perturbed \Node:
\begin{align*}
    &\dot{w}_n = F_n(w_n) + J_{\Psi'}\Delta + \delta, \\
    &w_n(0) = \Psi'(u_{N0} + \Delta_0) + \delta_0.
\end{align*}
We now compute the discrepancy between the reconstructed perturbed solution $w_N = \Psi(w_n)$ and $u_N$. We obtain the following result, see \Cref{sec:computations} for the computations:
\begin{equation}\label{eq:lyapunov_stab_node}
	\|w_n(t) - u_n(t)\| \leq \Big( L_{\Psi'}\|\Delta_0\| + \|\delta_0\| + \big(M_{\Psi'}\|\Delta\| + \|\delta\|\big)t \Big) e^{L_{F_n,\mu}t},
\end{equation}
for the discrepancy of $n$-solutions, and:
\begin{equation}\label{eq:lyapunov_stab_Node}
	\|w_N(t) - u_N(t)\| \leq L_\Psi\Big( L_{\Psi'}\|\Delta_0\| + \|\delta_0\| + \big(M_{\Psi'}\|\Delta\| + \|\delta\|\big)t \Big) e^{L_{F_n,\mu}t},
\end{equation}
for the discrepancy between the reconstructed and the original solution. Note that these bounds depend on $\mu$ because $L_{F_n,\mu}$ does. For a global bound, we can replace $L_{F_n,\mu}$ with $L_{F_n}$.

\begin{theorem}[Well-posedness of reduced problem]\label{th:well_posedness_node}
    \textit{Under the assumption of \Cref{th:L_{F_n}}, the reduced problem
	\begin{align*}
		\begin{cases}
			u_N = \Psi(u_n), \\
			\dot u_n = F_n(u_n), \\
			u_n(0) = \Psi'(u_{N0}).
		\end{cases}
	\end{align*}
	is well-posed.}
\end{theorem}
\begin{proof}
	The proof follows from \Cref{th:L_{F_n}} and Lyapunov stability \eqref{eq:lyapunov_stab_Node}, which are sufficient requirements for well-posedness \cite{Quarteroni2007}.
\end{proof}

\paragraph{Scaling reduced problem for improved stability.}

We can exploit the scaling to reduce the effects of the perturbation of the \node\ propagated to the reconstructed $u_N$. Rescaling the autoencoder with a factor $K_s$: $\Tilde{\Psi}' \coloneqq K_s \Psi'$, and $\Tilde{\Psi}(u_n) \coloneqq  \Psi(1/K_s u_n)$, \rev{and as a consequence $\Tilde{F}_n(u_n) = J_{\Tilde{\Psi}'}F_N(\Tilde{\Psi}(u_n))$}, we easily obtain the following relations: $J_{\Tilde{\Psi}'} = K_s J_{\Psi'}$, $M_{\Tilde{\Psi}'} = K_s M_{\Psi'}$, $L_{\Tilde{\Psi}'} = K_s L_{\Psi'}$, and $L_{\Tilde{\Psi}} = 1/K_s L_\Psi$. It follows that
\begin{equation}\label{eq:lyapunov_scaled}
	\|w_N(t) - u_N(t)\| \leq L_\Psi\Big( L_{\Psi'}\|\Delta_0\| + 1/K_s\|\delta_0\| + \big(M_{\Psi'}\|\Delta\| + 1/K_s\|\delta\|\big)t \Big) e^{L_{\tilde{F}_{n,\mu}}t}.
\end{equation}
\rev{We first observe that the bound for the perturbation of the \node\ is influenced by the decoder, as reflected in the terms $L_\Psi / K_s \|\delta_0\|$ and $L_\Psi / K_s \|\delta\|$. In general $L_{\tilde{F}_{n,\mu}} \neq L_{F_{n,\mu}}$ although it is possible to check that they have the same upper bound according to the computations in \Cref{sec:computations}; the exact values of the Lipschitz constants depend on the specific problem at hand and on the nonlinearities of $F_N$. Since $L_\Psi \geq 1$, the bound is increased by the presence of the decoder, and we can expect that the decoder amplifies the perturbation of the \node. However, due to the presence of the scaling factor $K_s$, we conclude the following:}
\begin{proposition}\label{th:scaling}
	\rev{
		Supposing that the maximum norm of the perturbation of the \node\ is independent of the scaling $K_s$, and that $F_n$ and $\Tilde{F}_n$ share the same Lipschitz constant, the bound for stability decreases as $K_s$ increases, see \eqref{eq:lyapunov_scaled}.
	}
\end{proposition}
\begin{example}{SIR.}
We present now a numerical example \rev{where the scaling is practically effective. We consider here} the SIR (susceptible, infected, recovered) model of infectious diseases:
\begin{equation}\label{eq:sir}
	\begin{cases*}
		\dot S = -\frac{\beta}{N} IS,\\
		\dot I = \frac{\beta}{N} IS-\gamma I,\\
		\dot R = \gamma I,\\
	\end{cases*}
\end{equation}
where $S$ represents the susceptible population, $I$ infected people and $R$ recovered people, $N$ is the total number of people. \rev{The initial conditions are: $S(0) = 90$, $I(0) = 10$, $R(0)=0$}. The recovery rate, $\gamma$, and the infection rate, $\beta$, are considered uncertain and randomly sampled 200 times in the range $[0.5, 2.5]$ to create $\mathcal{M}$. The autoencoder is made of fully connected layers and is trained until the mean squared error loss reaches a low values of about $10^{-5}$. The latter is computed on the test dataset, made of 50 samples. $F_n$ is computed from \eqref{eq:def_F_n}, after the training of the autoencoder. The perturbations are set to zero except for $\delta$, which is a random value in time with a uniform distribution $\mathcal{U}(-0.2, 2)$. 
Fig.~\ref{fig:lyapunov_scaled} shows the \rev{relative $\ell_2$-error in time, $\|w_N-u_N\|_2 / \|u_{N0}\|_2$}, for two solutions of the SIR, $w_N$ and $u_N$, obtained with 2 samples of $\beta$ in the test dataset. The solutions are obtained numerically with the Forward Euler scheme with a time step size small enough to make the error introduced by the numerical integration negligible compared to $\delta$. \rev{As predicted by \Cref{th:scaling},} we see that a scaling factor of $K_s =2$ reduces the upper bounds and also the error in most of the computed time steps (yellow line). Note that this occurs because the perturbation is independent of the solution itself. Therefore, rescaling alters the ratio between the magnitude of the external perturbation and the magnitude of $u_n$. However, this condition may not hold in practical scenarios when, for instance, the perturbation arises from floating-point truncation errors, which are dependent on the magnitude of the solution. \rev{Additionally, we remark that the error bound is derived using Gr\"onwall's lemma (see \Cref{sec:computations}), which is known to be non-sharp. As a consequence, the numerical value of the bound itself may not provide a reliable reference for practical use while it remains useful for highlighting the role of the different quantities in the stability analysis.}
\begin{figure*}[h]
	\hspace*{5cm}
	\begin{minipage}{0.3\textwidth}
		\centering
		\includegraphics[width=1\textwidth]{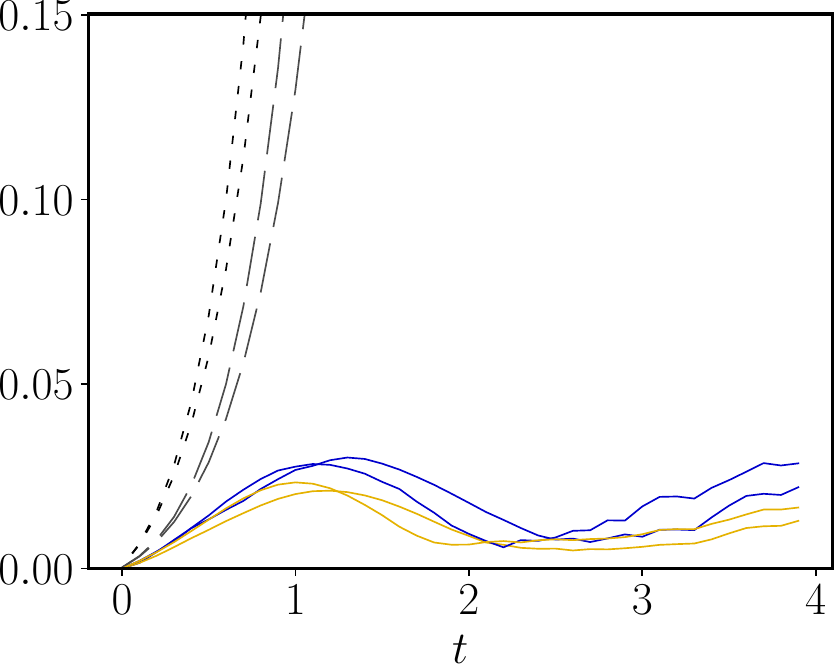}
	\end{minipage}
	\begin{minipage}{0.25\textwidth}
		\vspace{-2cm}
		\hspace{0cm}
		\includegraphics[width=1\textwidth]{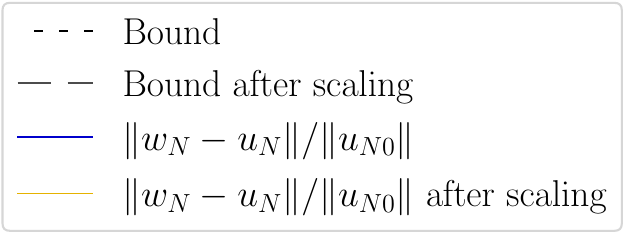}
	\end{minipage}
	\caption{Lyapunov stability. Bounds	 (black lines) for the given $\mu$ and errors \rev{$\|w_N-u_N\| / \|u_{N0}\|$} (blue and yellow lines) in time.}
	\label{fig:lyapunov_scaled}
\end{figure*}
\end{example}

\section{Analysis of the time discretization error and stability}\label{sec:time_discretization}
This section focuses on the time discretization, assuming to have access to a perfect autoencoder, i.e., one that is not subject to approximation errors introduced by neural networks. This latter source of approximation error will be addressed separately in \Cref{sec:nn_approx}.
The most important topic is the convergence of the numerical scheme, which ensures the reliability of the computed solutions. To complement this analysis, we also investigate the consistency and stability of the numerical solution.
For simplicity, we adopt the FE method as the reference scheme throughout most of the discussion.

\subsection{Convergence in time for exact autoencoder}
We consider a uniform time discretization, $t \in \{k\Delta t\}_{k=0}^K$, where $K$ is the total number of timesteps, including the initial condition. We call $\ut_n$ the discrete-in-time solution of the \node, and $\ut_N = \Psi(\ut_n)$ the corresponding reconstructed solution. 
We are interested in the error on the solution of \Node\ reconstructed from the reduced one, i.e. $\|u_N-\ut_N\|$. 
The discrete-in-time problem, $\mathrm{P}$, is: 
\begin{equation*}
	\mathrm{P}(\ut) = \begin{cases}
		\ut_N^{k+1} = \Psi(\ut_n^{k+1}), \\
		\mathscr{L}_n(\ut_n^{k+1},\ut_n^{k},\ldots,\ut_n^{k+1-P}, F_n, \Delta t) = 0, \quad k=P-1,P,\ldots\\
		\ut_n^0 = \Psi(\ut_{N0}),
	\end{cases}    
\end{equation*}
where $\ut = (\ut_N, \ut_n)$, and $\mathscr{L}_n$ is the difference operator \eqref{eq:lmm_diff_operator}.
%
%
Our goal is to ensure that the error between the N-solutions vanishes as the time step tends to zero, i.e., $\|u_N-\ut_N\| \to 0$, as $\Delta t \to 0$. By the well-posedness of the \node\ (\Cref{sec:stability_exact}) and assuming to use a convergent time integration scheme to solve the \node, the convergence of the \Node\ is simply ensured by the continuity of the decoder: $\|u_N-\ut_N\| \leq L_\Psi\|u_n-\ut_n\| \to 0$ as $\Delta t \to 0$ which holds for all the trajectories in $\mathcal{M}$.

\subsection{Consistency}
Although convergence has already been established, it remains interesting to investigate how the reduction process influences consistency to better understand the role of $\Psi'$ and $\Psi$. 
A method is said to be consistent if, upon substituting the exact solution of the \Node\ into the discrete problem, denoted by $\mathrm{P}$, the residual decays sufficiently rapidly.
Precisely, the method is consistent if $\lim_{\Delta t \to 0}\max_k|\tau_N^k| = 0$ \cite{Lambert1999, Quarteroni2007}, where $\tau_N^k$ is the local truncation error at timestep $k$ for the \Node.
In our setting, the consistency requirement can be reformulated in terms of the \node: substituting $u = (u_N, u_n)$ into $\mathrm{P}$ reads
\begin{subnumcases}{\mathrm{P}(u) = \label{eq:null_1}}
	u_N^{k+1} = \Psi(u_n^{k+1}), \label{eq:problem_consistency_2} \\
	\mathscr{L}_n(u_n^{k+1},u_n^{k},\ldots,u_n^{k+1-P}, F_n, \Delta t)  - \Delta t \, \tau_n^k = 0, \quad k=P-1,P,\ldots \label{eq:problem_consistency_3} \\
	u_n^0 = \Psi'(u_{N0}). \label{eq:problem_consistency_4}
\end{subnumcases}
Here, to simplify the notation, we added a superscript to denote the timestep: $u_N^k = u_N(t_k)$.
Equations \eqref{eq:problem_consistency_2} and \eqref{eq:problem_consistency_4} are satisfied by construction.
In \eqref{eq:problem_consistency_3}, $\Delta t \tau_n^k$ is the residual written in terms of the local truncation error of the \node, $\tau_n$. 
Equation \eqref{eq:null_1} indicates that the consistency of the reconstructed solution depends on the consistency of the \node. Consequently, using a consistent method for solving the \node\ ensures the consistency of the \Node.

To have a better insight into how the reduction process influences consistency, we now examine a one-step method, such as the FE scheme. All such methods can be written with the following difference operator\\ $\mathscr{L}_n(\ut_n^{k+1},\ut_n^{k},\ldots,\ut_n^{k+1-P}, F_n, \Delta t) = \ut_n^{k+1} - \Delta t \Phi(\ut_n^k\mg{, } F_n)$ where $\Phi(\ut_n^k, F_n)$ is called \textit{incremental function}. We proceed by substituting $u = (u_N, u_n)$ into $\mathrm{P}$ with $\mathscr{L}_n$ as previously defined:
\begin{numcases}{\mathrm{P}(u) = }
	u_N^{k+1} = \Psi(u_n^{k+1}), \nonumber  \\
	u_n^{k+1} - u_n^k - \Delta t \, \Phi(u_n^k) - \Delta t \, \tau_n^k = 0,  \label{eq:one_step_lte}\\
	u_n^0 = \Psi'(u_{N0}). \nonumber
\end{numcases}
Then, we compute a Taylor expansion centered in $u_n^k$ to express the variation of $u_n$ in a timestep: $u_n^{k+1} = u_n^k + \dot u_n \Delta t + \ddot u_n \Delta t^2/2 + \mathcal{O}(\Delta t^3)$, where $\dot u_n = F_n$ and, assuming an encoder smooth enough, $\ddot u_{n,i} = F_{N,m}\frac{\partial^2\Psi'_i(x)}{\partial x_m\partial x_j}F_{N,j} + J_{\Psi',ij}J_{F_N,jk}J_{\Psi,km}J_{\Psi',mz}F_{N,z}$, where $F_{N,m} = F_N(\Psi(u_n))_m$ and similarly for $J_{\Psi',ij}$, $J_{\Psi, ij}$, $J_{F_N,ij}$. Replacing it into \eqref{eq:one_step_lte}:
\begin{equation}\label{eq:lte_taylor}
    u_n^k + \dot u_n \Delta t + \ddot u_n\Delta t^2 / 2 + \mathcal{O}(\Delta t^3) = u_n^k + \Delta t \Phi(u_n^k) + \Delta t \tau_n^k,
\end{equation}
%
from which we obtain the local truncation error for the one-step method: 
\begin{align}\label{eq:lte}
    \tau_n^k = \dot u_n - \Phi(u_n^k; \Delta t) +  \ddot u_n\Delta t / 2 + \mathcal{O}(\Delta t^2).
\end{align}
By the fact the $\lim_{\Delta t \to 0} \Phi(u_n; \Delta t) = F_n(u_n)$, we have that $\lim_{\Delta t \to 0}\max_k \tau_n^k = 0$, so consistency is confirmed. We write explicitly $\tau_n^k$ for FE: 
\begin{equation*}
	\tau_n^k = \frac{1}{2}F_{N,m}\frac{\partial^2\Psi'_i(x)}{\partial x_m\partial x_j}F_{N,j} \Delta t + \frac{1}{2}J_{\Psi',ij}J_{F_N,jk}J_{\Psi,km}J_{\Psi',mz}F_{N,z}\Delta t + \mathcal{O}(\Delta t^2),
\end{equation*}
where we note that the encoder affects $\tau_n$ with both its second and first derivatives, while the decoder appears in terms of its Jacobian, as argument in $F_N$.

\subsection{Zero stability}
We aim to understand how the reduction, and therefore the autoencoder, impacts the zero stability of the time-marching method, focusing on one-step schemes, such as the FE scheme. Similarly to \Cref{sec:stability_exact}, we consider a perturbed problem with perturbed solution, $\wt_n^k$, at timestep $k$ and bounded perturbations $\Delta$, $\Delta_0$, $\delta$, $\delta_0$:
\begin{align*}
    &\wt_n^{k+1} = \wt_n^k + \Delta t \Phi(\wt_n^k) + \Delta t J_{\Psi'}\Delta + \Delta t \delta, \\
    &\wt_n^0 = \Psi'(\ut_{N0} + \Delta_0) + \delta_0.
\end{align*}
For simplicity, we sum the perturbations contribution, so we group them in a unique \rev{$\overline{\delta}_0 = \Psi'(u_{N0} + \Delta_0) + \delta_0 - \Psi'(u_{N0})$}, and $\overline{\delta} = J_{\Psi'}\Delta + \delta$, bounded by $\|\overline{\delta}_0\| < \eps_0$, with $\eps_0>0$, and $\|\overline{\delta}\| < \eps$, with $\eps>0$. Thus, from the previous equation we obtain 
\begin{align*}
    &\wt_n^{k+1} = \wt_n^k + \Delta t \Phi(\wt_n^k) + \Delta t \overline{\delta}, \\
    &\wt_n^0 = \Psi'(\ut_{N0} ) + \overline{\delta}_0.
\end{align*}
Writing the solution at timestep $k+1$ in terms of the initial conditions, we have: $\ut_n^{k+1} = \ut_n^0 + \Delta t\sum_{i=0}^k\Phi(\ut_n^i)$ and
$\wt_n^{k+1} = \wt_n^0 + \Delta t\sum_{i=0}^k\Phi(\wt_n^i) + \Delta t \sum_{i=0}^k\overline{\delta}^i$ for the perturbed one. Using the triangular inequality and the discrete Gr\"onwall's lemma (see e.g., \cite{Quarteroni2007}), the error for the \node\ can be bounded as:
\begin{align*}
    \|\wt_n^{k+1} - \ut_n^{k+1}\| \leq \|\overline{\delta}_0\| + \Delta t \sum_{i=0}^k\|\Phi(\wt_n^i)-\Phi(\ut_n^i)\| + \Delta t \sum_{i=0}^k\|\overline{\delta}^i\| \leq  \left(\eps_0 + K \Delta t \eps \right)e^{(k+1)\Delta t L_\Phi},
\end{align*}
where $L_\Phi$ is the Lipshitz constant of $\Phi$. Using the Lipschitz continuity of $\Psi$, we have:
\begin{equation*}
    \|\wt_N^{k+1} - \ut_N^{k+1}\|\leq L_\Psi\left(\eps_0 + (k+1) \Delta t \eps \right)e^{(k+1)\Delta t L_\Phi}.
\end{equation*}
The latter inequality says that the perturbation on the reconstructed solution is bounded. Moreover, the decoder appears in both $L_{\Psi}$ and in $L_{\Phi}$, while the encoder affects only $L_{\Phi}$. Indeed, adopting for example FE, we have $L_\Phi = L_{F_{n,\mu}}$, where $L_{F_{n,\mu}}$ depends on both the encoder and decoder, see equation \eqref{eq:L_F_n}.

\subsection{\rev{Linear} stability}\label{sec:a_stability}

We now turn to the study of whether the \rev{regions of absolute stability are} preserved after reduction. We consider the following linear \Node, called \textit{model problem}
\begin{align}\label{eq:model_problem}
\begin{aligned}
    &\dot u_N = \Lambda u_N, \quad t \in (0,\infty), \\
    &u_N(0) = u_{N0},
\end{aligned}
\end{align}
where $\Lambda$ is a square matrix whose eigenvalues have negative real part, thus, $u_N\to0$ as $t\to\infty$.
The goal is to determine under which conditions the numerical reconstructed solution of \eqref{eq:model_problem} goes to zero as well.
After brief observations on the exact \node, we study the properties of the time-integration method in relation to the reduction.

\paragraph{Step 1: reduction of model problem.} 
The \node\ associated to \eqref{eq:model_problem} is 
\begin{align*}
	&\dot u_n = J_{\Psi'}\Lambda \Psi(u_n), \\
	&u_n(0) = \Psi'(u_{N0}).
\end{align*}
Following the three relevant combinations of encoder-decoder introduced in \Cref{sec:reduction}, we make some observations:
\begin{itemize}
    \item Linear encoder, linear decoder. We express the encoder in terms of its constant Jacobian matrix $\Psi'(u_N) = J_{\Psi'}u_N$. Given that $u_n = J_{\Psi'} u_N$, the exact $n$-solution is $u_n(t) = J_{\Psi'}\exp(\Lambda t)u_{N0}$. We observe that the $n$-solution goes to zero as the original one. Due to the linearity of $\Psi'$ and $\Psi$, the autoencoder defined in $[0, t_1]$, $t_1>0$, remains valid also for $(t_1, t_2], \ \forall t_2 > t_1$, indeed, $J_{\Psi'}$ is a constant matrix. From a practical standpoint, this means that the autoencoder can be determined on a short interval of time and we can then extrapolate the solution for any future time.  
	

	
	\item Linear encoder, nonlinear decoder. Due to the linearity of the encoder, we still have that $\lim_{t \to \infty} u_n = 0$. However, due to its nonlinearity, the decoder has to be determined over the entire manifold $\mathcal{M}$, and in particular for $t \in [0, \infty)$, in order for the autoencoder to reproduce $\id_{\mathcal{M}}$ (assuming that $\mathcal{M}$ is defined for $t\in[0, \infty)$).
	
	\item Nonlinear encoder, nonlinear decoder. Similarly to the previous case, the autoencoder must be determined over the entire time interval of interest. Note that the nonlinearity of the encoder may allow $u_n$ to converge to a non-zero stationary value as $t \to \infty$, despite the fact that $\lim_{t \to \infty} u_N = 0$.
	
	
	
		
\end{itemize}

\paragraph{Step 2: region of \rev{absolute} stability for Forward Euler.}
We now investigate the effects of time discretization of the \node\ propagated to the \Node, in particular, we want to establish under what conditions the numerical solution tends to zero towards infinite time\footnote{Note that the stability results hold also for a matrix $\Lambda$ with positive real part of the eigenvalues \cite{Hairer1996, Lambert1999}, although in this paper we are not really interested in this scenario.}. For simplicity, we analyze the FE scheme. As in Step 1, the analysis proceeds studying separately three types of autoencoder associated to the different types of $\mathcal{M}$:
\begin{itemize}
	\item Linear encoder, linear decoder. We express the decoder in terms of its Jacobian matrix $\Psi(u_n) = J_{\Psi}u_n$. By the linearity of the decoder, the reconstructed solution $J_\Psi u_n$, tends to zero if $u_n$ does. Therefore, it is sufficient to check the stability of the \node. By calling $A = J_{\Psi'}\Lambda J_{\Psi}$, the numerical solution at timestep $k+1$ is:
	\begin{equation*}
		\ut_n^{k+1} = (\mathbb{I} + \Delta t A)^k \ut_{n0},
	\end{equation*}
	which leads to the well-known upper bound on the timestep size: $\max_i|1+\Delta t \lambda_i| < 1$, where $\lambda_i$ is the $i^{th}$ eigenvalue of $A$. 
	
	\rev{
	In the following, we show that the upper bound for the timestep for the \node\ is at most equal to that of the \Node.}
	The n-eigenproblem, $A v = \lambda v$, $v\in\mathbb{R}^n$, can be written as 
	\begin{align*}
		J_{\Psi}Av = J_{\Psi}(J_{\Psi'}\Lambda J_{\Psi}v) = \lambda J_{\Psi}v,
	\end{align*}
	and by calling $w \coloneqq J_{\Psi}v$, $w \in \mathcal{M}$, 
	\begin{equation*}
		(J_{\Psi}J_{\Psi'}) \Lambda w = \lambda w.	
	\end{equation*}
	\rev{We observe that the matrix $P = J_{\Psi}J_{\Psi'}$ is a projection, since $(J_{\Psi}J_{\Psi'})(J_{\Psi}J_{\Psi'}) = J_{\Psi}(J_{\Psi'}J_{\Psi})J_{\Psi'} = J_{\Psi}J_{\Psi'}$.
	In general, a projection $P$ may modify the eigenvalues of $\Lambda$, unless it is an orthogonal projection. Orthogonality here means that $(w-Pw)\perp Pw$, which is equivalent to $w^\top P^\top(P-I)w = 0$.
	It can be checked that the projection induced by the autoencoder,
	$P = J_{\Psi}J_{\Psi'}$, is an orthogonal projection for points on the manifold:
	\begin{align*}
		&u_N^\top P^\top (P-I)u_N = u_N^\top (J_{\Psi}J_{\Psi'})^\top(J_{\Psi}J_{\Psi'}-I)u_N = \\
		&= u_n^\top J_{\Psi}^\top (J_{\Psi}J_{\Psi'})^\top(J_{\Psi}J_{\Psi'}-I)J_{\Psi}u_n = \\
		&=u_n^\top\left(  (J_{\Psi}^\top J_{\Psi'}^\top)\; J_{\Psi}^\top J_{\Psi} \; (J_{\Psi'} J_{\Psi})  - (J_{\Psi}^\top J_{\Psi'}^\top)\; J_{\Psi}^\top J_{\Psi} \right) u_n = \\
		&=u_n^\top\left(  I_{n\times n} \; J_{\Psi}^\top J_{\Psi} \; I_{n\times n}  - \; I_{n\times n} J_{\Psi}^\top  \; J_{\Psi} \right) u_n = 0.
	\end{align*}
	This result shows that the eigenvalues of $A$ form a subset of those of $\Lambda$, which implies that the timestep bound for the \node\ coincides, in the worst case, with that of the \Node.}
	
	\item Linear encoder, nonlinear decoder.
	 By the linearity of the encoder $\lim_{t\to\infty}u_N = 0 \Rightarrow \lim_{t\to\infty}u_n = 0$ and by the exact reduction we have that $\lim_{t\to\infty}\ut_n = 0 \Rightarrow \lim_{t\to\infty}\ut_N = 0$. So we only need to assure that $\lim_{t\to\infty}\ut_n = 0$. Unfortunately, the nonlinear decoder leads to a nonlinear \node: 
	\begin{align*}
		\ut_n^{k+1} = \ut_n^k + \Delta t J_{\Psi'}\Lambda \Psi(\ut_n^k),
	\end{align*}
	which hinders the possibility of assessing the stability by analyzing the eigenvalues of the Jacobian of the right-hand side, $J_{F_n}$ \cite{Lambert1999}, due to the nonlinear nature of the r.h.s. We refer the reader to \Cref{sec:weakly_b_stability} for a more in-depth discussion. Nonetheless, a necessary condition for stability is that small perturbations around the equilibrium point, $u_n = 0$, do not lead to divergence.
	We set $u_n^k = 0$ and we consider a sufficiently small perturbations $\delta u$ form the equilibrium point $u_n^k = 0$. Using a Taylor expansion, we have
	\begin{align*}
		u_n^{k+1} = \delta u_n + \Delta t J_{\Psi'}\Lambda \Psi(\delta u_n) = (\mathbb{I} + \Delta t J_{\Psi'}\Lambda J_{\Psi}(0))\delta u_n + \mathcal{O}(\delta u_n^2),
	\end{align*}
	from which, by setting $J_{\Psi} = J_{\Psi}(0)$, conclusions analogous to those of the previous case can be drawn.

%

	\item Nonlinear encoder, nonlinear decoder. Due to the presence of nonlinearities, $\lim_{t \to \infty} u_N = 0 \Rightarrow \lim_{t \to \infty} u_n = \alpha$, where $\alpha$ satisfies $\Psi(\alpha) = 0$. For simplicity, without loss of generality, we set $\alpha = 0$ (see \Cref{th:equivalence_autoenc}). Hence, we still need to analyze the conditions under which $\ut_n$ tends to zero. However, due to the nonlinear nature of the system, only limited conclusions can be drawn, and we refer the reader to \Cref{sec:weakly_b_stability} for further discussion.
\end{itemize}

\subsection{\rev{Nonlinear} stability}\label{sec:weakly_b_stability}
Because of the nonlinearities intrinsically present in our problem, it becomes relevant to study \rev{the region of nonlinear stability}. The goal here is to consider a dissipative \Node\ and check for which timestep size the dissipative property is satisfied also by the numerical solution.
The dissipative property means that the error between two solutions, $w_N$, $u_N$ decreases monotonically over time: $\partialderlong{\|w_N-u_N\|_2^2}{t}<0$. In what follows, the use of the 2-norm will be necessary. This property is satisfied if the one-sided Lipschitz constant of $F_N$, namely $\nu_{F_N}$, is negative \cite{Lambert1999, Hairer1996, Butcher2016}. Therefore, given $\nu_{F_N} < 0$, our model problem is
\begin{equation*}
	\begin{aligned}
		&\dot u_N = F_N(u_N), \quad t\in(0, \infty), \\
		&u_N(0) = u_{N0}.
	\end{aligned}
\end{equation*}
The satisfaction of \rev{nonlinear stability implies absolute stability} \cite{Hairer1996}. Indeed, if $\lim_{t \to \infty} u_N = 0$, the dissipative property ensures that any perturbed solution also converges to zero. 
Moreover, we notice that, if $\Lambda$ is symmetric, $F_N$ defined in \eqref{eq:model_problem} satisfies the dissipative condition. 
Therefore, with the aforementioned conditions, the results shown in this section are sufficient to meet the requirements outlined in \Cref{sec:a_stability} for symmetric $\Lambda$. 
Let us now find the relation between $\|\ut_n^{k+1} - \wt_n^{k+1}\|_2^2$ and $\|\ut_n^k - \wt_n^k\|_2^2$ obtained with the FE scheme:
\begin{align}\label{eq:b_stab_n_err}
\begin{aligned}
	&\|\ut_n^{k+1} - \wt_n^{k+1}\|_2^2 = \langle \ut_n^k + \Delta t F_n(\ut_n^k)-\wt_n^k - \Delta t F_n(\wt_n^k), \ut_n^k + \Delta t F_n(\ut_n^k)-\wt_n^k - \Delta t F_n(\wt_n^k)\rangle = \\
	&\| \ut_n^k - \wt_n^k\|_2^2 + \Delta t^2\|F_n(\ut_n^k)-F_n(\wt_n^k)\|_2^2 + 2\Delta t\langle F_n(\ut_n^k)-F_n(\wt_n^k), \ut_n^k-\wt_n^k \rangle < \\
	&< \| \ut_n^k - \wt_n^k\|_2^2 + L_{F_n}^2\Delta t^2\| \ut_n^k - \wt_n^k\|_2^2 + 2\Delta t\nu_{F_n}\| \ut_n^k - \wt_n^k\|_2^2 = (1 + L_{F_n}^2\Delta t^2 + 2\Delta t\nu_{F_n})\| \ut_n^k - \wt_n^k\|_2^2, 
\end{aligned}
\end{align}
where $\nu_{F_n}$ is the one-sided Lipschitz constant of $F_n$, assumed to be negative.
For general nonlinear encoders and decoders, the dissipativity property, i.e. $\nu_{F_n} < 0$, is not guaranteed in a straightforward manner, but it can be encouraged, for example, through an appropriate penalization term in the training loss to enforce a negative one-sided Lipschitz constant. \rev{For example, the one-sided Lipschitz constant itself computed on the minibatch, or alternatively the quantity $\exp(\langle F_n(\Psi'(u_{N,1}))-F_n(\Psi'(u_{N,2})), \ \Psi'(u_{N,1})-\Psi'(u_{N,2})\rangle)$, also evaluated selecting $u_{N,1}, u_{N,2}$ on the minibatch, may be used as a penalty term.} In this way, the stability properties of the original system can be preserved in the reduced-order surrogate. Then, from \eqref{eq:b_stab_n_err}, we derive the bound for the timestep size: 
\begin{equation}\label{eq:b_stab_timestep_bound}
1 + L_{F_n}^2\Delta t^2 + 2\Delta t\nu_{F_n} < 1 \quad\Rightarrow \quad 
    \Delta t <  \frac{2|\nu_{F_n}|}{L_{F_n}^2}, 
\end{equation}
which guarantees that the error between any two solutions of the \node\ asymptotically vanishes.
The bound in \eqref{eq:b_stab_timestep_bound} can be interpreted as a sufficient condition to enforce \rev{absolute} stability in the presence of a nonlinear \node.

\section{Approximation of \node\ by neural networks}\label{sec:nn_approx}

%
%

In this section, we consider three neural networks, $N_{\Psi'}$, $N_{\Psi}$, $N_{F_n}$, with trainable weights $\myw_{\Psi'}$, $\myw_{\Psi}$, $\myw_{F_n}$, to approximate, respectively, $\Psi'$, $\Psi$, and $F_n$ and we study the effects of this approximation. $F_n$ depends on $\mu$ and to account for this dependence we provide $\mu$ as input to the neural network, so $N_{F_n} = N_{F_n}(u_n, \mu)$. For the sake of notation, unless it is relevant for the context, we avoid to explicitly write the parameters as argument.

The core idea of the following analysis is to interpret the functions approximated by neural networks as perturbations of their exact counterparts. 
Accordingly, we make the following definition
\begin{definition}[Discrepancies $\kappa$, $\upsilon$, $\omega$] We denote by $\kappa$, $\upsilon$, $\omega$ the discrepancies between the encoder, decoder, right hand side of the \node and their approximations, such that 
\begin{equation}\label{eq:networks}
    N_{\Psi'}(u_N) = \Psi'(u_N) + \kappa(u_N), \
    N_{\Psi}(u_n) = \Psi(u_n)  + \upsilon(u_n), \
    N_{F_n}(u_n) = F_n(u_n)  + \omega(u_n), 
\end{equation}
\end{definition}
The universal approximation theorems state that, within appropriate functional spaces, neural networks can approximate functions with arbitrarily small error. We refer to \cite{Park2020, Cai2022, Li2023, Hanin2017} and the references therein for results concerning the approximation of continuous functions, such as $\Psi'$, $\Psi$, and $F_n$. 



A further observation is in order: since $F_n$ depends on the Jacobian of the encoder (see \eqref{eq:def_F_n}), it is advisable to employ smooth activation functions for $N_{\Psi'}$ \rev{(and not necessarily for $N_{F_n}$ nor $N_{\Psi}$)}, such as the ELU, to ensure differentiability.

According to the aforementioned literature, and considering the discussion in \Cref{sec:reduction}, the approximation errors $\kappa$, $\upsilon$, and $\omega$ can be bounded by constants provided the neural networks are sufficiently expressive and well trained. Specifically, these errors satisfy
$\|\kappa\| \leq \eps_{\Psi'}$,
$\|\upsilon\| \leq \eps_{\Psi}$, and
$\|\omega\| \leq \eps_{F_n}$,
where each $\eps$ depends on the particular function being approximated.
The sources of error can be due to 
\begin{itemize} 
	\item insufficient expressivity of the neural networks and/or inappropriate architecture choice, represented by $\Varepsilon_1$ and $\Varepsilon_2$. However, the approximation theorems in \cite{Park2020, Cai2022, Li2023, Hanin2017} and the discussion in \Cref{sec:autoencoders} ensure that, using a proper architecture, as $|\myw| \to \infty$, this source of error vanishes; 
	\item optimization procedure used to determine $\myw$, which may yield a suboptimal set of parameters due to non-convergence or convergence to a local minimum and/or insufficient training data; 
	\item the use of a time integration scheme with finite timestep size for training the neural networks. 
\end{itemize}
Here, we are not interested in studying the second contribution. Instead, we assume that the optimization process leads to the global optimum and, additionally, that the neural networks are sufficiently large, and that enough samples in $\mathcal{P}$ are available. A brief discussion on the third error source is done in \Cref{sec:train}.
The upper bounds $\eps_{\Psi'}$, $\eps_{\Psi}$, and $\eps_{F_n}$ include all the aforementioned sources of error.

\paragraph{Stability - Propagation of neural network's errors.}\label{sec:stability_exact_nn}
We begin our analysis by proceeding step by step. First, we examine the Lyapunov stability of the exact-in-time problem, and subsequently, we account for the effects of time discretization.
From \eqref{eq:networks} we understand that the analysis resembles that of \Cref{sec:stability_exact}, but here the perturbed solution, $w_n$, derives from the application of neural networks to the \node. We have that the error is bounded by the accuracy of the neural networks (see \Cref{sec:computations} for the computations): 
\begin{equation*}
	\|w_N - u_N\| \leq L_\Psi\Big( \eps_{\Psi'} + \eps_{F_n}t \Big) e^{L_{F_n}t} + \eps_\Psi,
\end{equation*}
showing that the approximation error introduced solely by the neural networks is controlled.

\paragraph{Fully approximated problem.}
We move now to the fully approximated problem, $\mathfrak{P}(\uut)$, where $\uut = (\uut_N, \uut_n)$ and $\uut_N$, $\uut_n$ represent, respectively, the reconstructed and the $n$-solution: 
\begin{align}
	\mathfrak{P}(\uut) =
	\begin{cases}\label{eq:fully_approximated}
	\uut_N^{k+1} = N_\Psi(\uut_n^{k+1}), \\
	\mathscr{L}_n(\uut_n^{k+1},\uut_n^{k},\ldots,\uut_n^{k+1-P}, N_{F_n}, \Delta t) = 0, \quad k=P-1,P,\ldots \\
	\uut_n^0 = N_{\Psi'}(u_N^0).
\end{cases}
\end{align}
The main objective is to ensure that the fully-approximated numerical solution can represent accurately the exact one. We treat the time integration scheme combined with the neural networks as a new numerical method, whose accuracy can be controlled through some sets of hyper-parameters: the timestep size separating two consecutive training data, called \textit{training timestep size}, $\Delta t_{\text{train}}$, the timestep size used during the online phase, $\Delta t$, and the number of trainable weights, $|\myw|$ (\Cref{sec:neural_network}). Therefore, the goal is to ensure that $\|\uut_N - u_N\| \to 0$ as $\Delta t_\text{train}, \Delta t \to 0$ and $|\myw|\to \infty$.
Before continuing, we need to present the training	 procedure in \Cref{sec:train} and afterwards discuss the convergence in \Cref{sec:convergence}.

\subsection{Training}\label{sec:train}
In this section, we outline the methodology for training  the neural networks. The training process is formulated as an optimization problem, in which a loss function is computed and minimized using data. 
The first step involves the generation of the datasets. We sample the parameters from a uniform distribution over the parameter space, and the corresponding solutions, $u_N$, along with any additional required quantities, are computed and stored. The resulting data are partitioned into three subsets: training, validation, and test datasets. The validation and test sets each comprise at least $10\%$ of the total data.
The second step consists of the minimization procedure, using a gradient-based method, for which we consider two scenarios, detailed below.

In the following, we denote by $M_N$ the integration in time scheme for solving the \Node\ and similarly for $M_n$.
Let $\vt_N^k$ be the discrete $N$-solution obtained through $M_N$. The manifold defined by $\{\vt_{N,i}^k\}_{\mu_i \in \mathcal{P}}$ differs from the exact one, $\mathcal{M}$, depending on the accuracy of $M_N$. Note that we add the subscript $i$ in $\vt_{N,i}^k$ to denote the solution for the parameter $\mu_i$. The solution $\vt_N^k$ are computed with a timestep size $\Delta t_N$. Without altering the main results, for the sake of simplicity, we henceforth assume that $\Delta t_\text{train} = \Delta t_N$, and we simply write $\Delta t_\text{train}$. We consider the set $\{\vt_N^k\}$ to constitute the training data, as is commonly done in practical applications.


\subsubsection{Semi data-driven}\label{sec:semi_datadriven}
The first approach consists in jointly training $N_{\Psi'}$ and $N_{\Psi}$, followed by the training of $N_{F_n}$ to approximate $F_n$ by directly minimizing the discrepancy between $N_{F_n}$ and $F_n$. We compute $F_n$ from its definition \eqref{eq:def_F_n}, so it is necessary to store the values of $F_N$ together with the solutions $u_N$ for each timestep size and parameter instance. This procedure is not purely data-driven, as $F_N$ is typically not provided as an output by the solver. Therefore, some understanding of the solver's implementation is required to extract and store $F_N$. Following the aforementioned ideas, the first step is to force the autoencoder to reproduce the identity by minimizing a functional $\mathscr{J}_1$ computed on a minibatch of size $n_{\text{train}}\times N_{\text{time}}$, containing $n_{\text{train}}$ samples in $\mathcal{P}$ and $N_{\text{time}}(\mu)$ timesteps for each parameter sample. Calling $\myw_{\Psi'}$ and $\myw_{\Psi}$ the trainable weights of $N_{\Psi'}$ and $N_\Psi$, respectively, the loss $\mathscr{J}_1$ is defined as 
\begin{align}\label{eq:loss_autoenc}
	\mathscr{J}_1 = \frac{1}{n_{\text{train}} N_{\text{time}}}\sum_{i=1}^{n_{\text{train}}} \sum_{j=1}^{N_{\text{time}}} \mathscr{N}_{ij}^1, \qquad \mathscr{N}_{ij}^1(\myw_{\Psi'}\cup \myw_{\Psi}) = \| \vt_{N,i}^j - N_{\Psi}\circ N_{\Psi'}(\vt_{N,i}^j) \|_2^2. 
\end{align}
The second step is to learn $F_n$ by minimizing $\mathscr{J}_2$:
\begin{equation}\label{eq:loss_Fn_semi}
	\mathscr{J}_2 = \frac{1}{n_{\text{train}} N_{\text{time}}}\sum_{i=1}^{n_{\text{train}}} \sum_{j=1}^{N_{\text{time}}} \mathscr{N}_{ij}^2, \qquad \mathscr{N}_{ij}^2(\myw_{F_n}) =\| J_{N_{\Psi'}}F_{N,i}^j - N_{F_n,i}^j  \|_2^2,
\end{equation}
where $\myw_{F_n}$ are the trainable weights of $N_{F_n}$.

\subsubsection{Fully data-driven}\label{sec:fully_datadriven}
In this second approach, the data consist only in samples of $\vt_N$, not $F_N$. 
The training of $N_{\Psi'}$, $N_{\Psi}$, and $N_{F_n}$ can be made concurrently by minimizing:  
\begin{equation*}
	\mathscr{J}_3 = \frac{1}{n_{\text{train}} N_{\text{time}}}\sum_{i=1}^{n_{\text{train}}} \sum_{j=1}^{N_{\text{time}}}\sum_{r=\{1,3\}} \eta_r\mathscr{N}_{ij}^r,
\end{equation*}
where $\eta_r>0$ are user-defined scalar weights, $\mathscr{N}_{ij}^1$ is the loss for the autoencoder, defined in \eqref{eq:loss_autoenc}, $\mathscr{N}_{ij}^3$ is the loss contribution to approximate $F_n$ by minimizing the residual of \eqref{eq:lmm_diff_operator}, defined as:
\begin{equation}\label{eq:loss_2_fully_datadriven}
	\mathscr{N}_{ij}^3(\myw_{F_n}) = \left\| \sum_{p=0}^P \alpha_p \overline{N}_{\Psi'}(\vt_{N,i}^{j+1-p}) - \Delta t_\text{train} \sum_{p=0}^P \beta_p N_{F_n}(\mu_i, \overline{N}_{\Psi'}(\vt_{N,i}^{j+1-p}))\right\|_2^2,
\end{equation}
where $\overline{N}_{\Psi'}(\cdot)$ is the output of $\Psi'$ for fixed value of the trainable weights. Indeed $\mathscr{N}_{ij}^3$ depends only on the trainable weights of $N_{F_n}$.

It is relevant to note that, thanks to the use of the autoencoder, it is possible to compute the $n$-solution as $\overline{\Psi'}(\vt_{N,i}^k)$ breaking the dependence of $u_{n,i}^k$ on its value at the previous timestep, $u_{n,i}^{k-1}$. Therefore, the temporal loop can be broken, enabling the possibility of a straightforward parallelization of the computation of $\mathscr{N}_{ij}^3$. Doing so, it is possible to fully exploit the computational power of GPUs. Furthermore, by providing data at different times, the risk of overfitting is mitigated.

Moreover, we highlight that the reduced model is trained by minimizing a single loss function, $\mathscr{J}_3$, whose components, $\mathscr{N}_{ij}^1$ and $\mathscr{N}_{ij}^3$, depend exclusively on either $\myw_{\Psi'} \cup \myw_{\Psi}$ or $\myw_{F_n}$. This separation simplifies the potentially intricate and highly nonlinear structure of the loss function.

\subsubsection{Characterization of $F_n$ approximation error.}
We now study how accurate the two training strategies are in approximating $F_n$. 
We assume that $|\myw_{F_n}|\to \infty$ and layers large enough that $F_n$ is arbitrarily expressive. This assumption allows us to introduce an additional setting: we assume that the minimum value of the loss function is attained, i.e., $\mathscr{J}_i = 0$, $i=1,2,3$. 
Our goal is to identify necessary condition under which the minimum value $\mathscr{J}_i$ implies $\max_{k=0,1,\ldots} \|N_{F_n}(t_k) - F_n(t_k)\| \to 0$.

Regarding the semi data-driven methods, the minimization of \eqref{eq:loss_Fn_semi} implies enforcing a match between $N_{F_n}$ and $J_{N_{\Psi'}} F_N$. It follows that a necessary condition for $\max_{k=0,1,\ldots}\|N_{F_n}(t_k) - F_n(t_k)\| \to 0$ is that $\|J_{N_{\Psi'}}(t_k)-J_{\Psi'}(t_k)\|_\infty \to 0$, therefore, an accurate representation of the encoder and its Jacobian is necessary.

Fully data-driven methods should be treated with care. We first notice that, by the definition of the local truncation error, the following equation holds:
\begin{equation}\label{eq:diff_op_tau}
		\left\|\sum_{p=0}^P \alpha_p \Psi'(u_N(t_{j+1-p})) - \Delta t_\text{train}\sum_{p=0}^P \beta_p F_n(u_N(t_{j+1-p})) - \Delta t_\text{train}\tau_n^j  \right\|_2^2 = 0.
\end{equation}
We call $\zeta_j$ the difference between the discrete and the exact  $N$-solutions: $\zeta_j = \vt_N^j-u_N(t_j)$. Assuming a convergent $M_N$, we have that $\zeta_j \to 0$ for $\Delta t_\text{train} \to 0$, $j=0,1,\ldots$.
We introduce now the neural network for the encoder, $N_{\Psi'} = \Psi' + \kappa$. We remind that, assuming enough data, accurate training, and proper \rev{choice} of the architecture, by \Cref{th:uap} and \Cref{th:reduction}, the error $\kappa$ can be made arbitrarily small. By the last considerations, \eqref{eq:loss_2_fully_datadriven} can be rewritten as
\begin{align}\label{eq:what_nn_learns_1}
\begin{aligned}			
	\mathscr{N}_{ij}^3 = &\Bigg\| \sum_{p=0}^P \alpha_p\left( \Psi'(u_{N}(t_{j+1-p})+\zeta_{j+1-p}) + \kappa(u_{N}(t_{j+1-p})+\zeta_{j+1-p})\right) + \\
	&- \Delta t_\text{train} \sum_{p=0}^P \beta_p N_{F_n}(\Psi'(u_{N}(t_{j+1-p})+\zeta_{j+1-p}), \mu_i)\Bigg\|_2^2,	
\end{aligned}
\end{align}
Given the assumption $\mathscr{J}_i=0$, we set $\mathscr{N}_{ij}^3
= 0$. Then, setting \eqref{eq:what_nn_learns_1} to zero and comparing it to \eqref{eq:diff_op_tau}, we conclude that
\begin{proposition}\label{th:fully_works} 
	The fully data-driven is a convergent method for training $N_{F_n}$ in the sense that, under proper conditions, $\max_{k=0,1,\ldots}\|N_{F_n}(t_k) - F_n(t_k)\| \to 0$.
	A necessary condition to meet the requirement $\max_{k=0,1,\ldots}\|N_{F_n}(t_k) - F_n(t_k)\| \to 0$ is that $\tau_n^k, \zeta_k \to 0$ and $\kappa \to 0$. The former can be achieved with $\Delta t_\text{train} \to 0$, the latter can be achieved with $|\myw_{\Psi'} \cup \myw_{\Psi}| \to \infty$, or if $\mathcal{M}$ is linear so that a linear autoencoder, thus a finite number of trainable weights, is sufficient for having $\kappa = 0$.
\end{proposition}
From a practical point of view, equation~\eqref{eq:what_nn_learns_1} states that $N_{F_n}$, through the optimization process described in \Cref{sec:fully_datadriven}, approximates the correct right-hand side, up to an error introduced by the discretization of the \Node\ (due to $\zeta$), an error arising from the enforcement of the reference $n$-solution, obtained via compression of the full-order one, to fit the $M_n$ scheme (due to $\tau_n$), and an additional error due to the approximation of the encoder (due to $\kappa$). 

Further details can be found in \cite{Du2022}, where it is shown that, assuming exact data, which in our framework corresponds to the availability of $u_n$, the approximation error resulting from the minimization of the residual is bounded by $C_2(\Delta t_\text{train}^P + e_{\mathcal{A}})$, were $C_2$ is a constant depending on $M_n$, $P$ is the order of convergence of $M_n$, and $e_{\mathcal{A}}$ denotes the error introduced by the approximation capability of $N_{F_n}$.

In the fully data-driven approach, due to the intrinsic errors introduced by \eqref{eq:what_nn_learns_1}, which leads $N_{F_n}$ to possibly not represents $F_n$ accurately, the evaluation of the trained reduced order model should be performed using the same scheme $M_n$ and the same timestep size employed during training, $\Delta t = \Delta t_\text{train}$. On the other hand, the semi data-driven approach allows the possibility to change both $M_n$ and $\Delta t$ during the online phase, while still yielding reliable solutions. See \Cref{sec:case_2_time_convergence} for numerical evidence.

\subsection{Solution of the fully discrete problem}
We consider the $P$-steps scheme whose difference operator is written in \eqref{eq:lmm_diff_operator} and it is rewritten here with a more convenient form for this section:
\begin{equation*}
	\sum_{p=0}^{P}\alpha_p\uut_n^{k+1-p} =  \Delta t\sum_{p=0}^{P}\beta_p N_{F_n}(\uut_n^{k+1-p}), \quad k = P-1, P, \ldots
\end{equation*}
with the following auxiliary conditions for the first $P-1$ steps, derived from a one-sided finite difference of the same order of convergence of the $P$-steps scheme:
\begin{equation}\label{eq:auxiliary_conditions}
	\sum_{p=0}^{P}\gamma_p\uut_n^{k+p}=\Delta t N_{F_n}(\uut_n^k), \quad k=0,\ldots,P-1.
\end{equation}
To ease the notation, we set $n=1$ although the results are easily extendable to higher dimension. 
According to the time integration scheme, for $K\geq2P - 1$, we define the following matrices:
\begin{align*}
	&A = \begin{bmatrix}
		\alpha_P, \alpha_{P-1}, \ldots, \alpha_0, 0, \ldots, 0\\
		0,\alpha_P, \alpha_{P-1}, \ldots, \alpha_0, 0, \ldots, 0 \\
		\vdots \\
		0, \ldots, 0, \alpha_P, \alpha_{P-1}, \ldots, \alpha_0
	\end{bmatrix} \in \mathbb{R}^{(K-P+1)\times K}, \
	&&B = \begin{bmatrix}
		\beta_P, \beta_{P-1}, \ldots, \beta_0, 0, \ldots, 0\\
		0,\beta_P, \beta_{P-1}, \ldots, \beta_0, 0, \ldots, 0 \\
		\vdots \\
		0, \ldots, 0, \beta_P, \beta_{P-1}, \ldots, \beta_0
	\end{bmatrix} \in \mathbb{R}^{(K-P+1)\times K},\\
	&\Gamma = \begin{bmatrix}
		\gamma_0, \ldots, \gamma_{P-1},  \gamma_P, 0, \ldots, 0\\
		0,\gamma_0, \ldots, \gamma_{P-1},  \gamma_P, 0, \ldots, 0 \\
		\vdots \\
		0, \ldots, 0, \gamma_0, \ldots, \gamma_{P-1}, \gamma_P
	\end{bmatrix} \in \mathbb{R}^{(P-1)\times K}, \
	&&C = \begin{bmatrix}
		1, 0, \ldots, 0\\
		0, 1, 0, \ldots, 0 \\
		\vdots \\
		0, \ldots, 0,1,0,\ldots,0        
	\end{bmatrix}\in \mathbb{R}^{(P-1)\times K},
\end{align*}
and the following vector whose elements are the solution at each timesteps: $\m \uut_n = [\uut_n^0, \uut_n^1, \ldots, \uut_n^K]^T$ and similarly for $\m N_{F_n}$. Note that, to avoid any confusion, the arrays of solutions at each timestep are denoted with the underline.
The solution in time of the \node\ is obtained from the following nonlinear system:
\begin{equation}\label{eq:sys_lmm}
	\begin{bmatrix}
		\Gamma \\
		A
	\end{bmatrix}
	\m \uut_n = \Delta t 
	\begin{bmatrix}
		C\\
		B
	\end{bmatrix}
	\m N_{F_n}(\m \uut_n). 
\end{equation}

\subsection{Global convergence}\label{sec:convergence}

We now proceed to study the error $\|u_N - \uut_N\|$ between the exact solution and the solution of the fully approximated problem \eqref{eq:fully_approximated}, and the corresponding matrix form with initial conditions included \eqref{eq:sys_lmm}. This result explicitly relates the total error to the neural network approximation error of the reduced dynamics and the decoder reconstruction accuracy, offering a key theoretical guarantee for the effectiveness of the DRE method. \rev{We remark in advance that the results in this section depend on the approximation errors $\eps_{\Psi'}$, $\eps_{\Psi}$, and $\eps_{F_n}$, which are generally unknown since the exact $\Psi'$ and $\Psi$ are not available. It follows that the following discussion is primarily intended as a theoretical result, necessary to establish the convergence property of the proposed method.}
%
\begin{theorem}[Global convergence]\label{th:global_convergence}
\textit{We make the following assumptions: $\Psi'$, $\Psi$, and $F_n$ are continuous functions; the data are generated by approximating the \Node\ using a convergent scheme, $M_N$; the training leads to the global minimum of the loss functions; we use a convergent time integration scheme $M_n$ of order $P$. Then, we have the following bound:
	\begin{equation}\label{eq:final_error}
		\max_{\mu} \|\m u_N - \m \uut_N \|_{\infty} \leq L_{\Psi}(\eps_{\Psi'}(\myw) + \eps_{F_n}(\Delta t_\text{train},\myw)T) e^{L_{F_n} T} + L_{\Psi}C_{1}\Delta t^P + \eps_{\Psi}(\myw).
	\end{equation}
	So, for $|\myw|\to\infty$ and $\Delta t, \Delta t_{\text{train}} \to 0$, we have $\max_{\mu} \|\m u_N - \m \uut_N \|_{\infty} \to 0$.}
\end{theorem}

\begin{proof}	
	The error between the fully approximated solution and the exact one can be decomposed as
	\begin{equation}\label{eq:fully_err_2}
	    \|\m u_n - \m \uut_n\|_\infty \leq \underbrace{\|\m u_n - \m w_n\|_{\infty}}_{H} + \underbrace{\|\m w_n - \m \uut_n\|_{\infty}}_{Q},
	\end{equation}
	where $w_n$ is the exact solution of the following \node
	\begin{align}\label{eq:approxnode}
		\begin{aligned}
			&\dot{w}_n = N_{F_n}(w_n), \\
			&w_n(0) = N_{\Psi'}(u_{N0}).	
		\end{aligned}
	\end{align}
	In \eqref{eq:fully_err_2}, $H$ is due to neural network approximation of the right-hand side, and $Q$ is due to the time discretization.
	Using \eqref{eq:networks}, equation \eqref{eq:approxnode} can be written as
	\begin{align}\label{eq:approxnodeerrors}
		\begin{aligned}
			&\dot{w}_n = F_n(w_n) + \omega(w_n), \\
			&w_n(0) = \Psi'(u_{N0}) + \kappa(u_{N0}).	
		\end{aligned}
	\end{align}
	Subtracting \eqref{eq:node_0} to \eqref{eq:approxnodeerrors}, we have
	\begin{align}
		\begin{aligned}
			&\dot{w}_n - \dot{u}_n = F_n(w_n) - F_n(u_n) + \omega(w_n), \\
			&w_n(0) = \kappa(u_{N0}).	
		\end{aligned}
	\end{align} 
	Applying Gr\"onwall's lemma yields
	\begin{equation}\label{eq:bound_H}
	    H = \|\m w_n - \m u_n \|_\infty \leq (\eps_{\Psi'} + \eps_{F_n}T) e^{L_{F_n,\mu} T}.
	\end{equation}
	The term $Q$ in \eqref{eq:fully_err_2} describes the error introduced by the time discretization of \eqref{eq:approxnode}, which depends on the specific time integration scheme used and can generally be bounded by \cite{Hairer1996,Lambert1999, Butcher2016}
	\begin{equation}\label{eq:bound_Q}
	    Q = \|\m w_n - \m\uut_n\|_\infty \leq C_{1,\mu}\Delta t^P,
	\end{equation}
    where $C_{1,\mu} > 0$ generally depends of the specific r.h.s. at hand and time integration scheme.
	Finally, replacing \eqref{eq:bound_H} and \eqref{eq:bound_Q} into \eqref{eq:fully_err_2} we obtain the following bound
	\begin{equation}\label{eq:fully_err_un_2}
	    \| \m u_n - \m \uut_n \|_\infty \leq (\eps_{\Psi'}(\myw) + \eps_{F_n}(\Delta t_\text{train},\myw)T) e^{L_{F_n,\mu} T} + C_{1,\mu}\Delta t^P,
	\end{equation}
	where we have emphasized the dependence of the neural network approximation error on the trainable weights, $\myw$, and training timestep, $\Delta t_\text{train}$.
	
	From \eqref{eq:fully_err_2}, and using \eqref{eq:networks}, we obtain the error on the reconstructed solution for one trajectory: 
	\begin{align*}
		&\|\m u_N - \m \uut_N \|_{\infty} = \| \m \Psi(\m u_n) - \m \Psi(\m \uut_n) - \m \upsilon(\m \uut_n) \| = 
        \|\m \Psi(\m u_n) - \m \Psi(\m \uut_N) -\m \upsilon(\m \uut_n) - \m \Psi(\m w_n) + \m \Psi(\m w_n) \|_{\infty} \leq \\
		&\leq L_{\Psi}\|\m u_n - \m w_n\|_{\infty} + L_{\Psi}\|\m w_n - \m \uut_n\|_{\infty} + \eps_{\Psi} \leq L_{\Psi}(\eps_{\Psi'} + \eps_{F_n}T) e^{L_{F_n,\mu} T} + L_{\Psi}C_{1,\mu}\Delta t^P + \eps_{\Psi}.
	\end{align*}
	Taking the maximum across the trajectories leads to the global error: 
	\begin{equation}\label{eq:final_error}
		\max_{\mu} \|\m u_N - \m \uut_N \|_{\infty} \leq L_{\Psi}(\eps_{\Psi'}(\myw) + \eps_{F_n}(\Delta t_\text{train},\myw)T) e^{L_{F_n} T} + L_{\Psi}C_{1}\Delta t^P + \eps_{\Psi}(\myw),
	\end{equation}
	where $C_1 = \max_{\mu}C_{1,\mu}$.
	Using the fully data-driven training, by \Cref{th:fully_works}, we have that $\eps_{F_n} \to 0$ as $|\myw| \to \infty$ and $\Delta t_{\text{train}} \to 0$. By the assumption of sufficient training data and that a global minimum of the loss is reached, the term $\eps_\Psi \to 0$ as $|\myw| \to \infty$. Finally, the error introduced by the integration in time scheme for the solution of the \node\ $L_{\Psi} C_1 \Delta t^P \to 0$ as $\Delta t \to 0$.
\end{proof}
We observe that the encoder appears in the error constant that multiplies the exponential ($\eps_{\Psi'}$) and in $L_{F_n}$ (see \eqref{eq:L_F_n}). The decoder appears in terms of its Lipschitz constant, in $L_{F_n}$, and as an additional term with $\eps_\Psi$. Therefore, we understand that it is important to ensure that both the encoder and the decoder are accurately approximated.


\begin{remark}Similarity with proper orthogonal decomposition:
the DRE can be interpreted as a nonlinear generalization of Proper Orthogonal Decomposition (POD) \cite{Hesthaven2016, Quarteroni2016, Ohlberger2016}, in which the linear transition matrices obtained via singular value decomposition of the snapshot matrix are replaced by neural networks. Furthermore, our methodology also involves approximating the right-hand side $F_n$ of the \node, in order to further accelerate the overall computational procedure.
\end{remark}

\clearpage

\section{Numerical tests}\label{sec:numerical_studies}
In this section, we present a few test cases to numerically verify the theory developed and show a generic application of DRE. The search of the optimal neural network is not the goal of this paper, so only a preliminary manual optimization of the hyperparameters, such as the number of data, batch size, number of layers, learning rate and its schedule, and type of activation function, is done. 
In all test cases, we use the Adam \cite{Kingma2014} optimizer implemented in PyTorch v2.6 \cite{Pytorch} using the default settings to perform minimization.

\subsection{Test 1: \rev{absolute} stability}\label{sec:case_1_a_stability}
Here we study numerically the stability after reduction (\Cref{sec:a_stability}). To properly isolate this property from other discretization error sources, we do not approximate $F_n$ with $N_{F_n}$ but, instead, we write it as its definition \eqref{eq:def_F_n}. For practical simplicity, $\Psi'$ and $\Psi$ are approximated with $N_{\Psi'}$ and $N_{\Psi}$, respectively. When not linear, these neural networks are made highly expressive to ensure a negligible approximation error for the purpose of this study.
Due to the limited theoretical results available for the general case, we focus here on two specific scenarios: the linear autoencoder and the linear–nonlinear autoencoder.

\subsubsection{Linear encoder -- linear decoder}
The \Node\ taken into account is 
\begin{align}\label{eq:lin_ode_test}
\begin{aligned}
	&\dot u_N = \Lambda u_N,\\
	&u_N(0) = [\mu, 1,1]^T,
\end{aligned}
\end{align}
where $u_N \in \mathbb{R}^3$, $\Lambda \in \mathbb{R}^{3 \times 3}$ is a diagonal matrix with entries $[-3, -5, -5]$ along the diagonal, and $\mu\sim\mathcal{U}(1,5)$ is the parameter. We observe that a linear ODE system, such as \eqref{eq:lin_ode_test}, may arise from the spatial discretization of a PDE representing the heat equation. As mentioned in \Cref{sec:a_stability}, in this scenario we expect an arbitrarily long extrapolation in time without instabilities.

The reduction maps a vector in $\mathbb{R}^3$ to a vector in $\mathbb{R}^2$. A low-dimensional setting is intentionally chosen to ease the graphical visualization of the results. Since $\Lambda$ has real eigenvalues, no oscillatory behavior is observed, and the manifold is linear. Indeed, by calling $x=e^{-3t}$ and $y = x^{5/3}$, we have $[u_{n,1}, u_{n,2}, u_{n,3}]=[\mu e^{-3t}, e^{-5t}, e^{-5t}] = [\mu x, y, y]$, which represents a flat surface. Therefore, a linear autoencoder is sufficient to approximate the identity map $\id_{\mathcal{M}}$. In this case, $\Psi'$ and $\Psi$ are linear and, for practical convenience, they are approximated by $N_{\Psi'}$ and $N_{\Psi}$, that are matrices whose entries are still determined through an optimization process to minimize the loss in \eqref{eq:loss_autoenc}. \rev{The networks' architectures are reported in \Cref{sec:computations}, Tab.~\ref{tab:case_1_architecture}}.\\
\rev{We report below the numerical values, up to the third digit, of the Jacobians of the encoder and decoder, together with the reduced matrix $A = J_{N_{\Psi'}}\Lambda J_{N_\Psi}$ and its eigenvalues:
\begin{align*}
\begin{aligned}
	&J_{N_{\Psi'}} = \begin{bmatrix}
		1.221 &  0.200 &  1.097 \\
		-0.785 & -0.020 &  0.204
	\end{bmatrix}, && J_{N_{\Psi}} =
	\begin{bmatrix}
		0.149 & -1.043 \\
		0.631 &  0.982 \\
		0.631 &  0.982
	\end{bmatrix}, \\
	&A =\begin{bmatrix}
		-4.637 & -2.547 \\
		-0.233 & -3.363
	\end{bmatrix}, && \text{eig}(A) = -5, -3. 
\end{aligned}
\end{align*}
As expected, the numerical values of the eigenvalues of $A$ coincide with those of $\Lambda$ (see \Cref{sec:a_stability}), which will lead to a timestep size bound for the \node\ equal to that of the \Node.
}
The manifold $\mathcal{M}$ is composed of $100$ distinct trajectories, each defined over the time interval $[0, T]$ with $T = 0.5$.
In Fig.~\ref{fig:a_stab_linear_linear}, we present the components of the reconstructed solution, $\Psi(\ut_n)$, as a function of time, where $\ut_n$ is obtained by solving the \node\ using the FE scheme. 
\rev{The \node\ is solved using two different timestep sizes. The first is $0.99\Delta t_{\max} = 0.99 \times \frac{2}{5}$ and is shown in the first row of Fig.~\ref{fig:a_stab_linear_linear}; the second is $0.5\Delta t_{\max} = 0.5 \times \frac{2}{5}$ and is shown in the second row of the same figure.	As expected, for $\Delta t = 0.99\Delta t_{\max}$ the solution exhibits oscillations, although it decays to zero as time tends to infinity. No oscillations are observed for the smaller time-step size.
} 
Additionally, we remark that the end of the training time, $T = 0.5$, is highlighted in blue in Fig.~\ref{fig:a_stab_linear_linear}, while the solution is extrapolated up to $t = 100$.    
\begin{figure*}[h]
	\centering
	\begin{turn}{90}
		\hspace{10mm} {$\Delta t = 0.99\Delta t_{\max}$} 
	\end{turn}
	\hspace*{1mm}
	\includegraphics[width=0.285\textwidth]{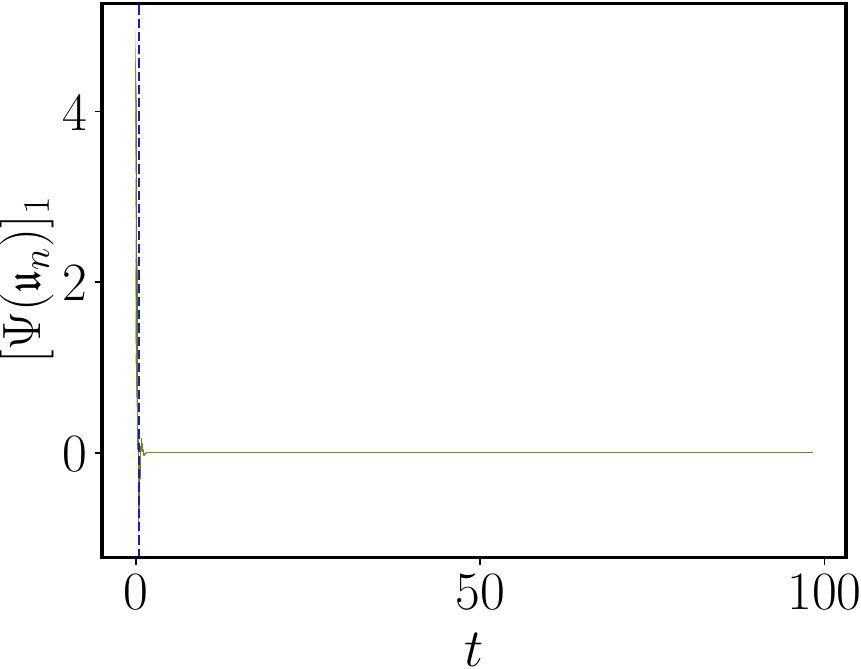}\hspace{0.3cm}
	\includegraphics[width=0.31\textwidth]{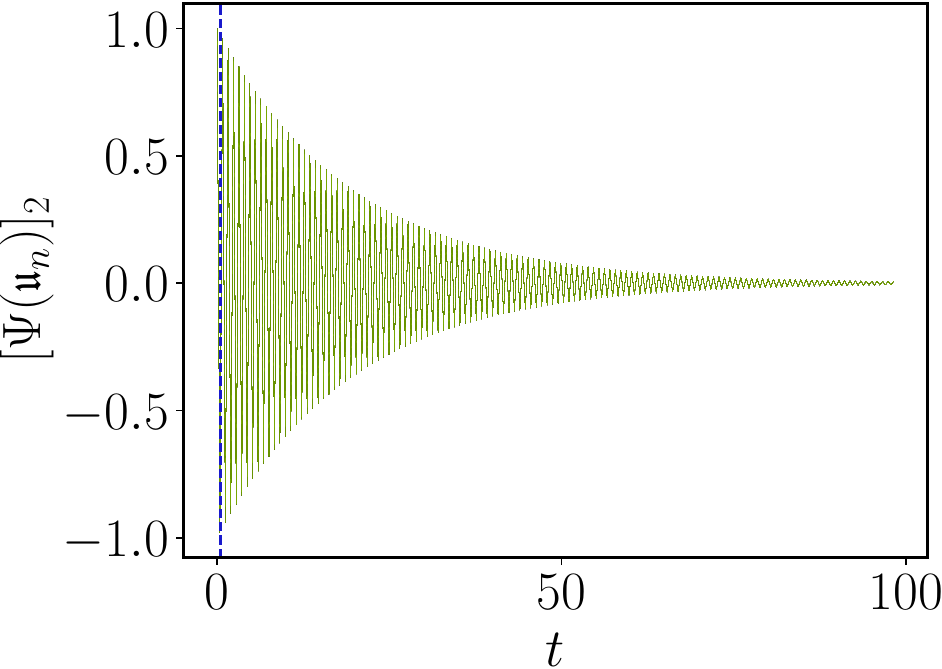}\hspace{0.3cm}
	\includegraphics[width=0.31\textwidth]{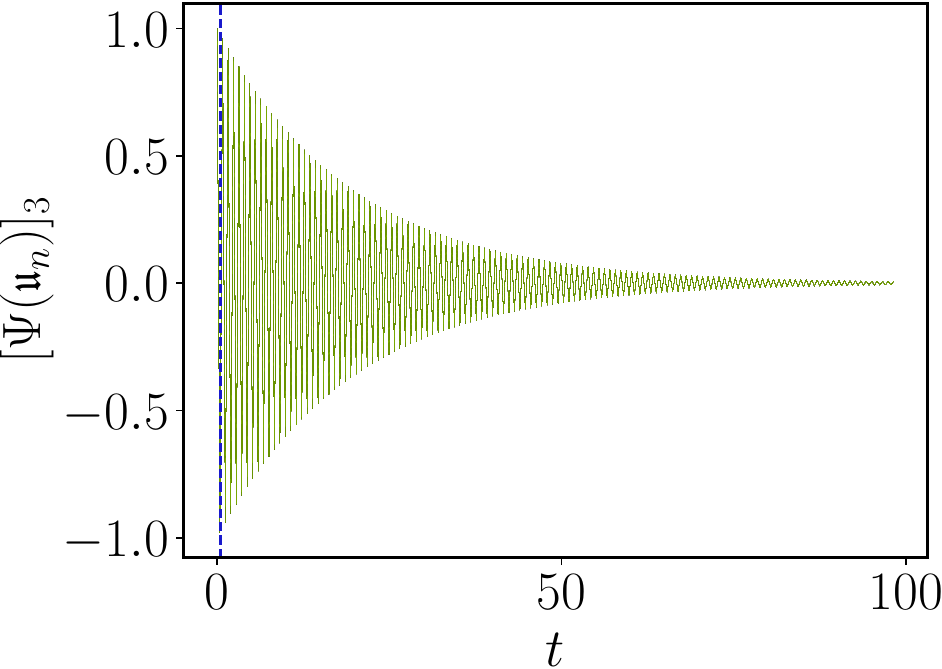} \\
	\vspace*{4mm}
	\begin{turn}{90}
		\hspace{10mm} {$\Delta t = 0.5\Delta t_{\max}$} 
	\end{turn}
	\hspace*{1mm}
	\includegraphics[width=0.285\textwidth]{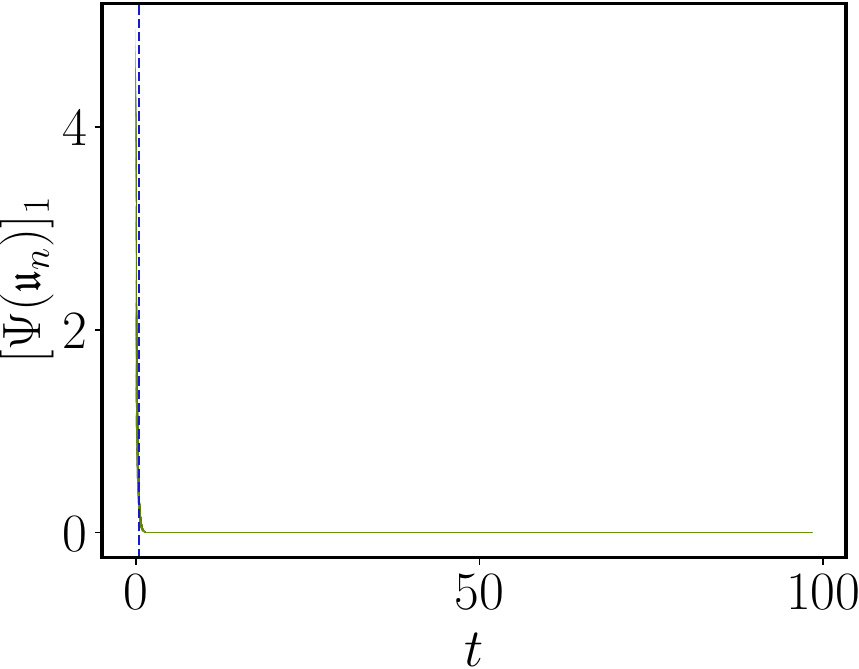}\hspace{0.3cm}
	\includegraphics[width=0.31\textwidth]{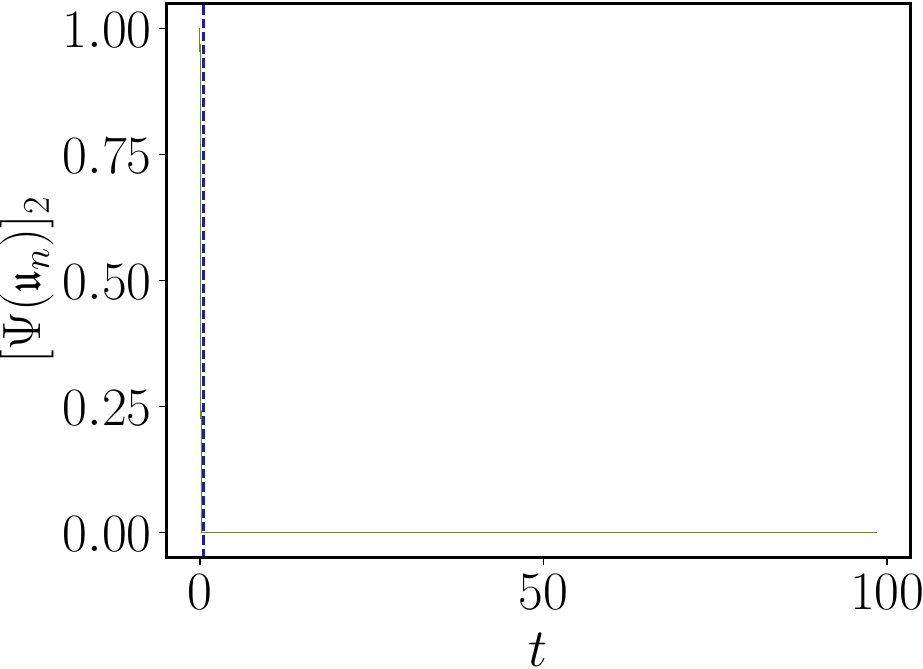}\hspace{0.3cm}
	\includegraphics[width=0.31\textwidth]{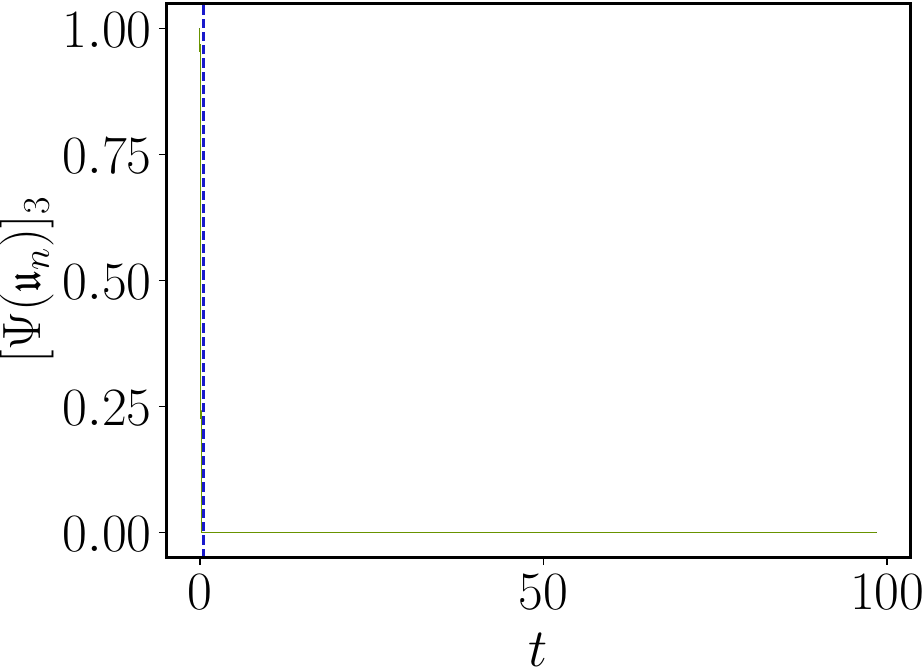}
	\caption{\rev{Linear case. Three components, one for each panel, of the reconstructed solutions in the test dataset.
	First row: the numerical solutions are computed using FE with $\Delta t = 0.99 \Delta t_{\max}$. Second row: The timestep size is reduced to $\Delta t = 0.5 \Delta t_{\max}$, at which the oscillations are absent. 
	The vertical blue-dashed line is located at the end of the training time.}}
	\label{fig:a_stab_linear_linear}
\end{figure*}

\subsubsection{Linear encoder -- nonlinear decoder}

We make here a parametric study with $N\in\{3,100,500\}$ to ensure the consistency of the results for different sizes of the \Node. For the sake of clarity, we describe below in detail only the case $N=3$; analogous settings are used for the remaining two test cases. \rev{For each case $n=2$}.

The \Node\ taken into account is that in \eqref{eq:lin_ode_test} \rev{with same parameter and initial condition; the only difference is the matrix $\Lambda$ which,} up to the third digit, is equal to 
\begin{equation*}
    \Lambda = \begin{bmatrix}
    	-4.75 & -1.55 & -0.79 \\
    	-1.55 & -4.31 & -1.02 \\
    	-0.79 & -1.02 & -5.2
    \end{bmatrix},
\end{equation*}
with eigenvalues $\lambda_1 = -7.0$, $\lambda_2=-2.93$,  $\lambda_3=-4.34$. In this test case, \rev{despite the only change compared to \eqref{eq:lin_ode_test} being $\Lambda$,} a nonlinear decoder is necessary since the manifold is not flat, see Fig.~\ref{fig:a_stab_manifold}. \rev{The networks' architectures are reported in \Cref{sec:computations}, Tab.~\ref{tab:case_1_architecture}.}
The training is accomplished on a dataset of $100$ samples and the results computed here are computed on the test dataset made of $50$ samples, those represented in Fig.~\ref{fig:a_stab_manifold}. The training is interrupted after 50 epochs when the loss reaches a low value of about $10^{-4}$. After the training, $\nu_{F_n}$ and $L_{F_n}$ are computed numerically on the test dataset, following the procedure presented in \cite{DAmbrosio2021}. Up to the third digit, \rev{they result} to be $\nu_{F_n} = -3.14$, which is similar to that of $F_N$ ($\nu_{F_N} = \lambda_2 = -2.93$), and $L_{F_n} = 10.4$, from which follows the bound on the timestep size for FE: $\Delta t \leq \rev{\Delta t_{\max} = } \frac{-2\nu_{F_n}}{L_{F_n}} = 0.058$, which is about five time smaller to that to assure the stability of the \Node\ if solved directly, $\Delta t \leq \frac{2}{\max_i|\lambda_i|} \approx 2/7.0 \approx 0.28$ and it is about the half to that for assuring the dissipative behavior on the numerical solution of the \Node\ $\Delta t \leq \frac{-2\nu_{F_N}}{L_{F_N}}\approx \frac{-2(-2.93)}{7.0^2}\approx 0.120$.
In the first row of Fig.\ref{fig:a_stab_linear_nonlinear}, we report the components of the reconstructed solution, $\Psi(\ut_n)$, as a function of time, where $\ut_n$ is obtained by solving the \node\ using the FE scheme. The time integration is performed with a timestep size of $0.99\Delta t_{\max}$. Despite the approximations introduced by time discretization and the neural network models, and despite the training on partial trajectories truncated at early times, the solution stabilizes at the equilibrium value close to zero. The training interval, $T = 0.5$, is highlighted in \rev{blue} in \rev{Fig.\ref{fig:a_stab_linear_nonlinear}}. We emphasize that, despite the training on a subset of $\mathcal{M}$, due to the simplicity of the problem at hand, it was possible to successfully extrapolate up to grater times, e.g. $t = 10$ in Fig.~\ref{fig:a_stab_linear_nonlinear}, without any instability. Similar results are obtained for $N=100$ and $N=500$, as depicted in Fig.~\ref{fig:a_stab_linear_nonlinear}.
\begin{figure*}
	\centering
	\subfloat[view 1]{\includegraphics[width=0.35\textwidth]{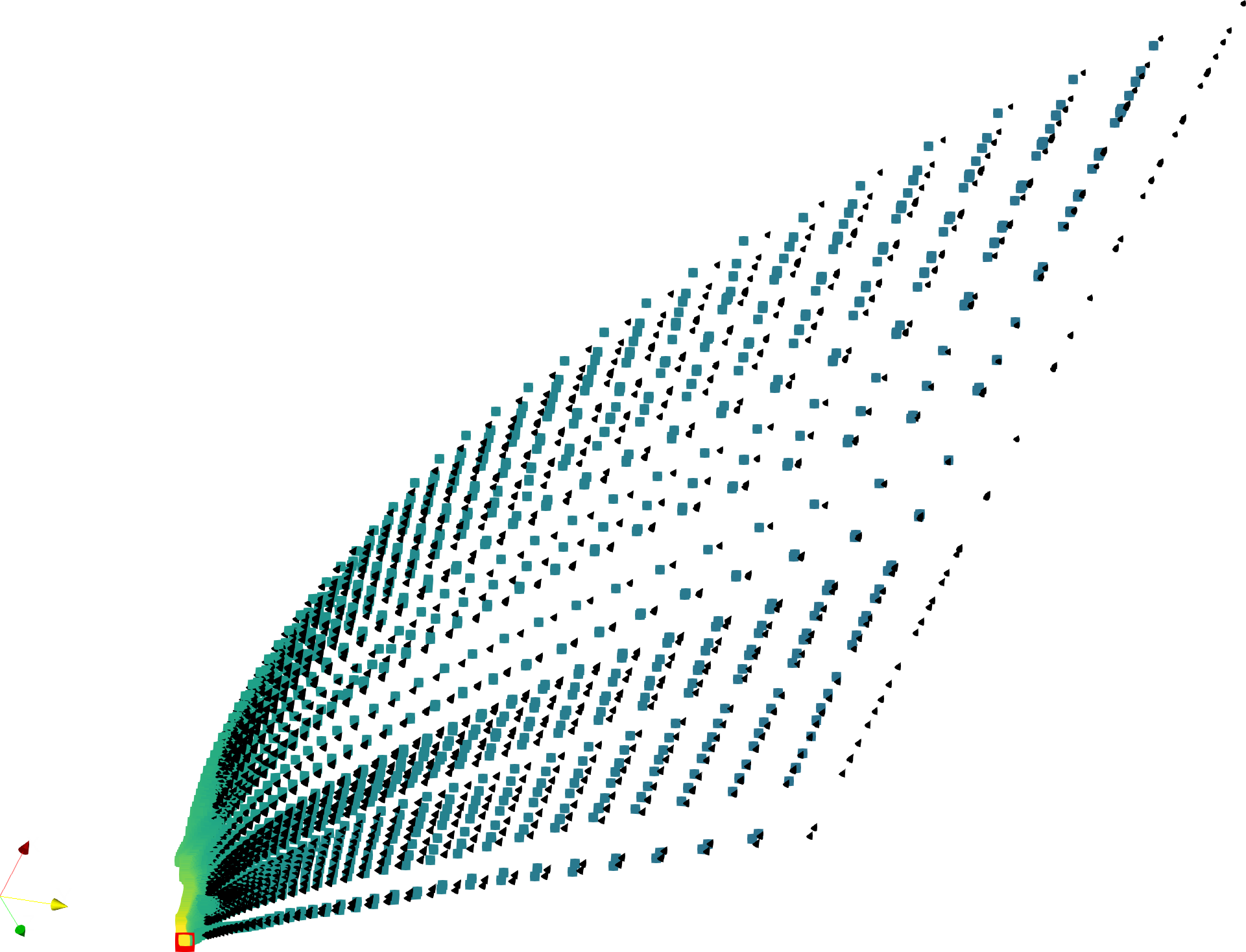}\hspace{0.3cm}}
	\subfloat[view 2]{\includegraphics[width=0.4\textwidth]{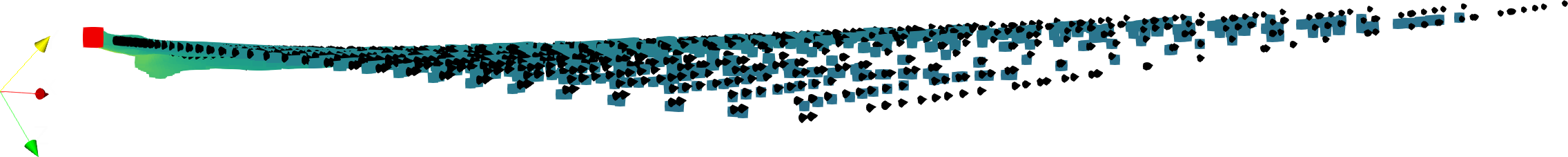}\hspace{0.3cm}}
	\caption{Two perspectives, \textbf{(a)} and \textbf{(b)}, of the curved manifold $\mathcal{M}$. Black cones indicate the data used to train the autoencoder, while the colored points represent the reconstructed $\Psi(\ut_n)$. The color scale (from blue to yellow) is related to the time, solely for the purpose ease the readability. It is worth noting that, in this test case, despite the extrapolation in time, the reconstructed solution still converges to a stable value near the origin (red point). The axes have been translated away from the origin (red point) to improve readability.}
	\label{fig:a_stab_manifold}
\end{figure*}
\begin{figure*}[h]
	\centering
        \vspace{5mm}
        \begin{minipage}{1\textwidth}
            \begin{turn}{90}
    			\hspace{-0.2cm} {\large\textbf{$N = 3$}}
            \end{turn}
            \hspace{0.2cm}
            \begin{minipage}{0.9\textwidth}
        	\includegraphics[width=0.29\textwidth]{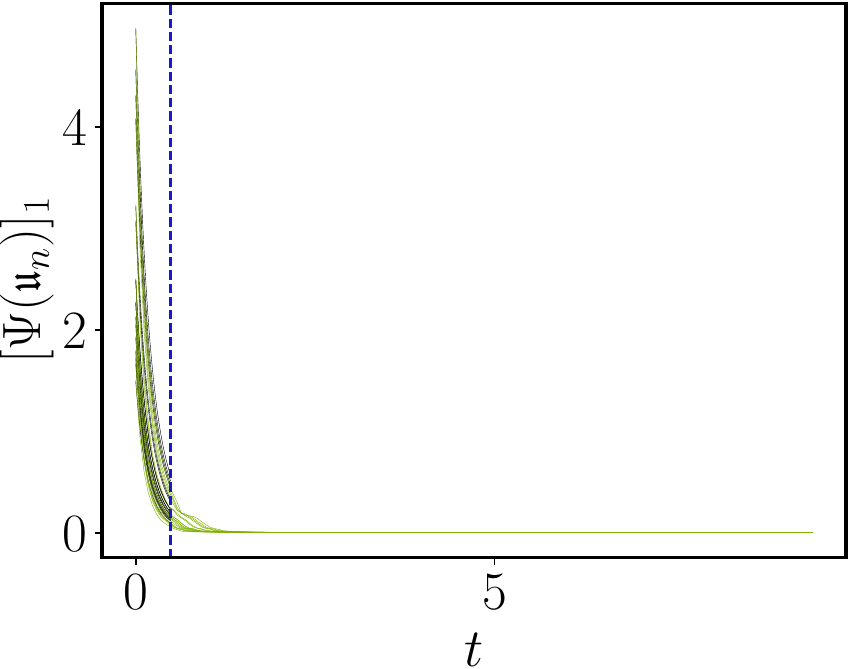}\hspace{0.3cm}
        	\includegraphics[width=0.29\textwidth]{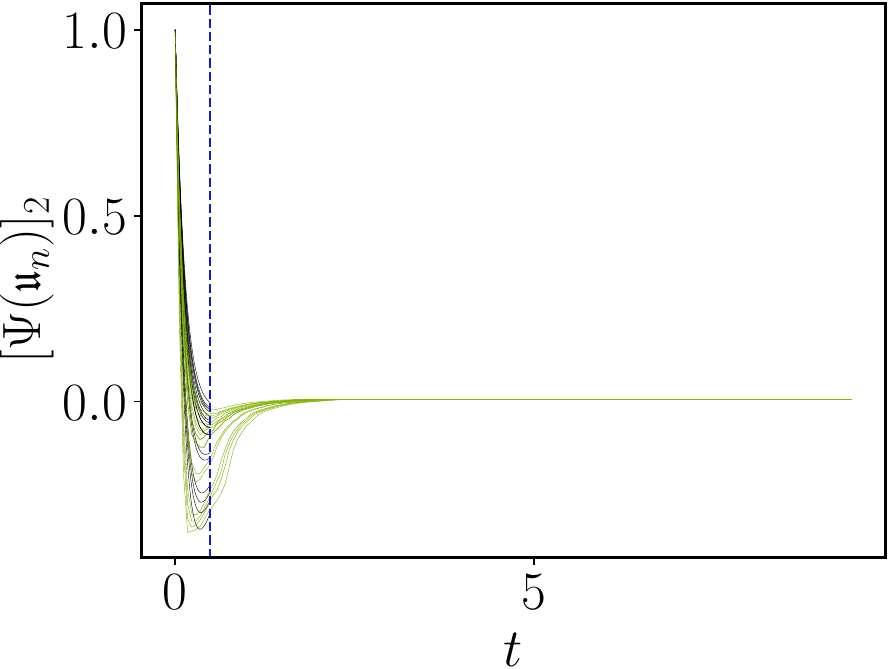}\hspace{0.3cm}
        	\includegraphics[width=0.29\textwidth]{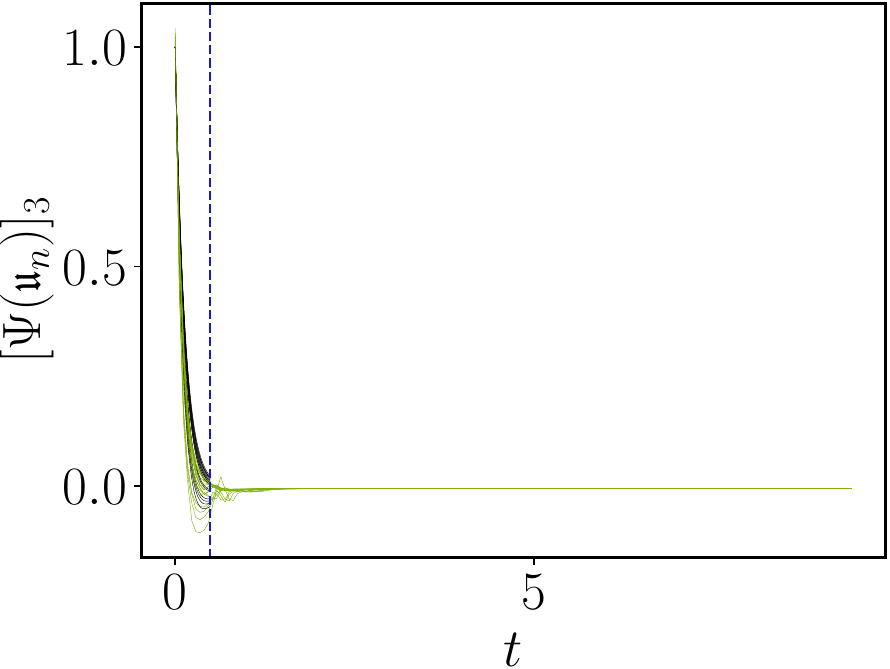}\hspace{0.3cm}
            \end{minipage}
        \end{minipage}

        \vspace{1cm}
        \begin{minipage}{1\textwidth}
            \begin{turn}{90}
                \hspace{-0.8cm} {\large\textbf{$N = 100$}}
            \end{turn}
            \hspace{0.2cm} 
            \begin{minipage}{0.9\textwidth} 
                \includegraphics[width=0.29\textwidth]{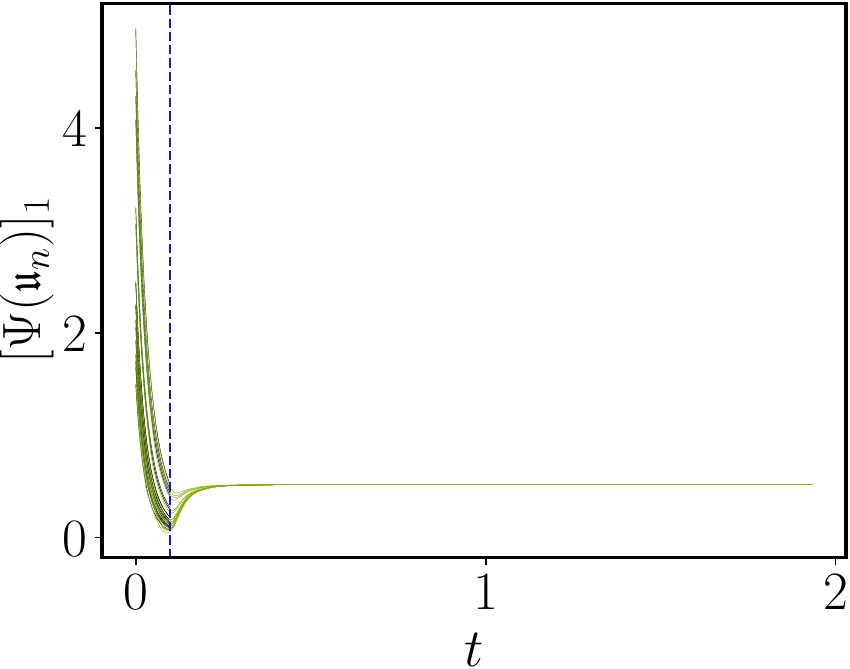}\hspace{0.3cm}
                \includegraphics[width=0.29\textwidth]{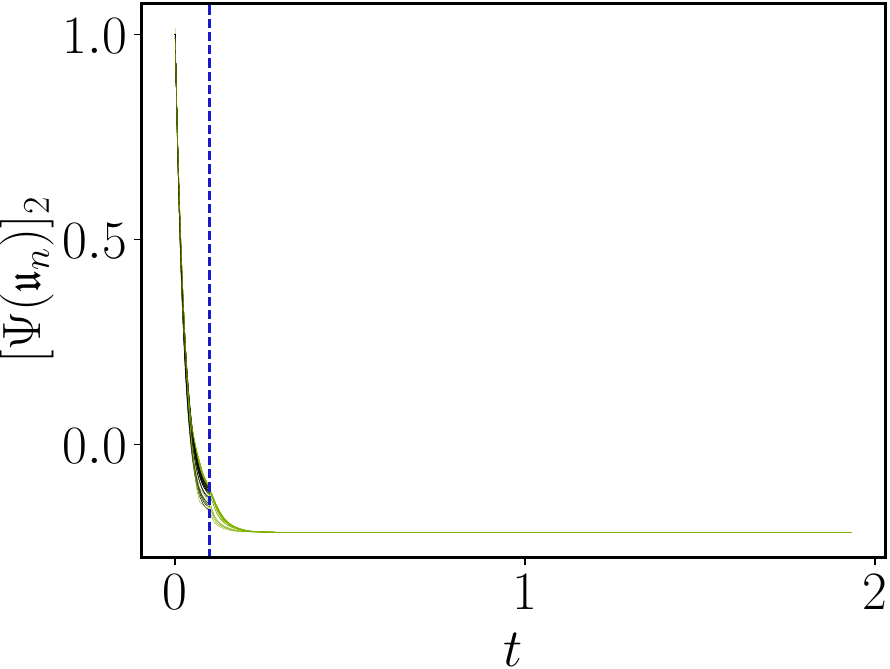}\hspace{0.3cm}
                \includegraphics[width=0.29\textwidth]{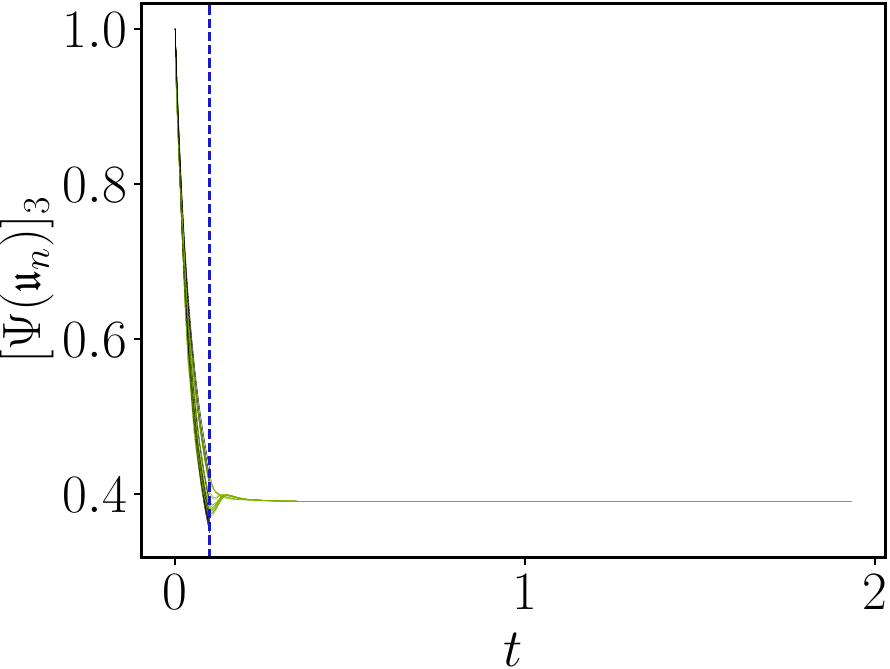}\\[0.3cm] 
                \includegraphics[width=0.29\textwidth]{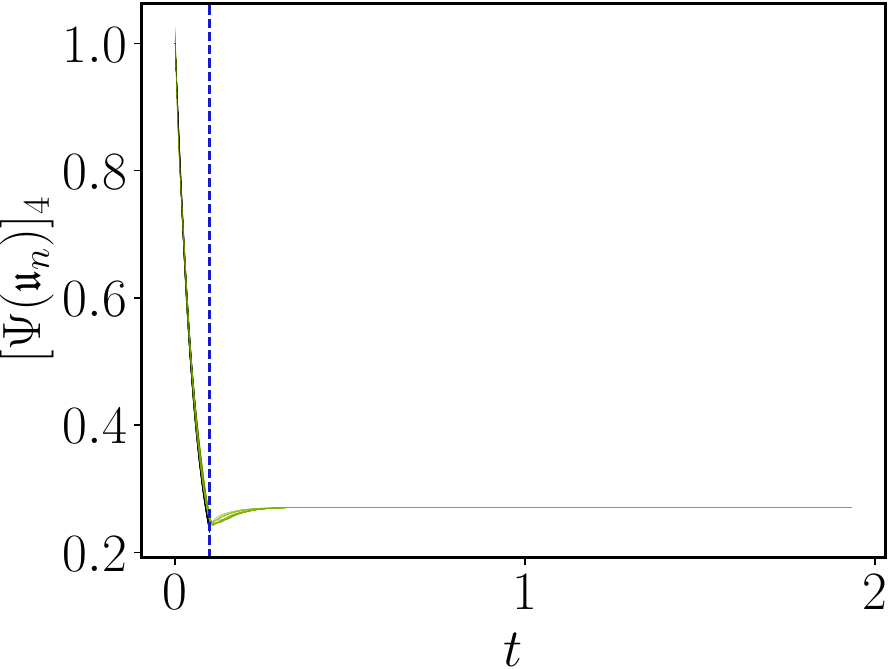}\hspace{0.3cm}
                \includegraphics[width=0.29\textwidth]{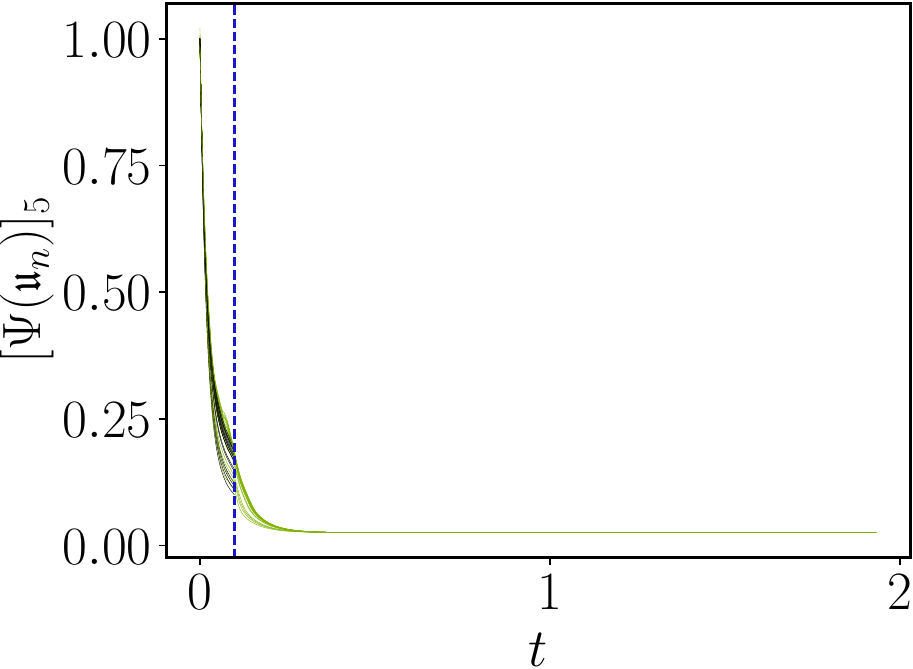}\hspace{0.3cm}
                \includegraphics[width=0.29\textwidth]{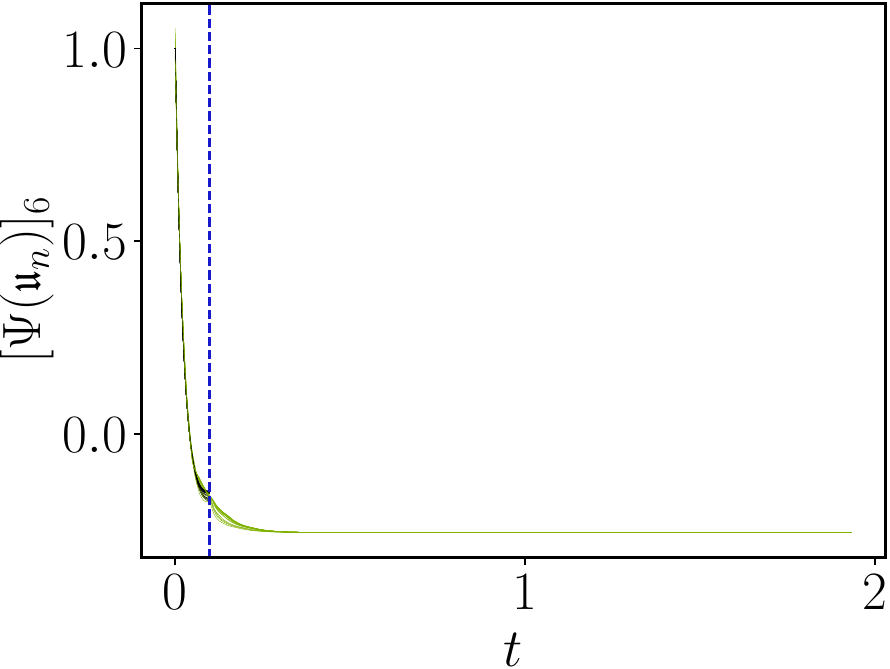}
        \end{minipage}
    \end{minipage}

        \vspace{1cm}
        \begin{minipage}{1\textwidth}
            \begin{turn}{90}
    			\hspace{-0.8cm}  {\large\textbf{$N = 500$}}
            \end{turn}
            \hspace{0.2cm} 
            \begin{minipage}{0.9\textwidth}
        	\includegraphics[width=0.29\textwidth]{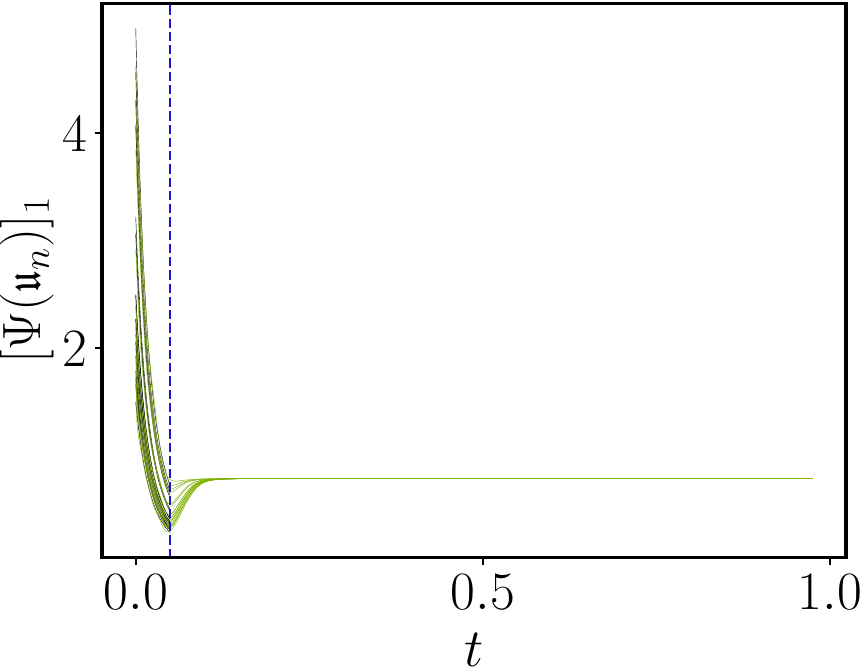}\hspace{0.3cm}
        	\includegraphics[width=0.29\textwidth]{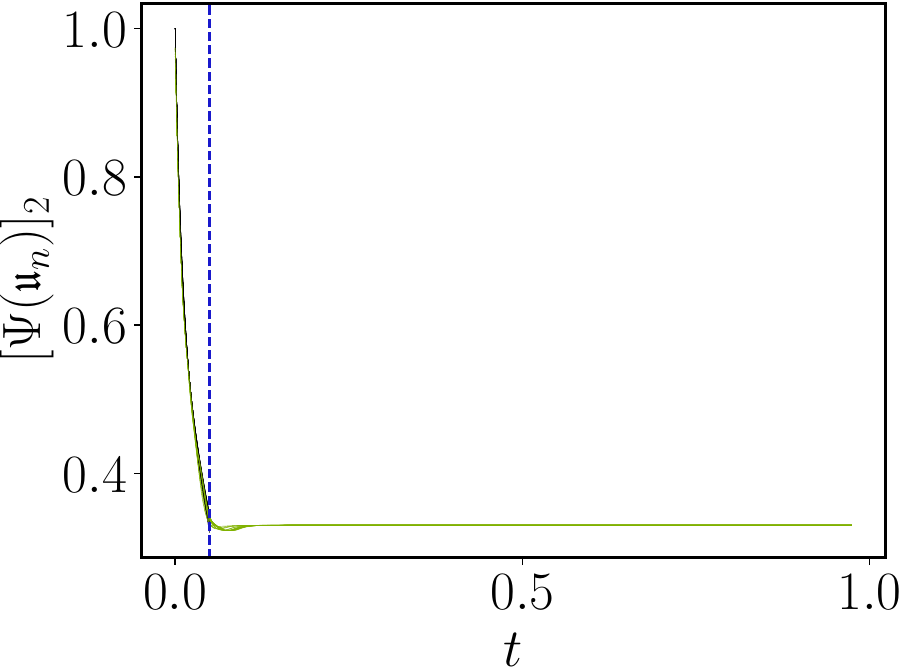}\hspace{0.3cm}
        	\includegraphics[width=0.29\textwidth]{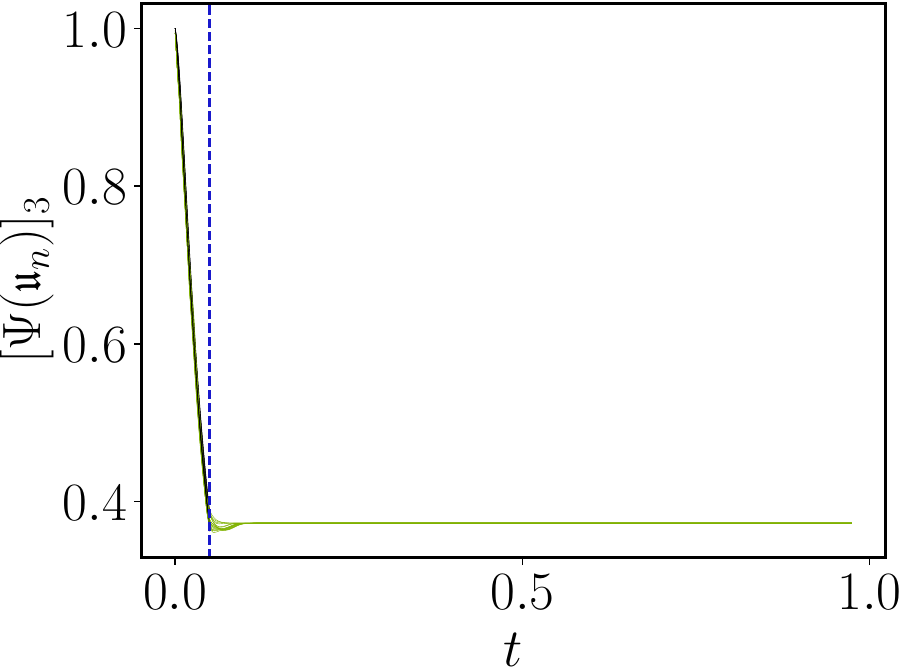}\\[0.3cm] 
                \includegraphics[width=0.29\textwidth]{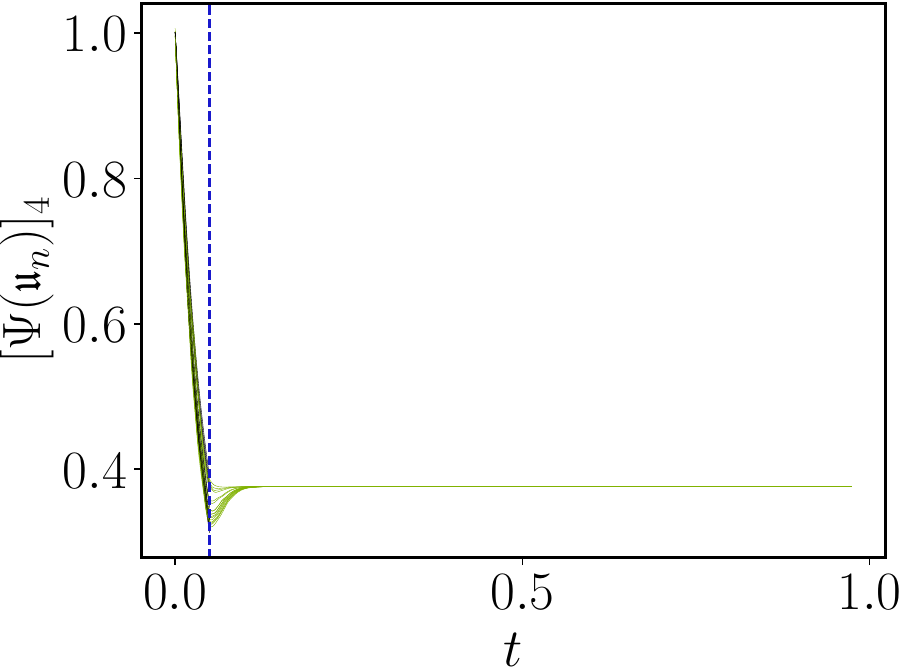}\hspace{0.3cm}
        	\includegraphics[width=0.29\textwidth]{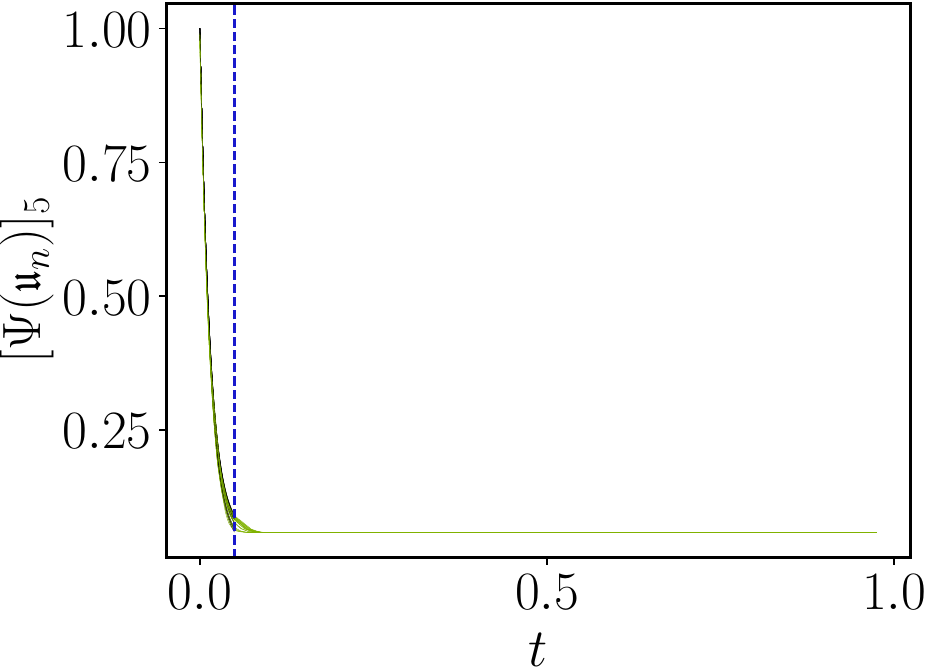}\hspace{0.3cm}
        	\includegraphics[width=0.29\textwidth]{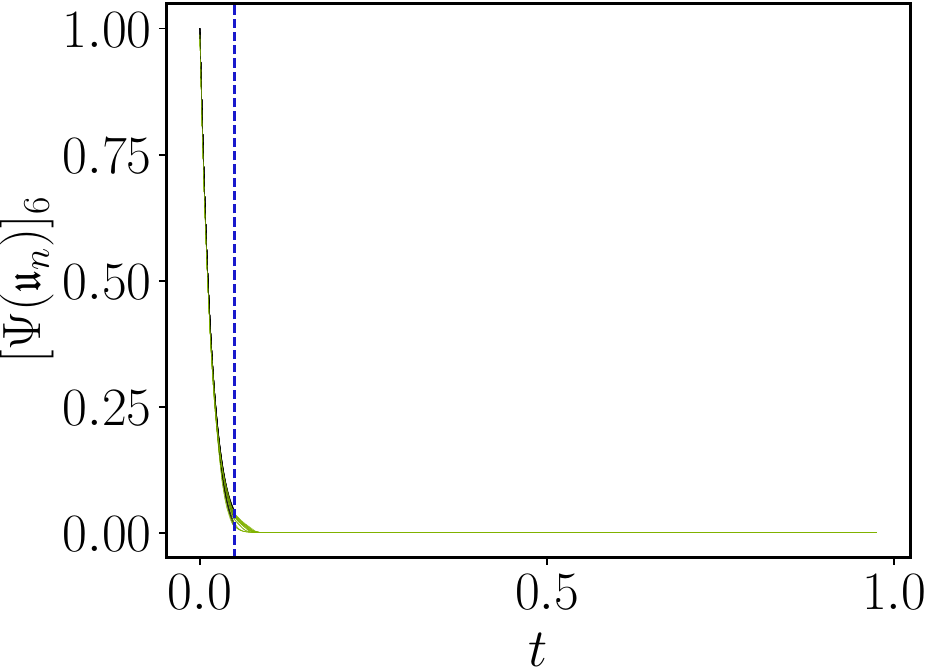}\hspace{0.3cm}
        \end{minipage}
    \end{minipage}
        
	\caption{Non linear case. Components of $\Psi(\ut_n)$, one for each panel. First row of panels corresponds to $N = 3$, second and third row corresponds to $N = 100$, while the last two rows corresponds to $N = 500$. For visualization purposes, only the first six components of $\Psi(\ut_n)$ are displayed. $\ut_n$ is computed using FE with $\Delta t = 0.99 \Delta t_{\max}$. The vertical blue-dashed line is located at the end of the training time. \rev{Note the use of different time ranges for each $N$ for graphical clarity.}}
	\label{fig:a_stab_linear_nonlinear}
\end{figure*}

\clearpage

\subsection{Test 2: Convergence in time}\label{sec:case_2_time_convergence}
In this section, we aim to numerically verify the convergence of the method as $\Delta t \to 0$, after the training of the neural networks. The global error in \eqref{eq:final_error} depends on both the accuracy of the time integration method and the neural networks. Here, our goal is to assess convergence by isolating the contribution of the time integration scheme in \eqref{eq:final_error}.
\rev{The first case, \Cref{sec:test_convergence_analytical}, concerns a manufactured problem which has an analytical solution, whereas the second case, \Cref{sec:test_convergence_generic}, addresses a more general setting.
To ease the readability, we report here the common setup of the two cases. 
}

Since our goal is to analyze temporal convergence, in light of \eqref{eq:final_error}, the approximation error of the neural networks must be reduced as much as possible to isolate the term $L_{\Psi}C_1\Delta t^p$. To this end, we employ deep neural networks with multiple layers and a large number of trainable parameters \rev{for both cases}. \rev{Their specific architecture will be detailed in the following.}

The width of the layers is chosen in agreement to what reported in \Cref{sec:autoencoders} and the guidelines provided in \cite{Li2023, Cai2022, Park2020}. In order to ensure continuity to $F_n$, the activation function used in the encoder are ELU $\in \mathcal{C}^1$. We investigate three training scenarios:
\begin{enumerate}
	\item \textit{Exact rhs}: $F_n$ is computed directly from its definition, $J_{\Psi'}F_N(\Psi(\cdot))$, yielding the most accurate result since only the autoencoder is approximated by neural networks.
	\item \textit{Semi data-driven}: $F_n$ is approximated by $N_{F_n}$ using augmented data and trained according to the strategy outlined in \Cref{sec:semi_datadriven}.
	\item \textit{Fully data-driven}: $F_n$ is approximated by $N_{F_n}$ which is trained following the methodology described in \Cref{sec:fully_datadriven}. The training is accomplished with the FE difference operator, see \eqref{eq:loss_2_fully_datadriven}. 
\end{enumerate}
After training, the quality of the DRE is assessed by computing the average relative error over the test dataset:
\begin{equation*}
	e_{\text{multi}} = \left(\frac{1}{n_{\text{test}}}\sum_{i=1}^{n_{\text{test}}}\sum_{j=1}^{N_{\text{time}}}\frac{\|u_{N,i}^j - N_{\Psi}(\uut_{N,i}^j)\|_2^2}{\|u_{N,i}^j\|_2^2}\right)^{1/2},
\end{equation*}
where $n_{\text{test}}$ is fthe number of data in the test dataset. 

\subsubsection{Analytical case - manufactured solution}\label{sec:test_convergence_analytical}

\rev{
In order to properly isolate the error associated with time discretization, we introduce a manufactured solution providing analytical expressions for \Node\ and \node. In particular, we derive explicit analytical expressions for $\Psi'$, $\Psi$, and $F_n$.
Nevertheless, we also report results in which the aforementioned quantities are approximated by properly trained neural networks, namely $N_{\Psi'}$, $N_\Psi$, and $N_{F_n}$, whose architectures are reported in \Cref{sec:architectures}, Tab.~\ref{tab:case_2_convergence}.
This strategy allows for a clear separation between the error contribution arising from time discretization and the additional error introduced by the approximation capabilities of the neural networks. This is achieved by comparing the results obtained using the analytical expressions of $\Psi'$, $\Psi$, and $F_n$ with those obtained by replacing them with $N_{\Psi'}$, $N_\Psi$, and $N_{F_n}$, respectively. 
In the following, we set $W \in \mathbb{R}^{N\times n}$ to be a matrix whose entries are randomly sampled from a uniform distribution over $[0,1)$. Its pseudoinverse is denoted by $W^\dagger$.  Here, $N = 3$ and $n = 2$. 
We denote by $w_1$  the first column of $W$.
We then consider the following \Node:
\begin{align}\label{eq:man_Node_conv}
\begin{aligned}
    &\dot u_N = \mu u_N\log(u_N),  \quad \text{for } t\in (0,T], \\
    &u_N(0) = \exp\left(w_1\right)
\end{aligned}
\end{align}
where $\log(\cdot)$ and $\exp(\cdot)$ are to be understood component-wise; $T = 1$ is the final time. 
The decoder is defined as $\Psi(u_n) = I_{N\times N}\,\exp(W \, u_n )$ while the encoder is defined as $\Psi'(u_N) = W^\dagger\log(I_{N\times N}\, u_N )$. Note that here the biases of the layers are null. 
%
%
We highlight that, in this setting, non-standard activation functions, namely $\exp$ and $\log$, are employed. This choice facilitates the analytical expressions of the quantities involved in this test case, which is the primary focus of the present analysis.

From the above definitions, we can obtain the expression for the \node: 
\begin{align*}
\begin{aligned}
     &[\dot u_n]_i = [J_{\Psi'}F_N]_i =\sum_k \left(\sum_j W_{ij}^\dagger\frac{1}{[u_N]_j}\delta_{ik}\right) [F_N]_k, \\
     &u_n(0) = [1,0]^\top. 
\end{aligned}
\end{align*}
Given the above definitions, the exact $N$- and $n$-solutions are given, respectively, by
$u_N = \exp\left(\exp(\mu t)w_1\right)$ and $u_n = \exp(\mu t)[1,0]^\top$. 
Note that the vanishing second component of $u_n$ arises from the choice of selecting only the first column of $W$, namely $w_1$, in the definition of $u_N$ and of the associated \Node. Moreover, we observe that, due to the expression of $u_N$, the resulting manifold is nonlinear.

The parameter $\mu$ is sampled 200 times from a uniform distribution, $\mu \sim \mathcal{U}(0.1, 2)$. Among these samples, 100 are used to form the training dataset, 50 the validation dataset (used to compare the loss during the training), and 50 the test dataset (used to assess the final accuracy of the model after training). \\
The neural networks, whose architectures are reported in Tab.~\ref{tab:case_2_convergence}, are trained using the FE scheme with timestep $\Delta t_\text{train} = 0.01$. During the online phase, the timestep size $\Delta t$ can be varied.

In Fig.~\ref{fig:convergence}, we report the values of $e_\text{multi}$ for both the case employing analytical expressions for $\Psi'$, $\Psi$, and $F_n$ (panel (a)) and the case in which the reduced model is constructed using $N_{\Psi'}$, $N_\Psi$, and $N_{F_n}$ (panel (b)). \\
The values of $e_{\text{multi}}$ are shown for the three aforementioned strategies, namely ``Exact rhs'', ``Semi data-driven'', and ``Fully data-driven''. \\
The trained neural networks are employed in combination with different time-integration schemes, possibly different from the one used during training. Specifically, we consider FE (black and yellow lines), the linear multistep Adams–Bashforth 2 method (AB2, green lines), and, for completeness, the multi-stage Runge–Kutta 5 method (RK5, purple lines). The training timestep size, $\Delta t_\text{train}$, is also indicated by a vertical line.

For the analytical case (panel (a)), we observe that the convergence order is the expected one and that the error steadily decreases, reaching very small values for RK5, where the observed stagnation can be attributed to truncation errors. Note that, in panel (a), due to the use of exact $\Psi'$, $\Psi$, and $F_n$, some scenarios collapse onto a single line.

Conversely, when these quantities are approximated using neural networks (panel (b)), the error exhibits stagnation to higher values than that of panel (a), indicating that the approximation error introduced by the neural networks becomes dominant. It is important to note that the semi data-driven training strategy allows for improving the accuracy of the method by reducing the timestep during the online phase to values \textit{smaller} than those used during training, therefore enabling temporal over-resolution.

Within the fully data-driven strategy, instead, no improvement is observed when decreasing the timestep. Moreover, the error behavior displays a local minimum at the training timestep value, which can be explained by the fact that $N_{F_n}$ approximates $F_n$ together with correction terms, some of which depend on the specific $\Delta t_\text{train}$, see \eqref{eq:what_nn_learns_1}.

\begin{figure}[h]  
    \subfloat[Manufactured model - Exact]{\includegraphics[width=0.3\linewidth]{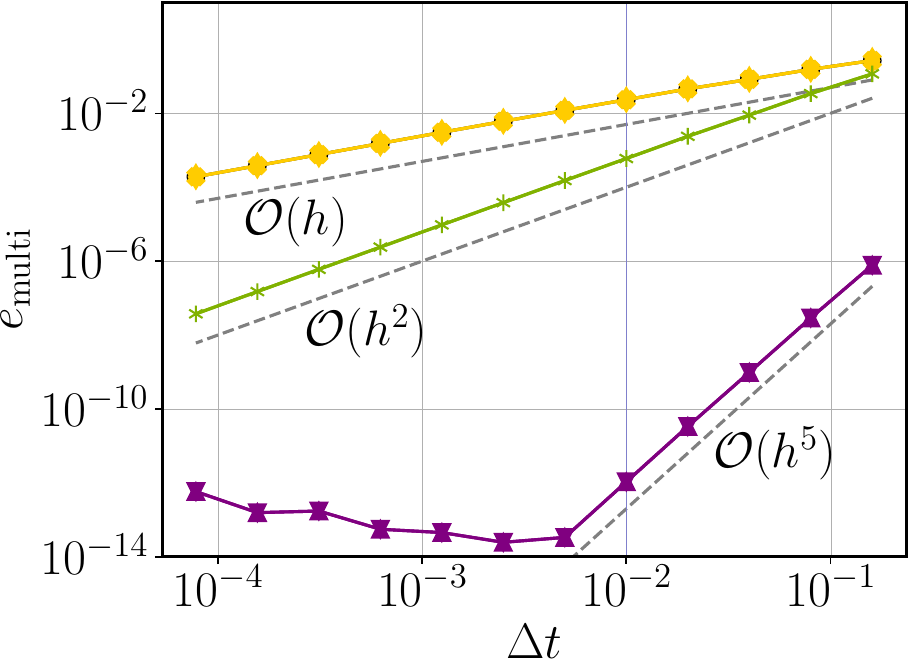}}\hspace{2mm}
    \subfloat[Manufactured model - NN approx.]{\includegraphics[width=0.3\linewidth]{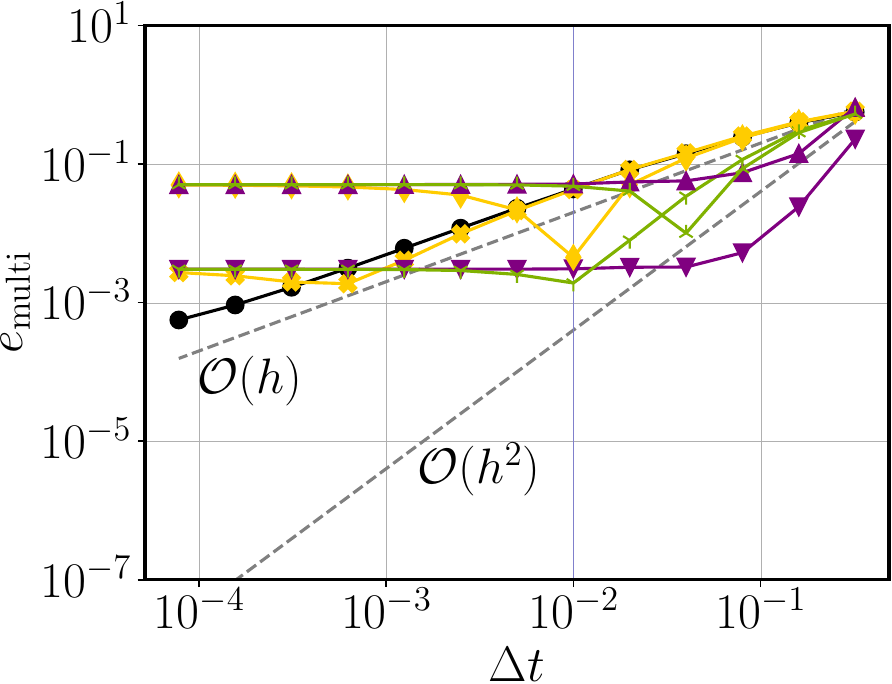}}\hspace{8mm}
    \subfloat[SIR - NN approx.]{\includegraphics[width=0.3\linewidth]{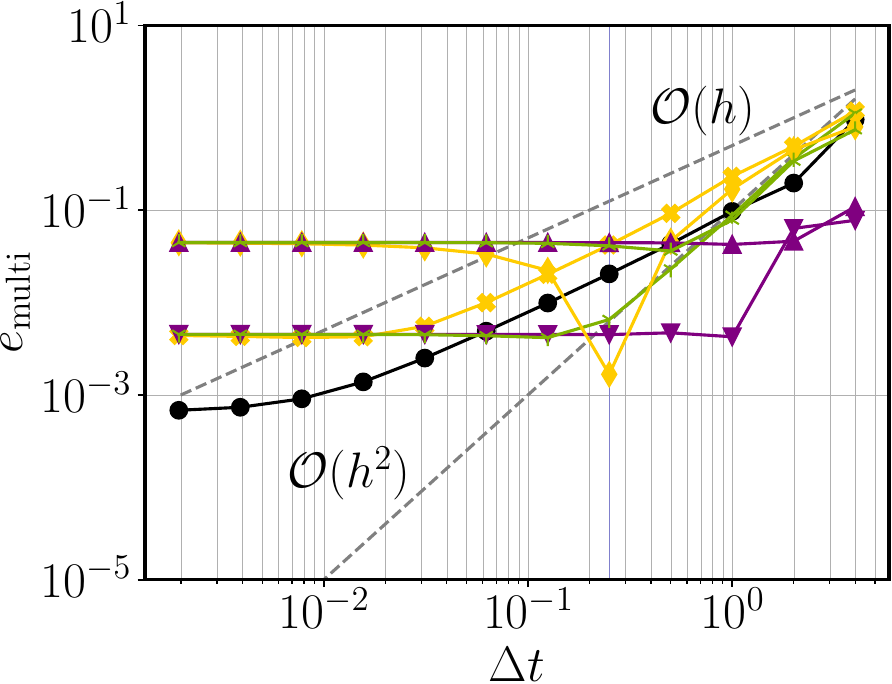}}\\
    {\phantom{0} \\}
    \begin{minipage}{1\textwidth}
    	\hspace{15mm}
	    \adjustbox{valign=t}{\includegraphics[width=0.15\linewidth]{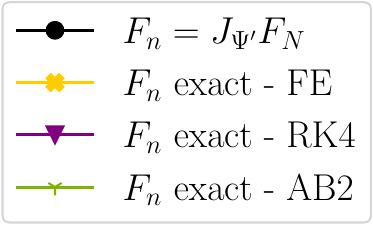}} \hspace{45mm}
	    \adjustbox{valign=t}{\includegraphics[width=0.33\linewidth]{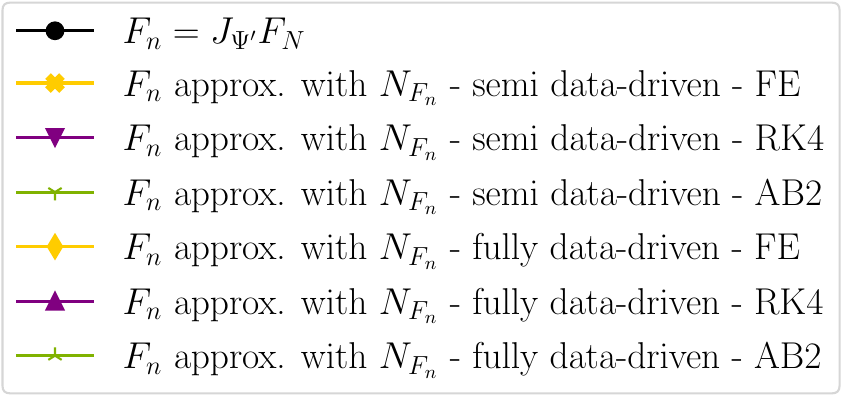}}
    \end{minipage}
    \caption{\rev{Time convergence for different strategies and integration schemes is shown. (a) the \Node  is defined in \eqref{eq:man_Node_conv}. The reduced order model is constructed using analytical $\Psi'$, $\Psi$, $F_n$. (b) the \Node  is defined in \eqref{eq:man_Node_conv}. The reduced order model is constructed using trained $N_{\Psi'}$, $N_\Psi$, $N_{F_n}$. (c) SIR model, see \eqref{eq:sir}. The observed orders of convergence match the expected ones up to a plateau, which arises due to the approximation error introduced by the neural networks.}}
    \label{fig:convergence}
\end{figure}

}

\subsubsection{Generic case - SIR}\label{sec:test_convergence_generic}

\rev{To strengthen the discussion on time convergence, we present here a case for which the analytical expressions of $\Psi'$, $\Psi$, and $F_n$ are not available.}
As a model case, we consider the SIR system \eqref{eq:sir}, with $\beta$ treated as a parameter. The main settings are as follows: $\beta \sim \mathcal{U}(0.5, 2.5)$, $\gamma = 0.5$, $N = 100$, and initial condition $u_{N0} = [S_0, I_0, R_0]^T = [90, 10, 0]^T$. The full dataset comprises $300$ trajectories in the time range $[0, T]$, $T=20$, each corresponding to a different realization of $\beta$. These data are partitioned into $200$ samples for training, $50$ for validation , and $50$ for testing.
Each trajectory is computed using the Runge–Kutta 4(5) method \cite{Butcher2016}, as implemented in Scipy \cite{Virtanen2020}, with adaptive time-stepping and stringent tolerances: $\text{atol} = 10^{-16}$ and $\text{rtol} = 10^{-12}$. These settings ensure that the time discretization error of the \Node\ can be neglected.

The neural networks' architectures are summarized in \Cref{sec:architectures}.
The training is carried out adopting the FE scheme for all the three aforementioned options for $2000$ epochs using a minibatch size of $32$. The learning rate is set to $10^{-3}$ for $\text{epochs} < 500$, reduced to $10^{-4}$ for $500 \leq \text{epochs} < 1500$, and further to $10^{-5}$ thereafter. Across all strategies, training decreases the loss function by approximately $6$–$7$ orders of magnitude relative to its initial value, typically reaching a minimum around $10^{-3}$. Given that the loss function is the mean squared error, this leads to an estimation of the relative error in $u_N$ and $F_n$ of the order of magnitude of $10^{-3}$. Training is performed on data sampled with timestep size $\Delta t_{\text{train}} = 10^{-3}$. \rev{As in the previous case, during the online phase, the timestep size $\Delta t$ can be varied.}

\rev{In Fig.~\ref{fig:convergence} panel (c), we report the values of $e_{\text{multi}}$ for the three strategies ``Exact rhs'', ``Semi data-driven'', and ``Fully data-driven'' as previously mentioned. The trained neural networks are used in combination of different schemes, possibly different from that used during the training: FE, AB2, and, for completeness, RK5. The training timestep size, $\Delta t_{\text{train}}$, is also indicated with a vertical line.}
First, we observe that with strategy 1 (black line), the error exhibits first-order convergence up to a stagnation point. This plateau can be explained with a balance between the approximation error of the neural network and the time discretization error, as predicted by \eqref{eq:final_error}. Looking at the yellow lines (FE), strategy 3 displays a minimum error near the training timestep size. This minimum is also close to the lowest value across all configurations. Reducing $\Delta t$ further leads to an increase in the error. This behavior is expected, as $N_{F_n}$ approximates $F_n$ together with correction terms, some of which depend on the specific $\Delta t_{\text{train}}$, see \eqref{eq:what_nn_learns_1}. This phenomenon disappears when adopting strategy 2. Indeed, its corresponding line shows a higher error at $\Delta t_{\text{train}}$ but continues to decrease with order 1 until reaching a stagnation point for smaller timestep sizes.

We now consider the green lines (AB2). As expected, we observe second-order convergence, and the minimum error is comparable to that obtained with FE. However, since $N_{F_n}$ is trained to minimize \eqref{eq:loss_2_fully_datadriven} using the FE difference operator, the error minimum valley around $\Delta t_{\text{train}}$ does not appear.

\rev{Regarding the RK5 scheme (purple lines), we observe that, due to its higher order of convergence, the overall error remains dominated by the neural network approximation even for relatively large timestep sizes. Moreover, the stagnation values coincides with those of the other methods, namely FE and AB2.}

We emphasize that $e_{\text{multi}}$ is computed on the test dataset, so unseen data,  thus confirming the stability of the results, the attainment of the expected convergence order, and the generalization capability of the trained models.
%
%
%

\subsection{Test 3: Chemical reactions kinetic}\label{sec:case_3_chemistry}

In this case, we investigate a generic transient system of chemical reactions \cite{Steefel1994, Fumagalli2021}. This test case serves to demonstrate the application of the DRE to a general nonlinear ODE. The main challenges arise from the strong nonlinearities of the \Node\ and from the larger value of $N$ compared to the previously analyzed cases.

We consider a total of $n_r = 6$ generic forward and backward reactions involving chemical species $A_i$, for $i = 1, \ldots, n_s$, with $n_s = 19$ denoting the total number of species. A generic reaction takes the following form:
\begin{equation*}
	\sum_{i\in \overline{\mathcal{K}}} s_{ij}A_i \rightarrow \sum_{k\in \mathcal{K}} -s_{kj}A_k, \quad \forall j=1,\ldots, n_r,
\end{equation*}
Here, $\overline{\mathcal{K}}$ and $\mathcal{K}$ represent the sets of indices corresponding to reactants and products, respectively, and $s_{ij}$ denote the stoichiometric coefficients.
We denote by $u_N:[0, T] \to \mathbb{R}^{n_s}$ the concentration vector of the chemical species, by $r:\mathbb{R}^{n_r} \to \mathbb{R}^{n_r}$ the reaction rate vector, and by $S \in \mathbb{R}^{n_s \times n_r}$ the stoichiometric matrix, whose entries are the coefficients $s_{ij}$. We choose the following $S$:
{\small
\begin{equation*}
	S^\top = 
	\begin{bmatrix}
		1 & 0  & -1 & -1 & 0  & 0  & 0  & 0  & 1  & 0 & -1 & 0  & 0  & 0  & 0  & 1  & -1 & 1  & 0 \\
	   -1 & 1  & 0  & 0  & -1 & 0  & 0  & 0  & 0  & 1 & 0  & -1 & 0  & 0  & 0  & 0  & 1  & -1 & 1 \\
		0 & -1 & 1  & 1  & 0  & -1 & 0  & 0  & 0  & 0 & 1  & 0  & -1 & 0  & 0  & 0  & 0  & 0  & 0 \\
		1 & 0  & -1 & 0  & 1  & 0  & -1 & 0  & 0  & 0 & 0  & 1  & 0  & -1 & 0  & 0  & 0  & 0  & 0 \\
	   -1 & 1  & 0  & 0  & 0  & 1  & 0  & -1 & 0  & 0 & 0  & 0  & 1  & 0  & -1 & 0  & 0  & 0  & 0 \\
		0 & -1 & 1  & 0  & 0  & 0  & 1  & 0  & -1 & 0 & 0  & 0  & 0  & 1  & 0  & -1 & 0  & 1  & -1\\ 
	\end{bmatrix}.
\end{equation*}
}

The time evolution of the concentrations is governed by the following \Node:
\begin{align}\label{eq:ode_chemistry}
\begin{aligned}
	&\dot u_N = Sr(u_N), \\
	&u_N(0) = [0.79, \mu_1, 0.92, 0.41, 0.32, 0.39, 0.68, 0.89, 0.02, 0.28, 0.58, 0.94, 0.1, 0.9, 0.83, 0.72, 0.72, 0.50, 0.84].
\end{aligned}
\end{align}
The first parameter of this problem is $\mu_1\sim\mathcal{U}(0.74, 0.94)$, which influences the initial condition of the second chemical species.
For the reaction rates, we employ the following model:
\begin{equation*}
	r_j = k_j \prod_{i=1}^{n_s} [u_N]_i^{|S_{ij}| \cdot \mathbf{1}_{S_{ij} < 0}} \quad \text{for } j = 1, \dots, n_r,
\end{equation*}
where $\mathbf{1}_{S_{ij} < 0}$ denotes the indicator function, which equals $1$ if $S_{ij} < 0$ and $0$ otherwise, and $k_j$ is the forward rate constant for the $j$-th reaction. The latter is typically subject to uncertainties in real experimental settings. To account for this, we treat the second component of the vector $k$ as the second parameter of the problem, denoted by $\mu_2$, and assume it follows the distribution $\mu_2 = k_2 \sim \mathcal{U}(5,11)$. We assume the parameters to be independently distributed. The forward rate vector is given by $k = [5, \mu_2, 5, 5, 5, 5]$.
 
The DRE is trained using a fully data-driven strategy, relying on the Adams--Bashforth 2 integration scheme for the loss contribution \eqref{eq:loss_2_fully_datadriven}.

We are dealing with a conservative system, which means that the following equation should be satisfied at all times
\begin{equation*}
    \sum_{k=1}^{N} [u_{N}(t)]_k = \sum_{k=1}^{N} [u_{N0}]_k, \quad \forall t \in [0, T].  
\end{equation*}
It is therefore relevant, in this scenario, that the reduced order model preserves this property. To this aim, a simple strategy is to modify the loss function by adding a contribution whose minimization enforces conservation:
\begin{equation*}
    \mathscr{N}_{ij}^\text{cons} = \rev{\left\| \sum_{k=1}^{N=19} [u_{N,i}^j]_k - [N_{\Psi}(\uut_{n,i}^j)]_k \right\|_2^2}.
\end{equation*}
The training dataset consists of $200$ samples uniformly distributed in the parameter space, with $100$ time steps per sample. The validation and test datasets each contain $10\%$ of the number of samples used for training.

After training, the quality of the DRE is assessed by computing the time-dependent average relative error, $e_\text{ave, rel}(t)$, and the average of discrepancy of the total concentration, $e_\text{ave, con}(t)$:
\begin{align}\label{eq:case_chemistry_errs}
\begin{aligned}
	&e_\text{ave, rel}(t_j) =  \left(\frac{1}{n_{\text{test}}}\sum_{i=1}^{n_{\text{test}}}\frac{\|u_{N,i}(t_j) - N_{\Psi}(\uut_{n,i}^j)\|_2^2}{\|u_{N,i}\|_2^2}\right)^{1/2}, \\
	&e_\text{ave, con}(t_j) =  \frac{1}{n_{\text{test}}}\sum_{i=1}^{n_{\text{test}}}\sum_{k=1}^{N=19} [u_{N,i}(t_j)]_k - [N_{\Psi}(\uut_{n,i}^j)]_k,
\end{aligned}
\end{align}
where $j=0,1,\ldots$ and $t_j = j\Delta t$.

\subsubsection{Results}
Fig.~\ref{fig:case_4_uN_err} shows the reference concentrations of each component,$[u_N]_i$, $i=1,\ldots,n_s$, obtained by solving \eqref{eq:ode_chemistry} with the Runge-Kutta 4(5) method implemented in SciPy \cite{Virtanen2020}, using strict tolerances ($\text{atol} = 10^{-16}$, $\text{rtol} = 10^{-12}$), represented by solid black lines. The reconstructed solution, $N_{\Psi}(\uut_n)$, \rev{which is computed using a timestep size $\Delta t = 0.025$,} is shown with dashed green lines. The reduced model is also queried at future, unseen times, \rev{$3 < t < 8.9$}, although training was performed only up to $t = 3$. \rev{We observe that both the conservation error and the relative error increase steadily at later times. Nevertheless, despite the evaluation on unseen data, the DRE provides reliable results for this test case.}

In the right panel of the same figure, the average relative error, $e_{\text{ave, rel}}$, is also displayed to provide a quantitative assessment of the DRE's accuracy. In the training time range, the error is low, reaching a peak of less than $2\%$.This holds even during the initial phase, where the concentrations vary rapidly. For unseen times, the error increases with a fairly linear trend reaching a maximum of about $4\%$.
In the right panel of Fig.~\ref{fig:case_4_uN_err} also $e_\text{ave, con}$ is represented. We observe a consistent overestimation of the total concentration. However, this discrepancy is small, as the total concentration is approximately $11.67$, and the deviation affects only the third significant digit.
\begin{figure*}[h]
	\centering
	\begin{minipage}{0.85\textwidth}
		\centering
		\subfloat[chemical species from 1 to 10]{\includegraphics[width=0.40\textwidth]{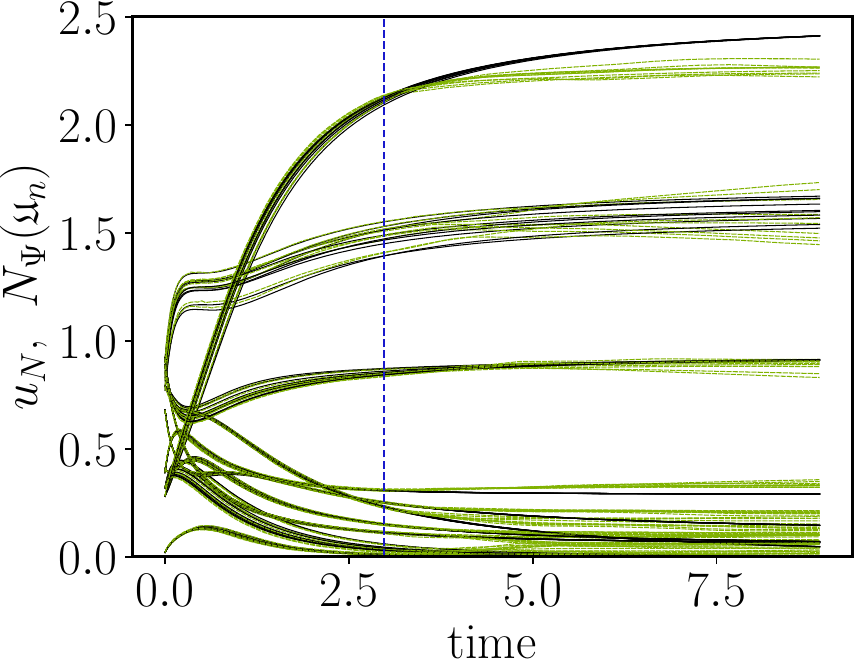}}
		\hspace{4mm}
		\subfloat[chemical species from 11 to 19]{\includegraphics[width=0.56\textwidth]{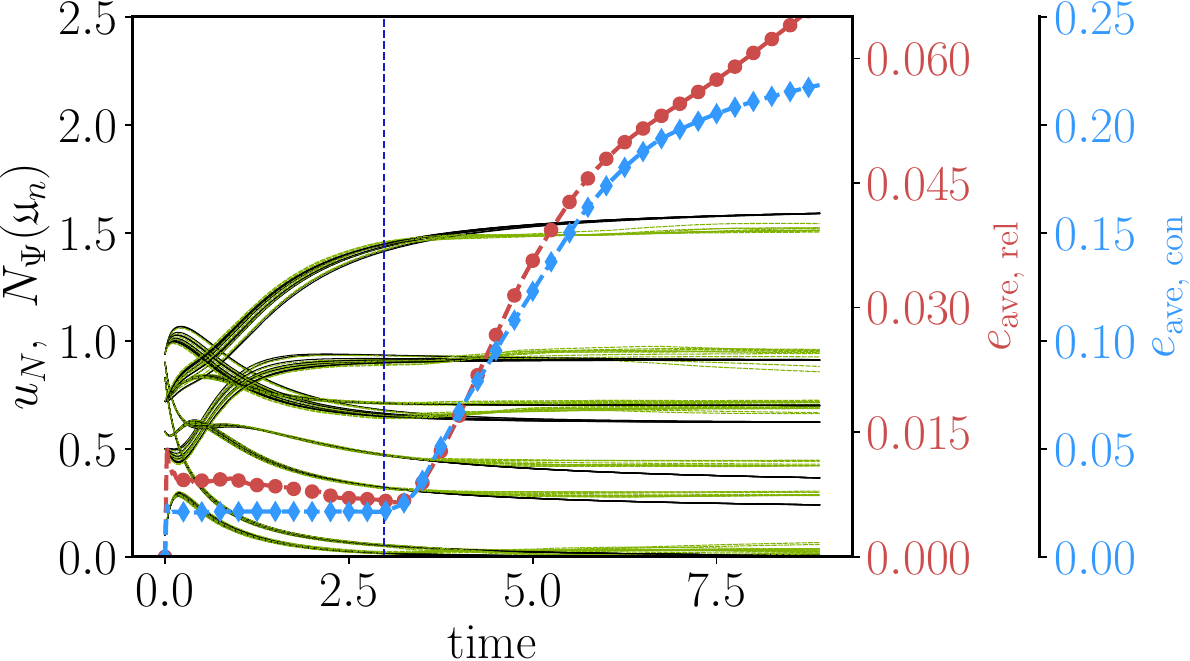}}\hfill
	\end{minipage}
	\begin{minipage}{0.14\textwidth}
		\raggedright	
		\vspace{-3cm}
		\hspace{1cm}
		\includegraphics[width=1\textwidth]{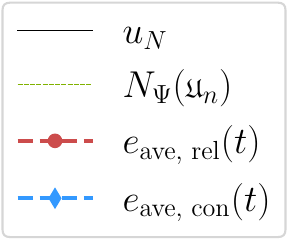}
	\end{minipage}
	\caption{Time evolution of the components of $u_N$. Solid black lines represent the reference solution obtained by accurately solving the \Node; dashed green lines show the reconstructed solution $N_{\Psi}(\uut_n)$. The two solutions exhibit good agreement visually. The training data, and hence the learned manifold, are restricted to the interval $t \leq 10$, which is marked by a vertical yellow dashed line. The red and blue dashed lines indicate the relative and conservation errors, as defined in \eqref{eq:case_chemistry_errs}.}
	\label{fig:case_4_uN_err}
\end{figure*}
Fig.\ref{fig:latent_set} shows the set ${ \uut_n(t;\mu) }$ obtained for $\mu$ belonging in the training dataset values. The set is partitioned into two subsets with different color scales: one corresponding to the $n$-solutions the training time interval $t \in [0, 3]$, called $U_t$, and the other corresponding to $n$-solution obtained for future times $t \in (3,6]$, called $U_e$. The two subsets are disjoint, indicating that $N_{F_n}$ is operating on unseen input during the training, those for future times. This highlights the challenge of time extrapolation. Indeed, $N_{F_n}$ is trained to approximate $F_n: U_t\to \mathbb{R}^n$ but subsequently evaluated on a larger domain $U_e \supset U_t$.
Nonetheless, in this particular case, the extrapolated predictions remain reliable, as also shown in Fig.\ref{fig:case_4_uN_err}. A more accurate extrapolation in time would be feasible if $u_n$ at future times remained within the training region as, for example, in the case where the solution of the \Node\ is periodic, so a training using data sampled in one period of time would be sufficient to characterize virtually all the values of $u_n$.
\begin{figure*}[h]
	\centering
	\includegraphics[width=0.4\textwidth]{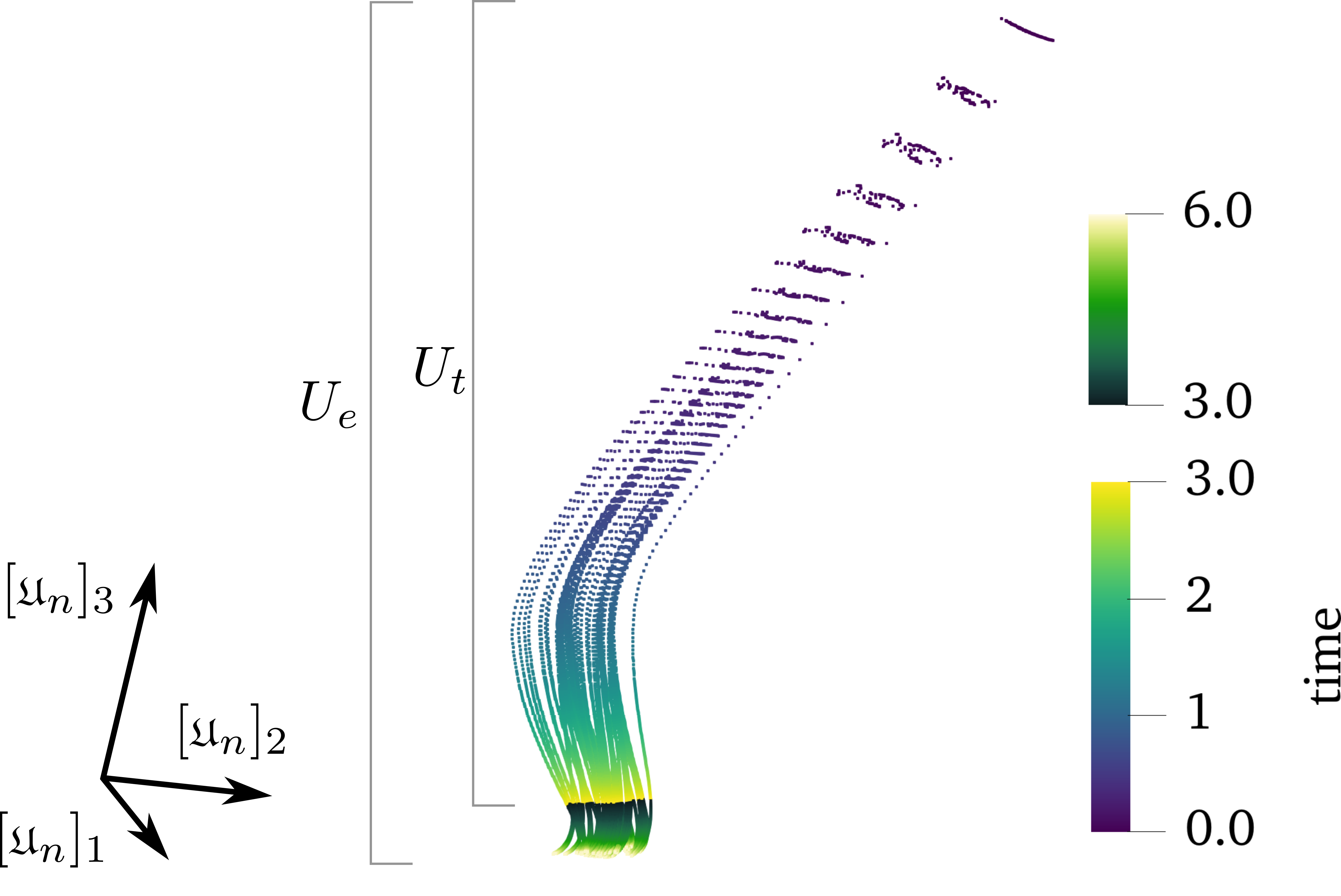}
	\caption{Set of $n$-solutions, $\uut_n^k$, at timesteps $k=0,\ldots,K$, with $\mu$ sampled from the training dataset, along with the extrapolated solutions for unseen future times.}
	\label{fig:latent_set}
\end{figure*}

\subsection{Test 4: Partial Differential Equations - Burgers' model}\label{sec:case_4_burgers}

\rev{
The purpose of this test is to illustrate an application to a PDE. 
We consider here the one-dimensional viscous Burgers' model:
\begin{align}\label{eq:burgers}
\begin{aligned}
	&\dot u + u \partialder{u}{x} -\mu_1\secondpartialder{u}{x} = 0, & t \in (0,T], \ x \in (a,b), \\
	&u(0, x) = \sin(2\pi \mu_2 (x-a)/(b-a)) + 1, &
\end{aligned}	
\end{align}
where $\mu_1 \sim \mathcal{U}(0.0001, 0.001)$ denotes the viscosity coefficient, $\mu_2 \sim \mathcal{U}(1,1.5)$ is the parameter affecting the initial conditions, and $a=0$, $b=1$, $T=1$. Periodic boundary conditions are imposed at the boundaries.
The viscosity is chosen to be small, which ensures the continuity of the solution while it may induce the formation of strong spatial and temporal gradients.

The PDE \eqref{eq:burgers} is discretized in space using centered finite-difference schemes for both the advection and diffusion terms, on a fine grid consisting of $1000$ nodes over $[a,b]$. We emphasize that the spatial discretization itself is not the focus of this test, rather, the aim is to demonstrate a possible application and effectiveness of the reduced model to a PDE setting. The resulting system of ODEs is then integrated in time using a fifth-order Runge--Kutta method with a timestep size of $0.005$. For simplicity, the discrete solution is subsequently downsampled in space to $50$ points. Therefore, the \Node\ has size $N=50$. \\
For a qualitative visualization of the underlying dynamics, see Fig.~\ref{fig:test_4_snaps}, which shows three time snapshots of \eqref{eq:burgers} for two different parameter configurations. \\
The neural network architectures are reported in \Cref{sec:architectures}, Tab.~\ref{tab:case_4_burgers}. They are trained on a dataset consisting of 200 parameter instances and tested on an independent dataset of 50 points. The fully data-driven training strategy is adopted, with a total of 500 epochs. \\
After training, the solution obtained from the DRE is compared against the full-order model. A first qualitative comparison is shown in Fig.~\ref{fig:test_4_snaps}, where an excellent agreement between the two solutions can be observed.

The average relative error, $e_\text{ave, rel}$, defined in \eqref{eq:case_chemistry_errs}, is reported in the right-most panel of Fig.~\ref{fig:test_4_err}. The results indicate a good level of accuracy, with an error of approximately $2\%$ over the entire training time interval. In contrast, the error increases rapidly for later times. This behavior can be explained by the fact that the functions are approximated over a prescribed domain but subsequently evaluated outside this domain, as discussed in \Cref{sec:case_3_chemistry} and illustrated in Fig.~\ref{fig:latent_set}.

Furthermore, in Fig.~\ref{fig:test_4_err}, we also report the initial conditions for selected values of $\mu_2$, together with the time evolution of two components of $\uut_N$, corresponding to $x=0$ and $x=0.5$, for the same parameter value. For graphical clarity, only two components are displayed. The strong nonlinearity of the problem is evident from the rapid temporal variations of the solution. Nevertheless, the reduced-order model is able to accurately track the time evolution.
}
\begin{figure}[H]	
    \centering
    \hspace{-50mm}
    \begin{minipage}{0.9\linewidth}
        \raggedleft
        \begin{turn}{90}
            $\mu = (0.00155,\, 1.03)$   
        \end{turn}
        \includegraphics[width=0.25\linewidth]{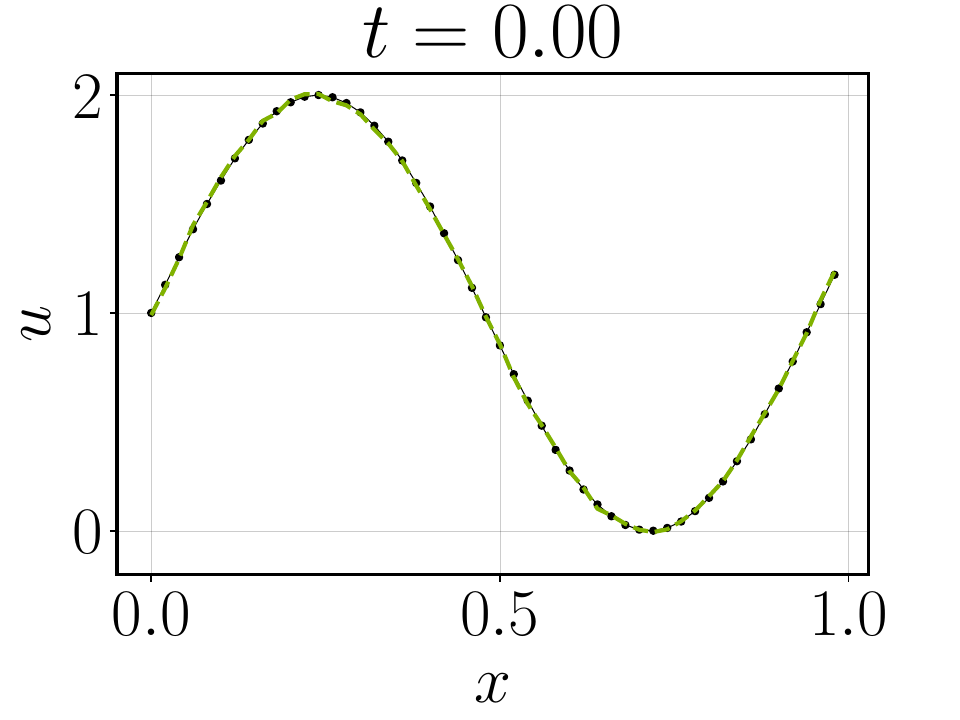}
        \includegraphics[width=0.25\linewidth]{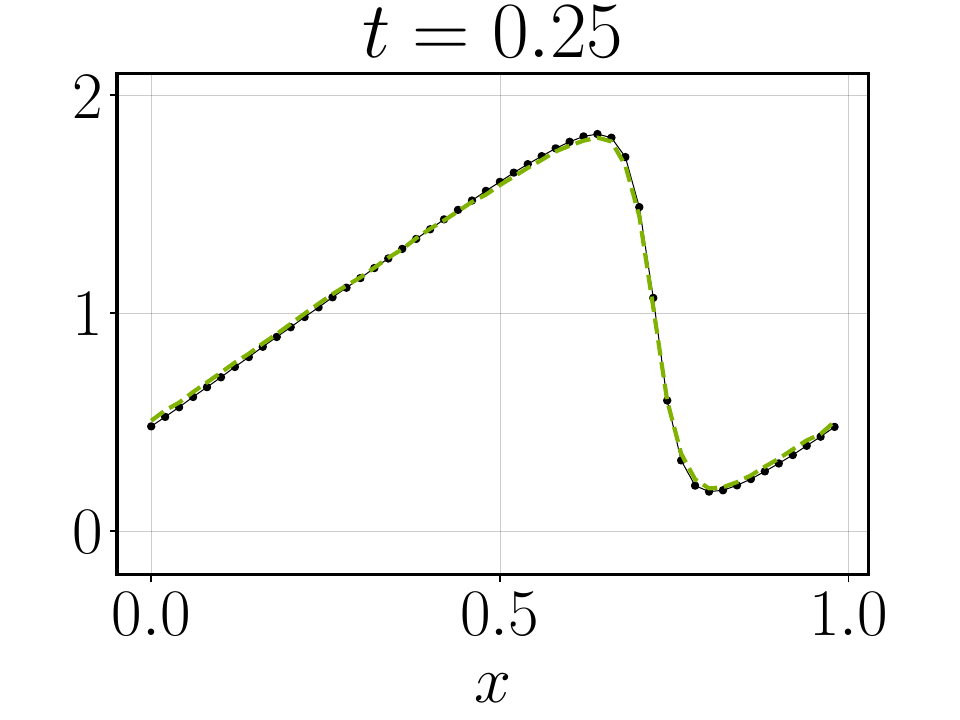}
        \includegraphics[width=0.25\linewidth]{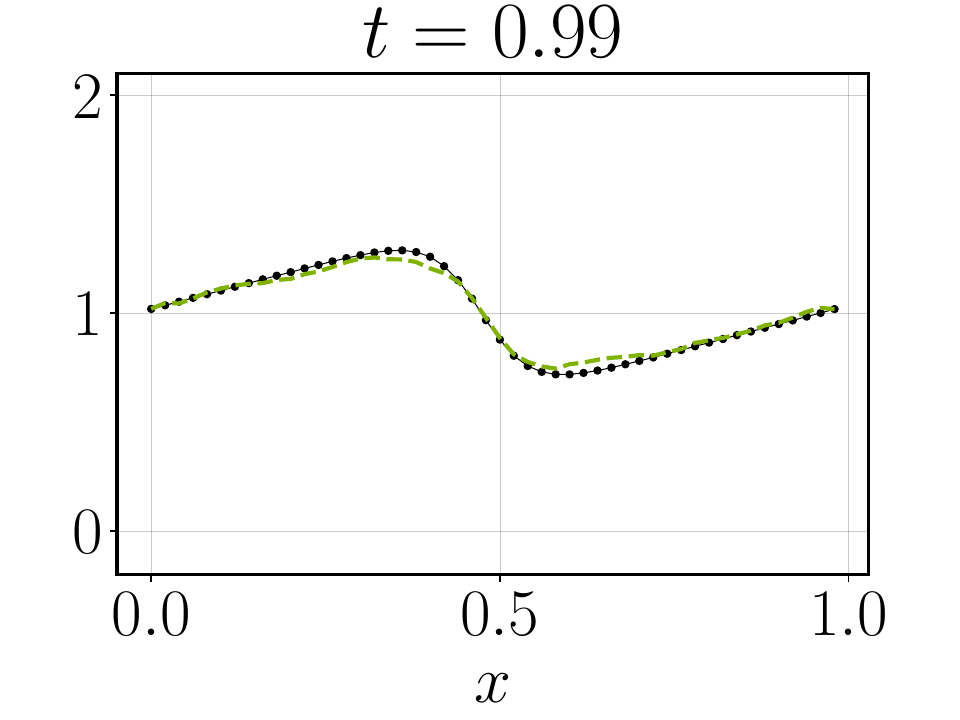}
    \end{minipage}
    \includegraphics[width=0.08\linewidth]{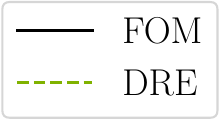}\\
    \vspace{5mm}
    \hspace{2mm}
    \begin{minipage}{0.9\linewidth}
        \begin{turn}{90}
            $\mu = (0.000444,\, 1.45)$   
        \end{turn}
        \includegraphics[width=0.25\linewidth]{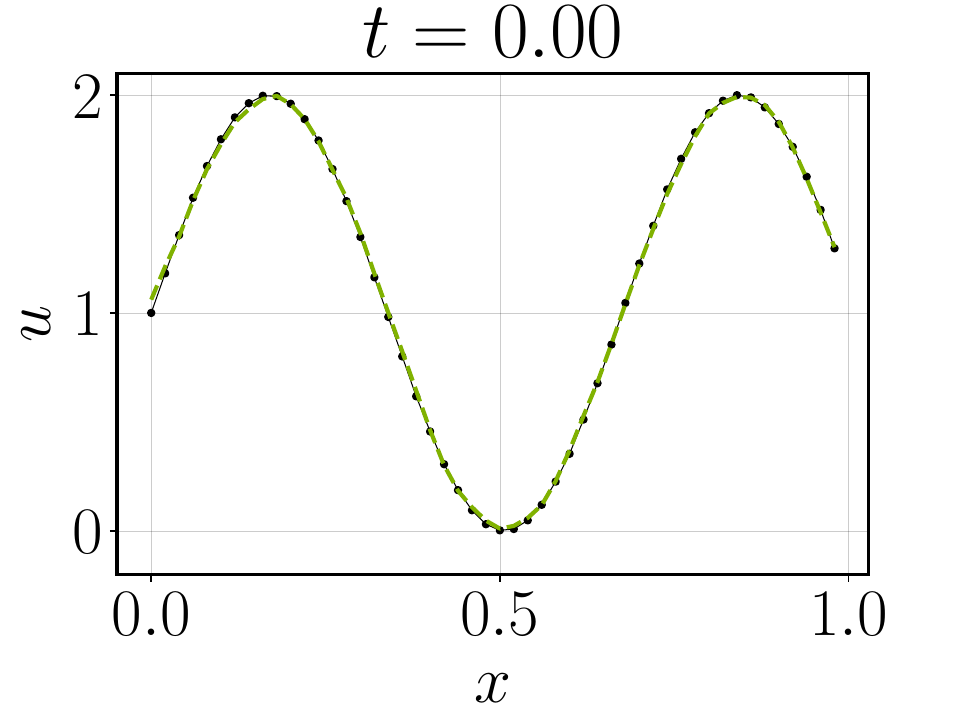}
        \includegraphics[width=0.25\linewidth]{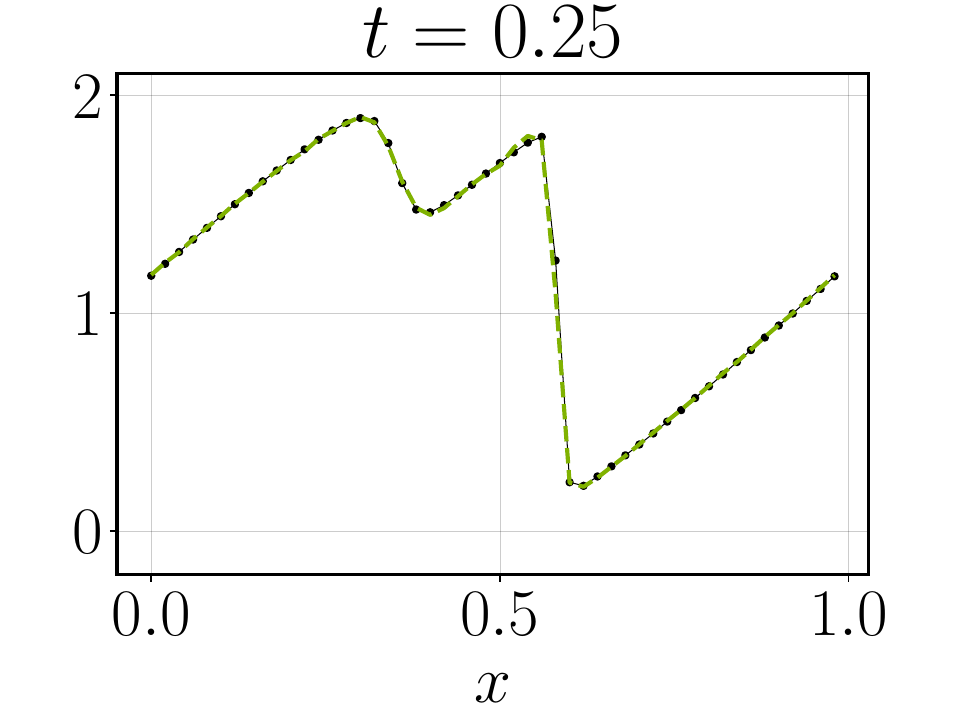}
        \includegraphics[width=0.25\linewidth]{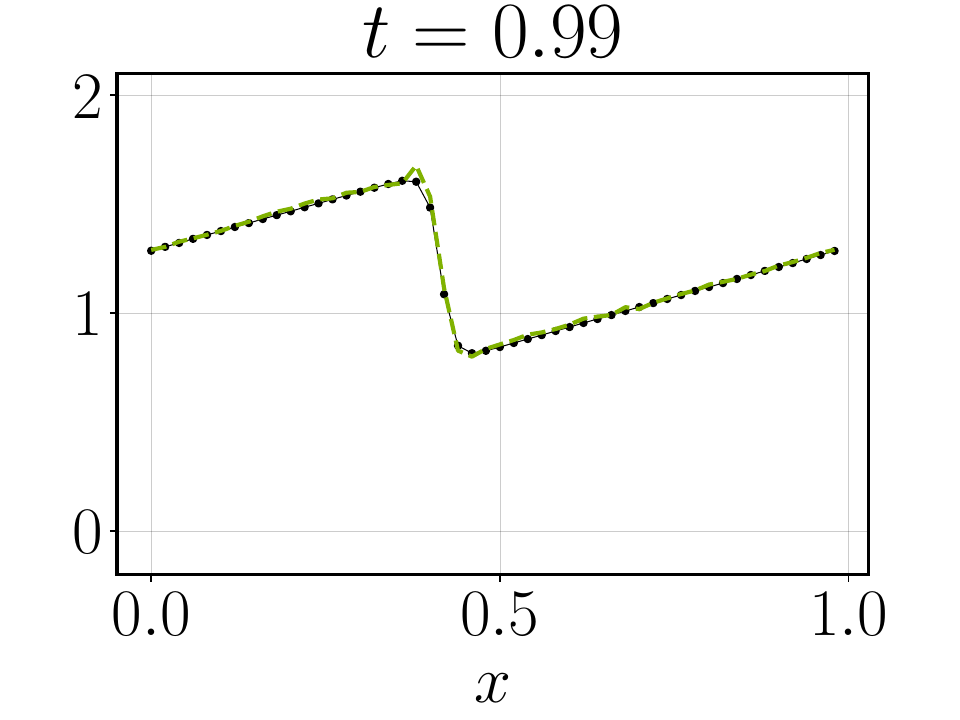}
    \end{minipage}
    \caption{\rev{$u_N$ and $\uut_N$ at three time steps for two values of $\mu$. This figure highlights the variability of the dataset and the strong agreement between the FOM and DRE solutions, despite the strong gradients.}}
    \label{fig:test_4_snaps}
\end{figure}
\begin{figure}[H]
    \centering
    \subfloat[Initial conditions]{\includegraphics[width=0.258\linewidth]{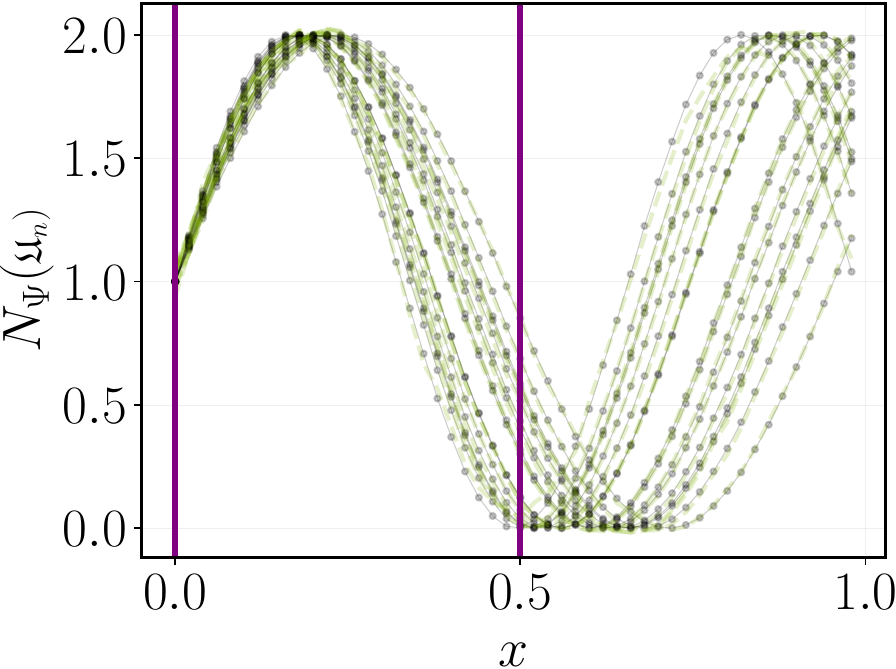}}\hspace*{3mm}
    \subfloat[{$\left.\uut_N\right|_{x=0}$}]{\includegraphics[width=0.245\linewidth]{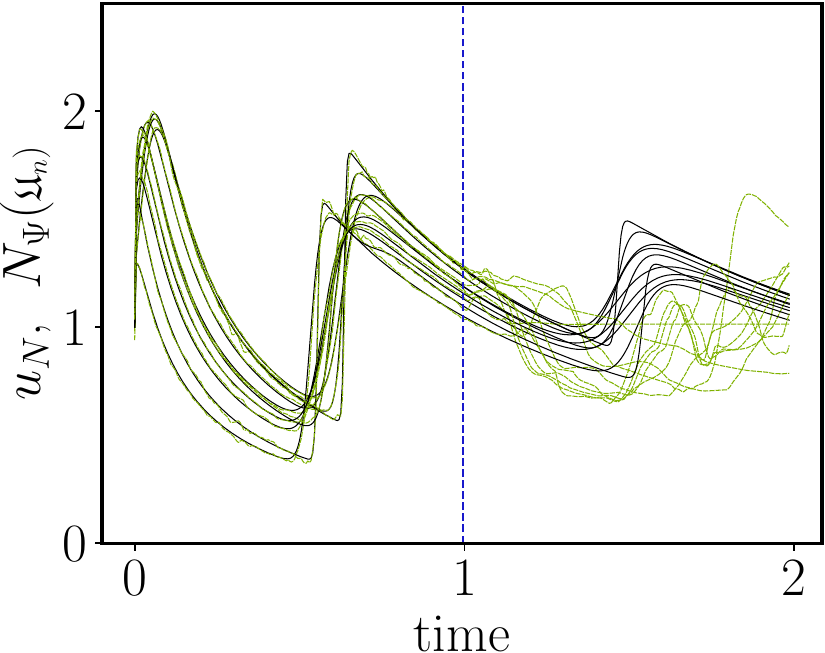}}\hspace*{3mm}
    \subfloat[{$\left.\uut_N\right|_{x=0.5}$} and error]{\includegraphics[width=0.3\linewidth]{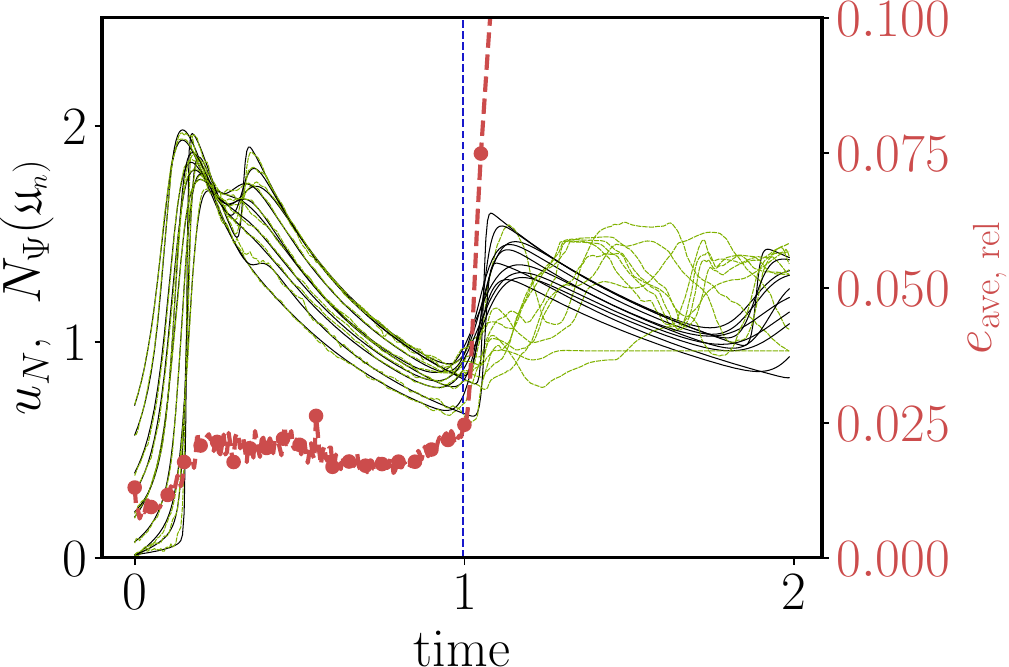}}\hspace*{2mm}
    \begin{minipage}{0.1\linewidth}
        \vspace{-45mm}
        \includegraphics[width=1\linewidth]{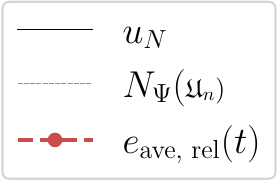}
    \end{minipage}
    \caption{\rev{Panel \textbf{(a)}: initial conditions for varying values of $\mu_2$ and, indicated by vertical purple lines, the components of $N_{\Psi}(\uut_N)$ displayed in the adjacent panels.
    	Panels \textbf{(b)} and \textbf{(c)}: time evolution of selected components of $u_N$. Solid black lines denote the reference solution obtained by accurately solving the \Node, while dashed green lines represent the reconstructed solution $N_{\Psi}(\uut_n)$. The two solutions show good visual agreement. The training data, and consequently the learned manifold, are restricted to the time interval $t \leq 1$, which is indicated by a vertical blue dashed line.
    	Panel \textbf{(c)}: the red dashed line indicates the relative error, as defined in \eqref{eq:case_chemistry_errs}.}}
    \label{fig:test_4_err}
\end{figure}

\subsection{Discussion on computational costs}

\rev{
In this section we comment on the computational cost of the reduced-order model and its different training strategies. Table~\ref{tab:costs} provides a concise overview of the main steps, from data generation to the evaluation of the reduced model.

In broad terms, the generation and use of the reduced-order model are organized into offline and online phases. The former includes data generation and network training, while the latter concerns the actual use of the reduced model. In general, the offline phase is the most expensive one, and the overall effectiveness of the reduced model strongly depends on the objective of the task at hand and, in particular, for how many new values of the parameters the reduced model will be evaluated. An example of a multi-query analysis is the gradient-free optimization procedure, where repeated evaluations of the model for a large number of parameter values are required \cite{Ballini2024}.

We begin by discussing some aspects of the offline phase. Since the method is data-driven, dataset generation is necessary. Its cost strongly depends on the specific problem under consideration and on the purpose of the analysis. When the problem of interest is a PDE, as in the example discussed in \Cref{sec:case_4_burgers}, the initial step consists in the construction of the \Node, that is, the discretization of the problem with respect to all variables except the time, which typically correspond to the spatial coordinates. This step can be particularly expensive, especially when the problem is non-affine, as this requires the reconstruction of the \Node\ for each parameter value; see, among others, \cite{Hesthaven2016, Ballini2024, Ballini2026a}.

Once the \Node\ has been constructed, it must be solved over the prescribed time interval and for a sufficiently large number of parameter samples in order to guarantee an adequate training of the neural networks. The number of samples is problem-dependent and is dictated by the desired accuracy of the results. Time integration must be performed sequentially, and each single timestep evaluation has a cost which depends on the numerical solver employed. A major influence on this cost arises from the choice between explicit and implicit schemes: the former are typically faster but require small timestep size for stability reasons, whereas the latter are generally more expensive per step but allow for larger timestep size. 

Based on the numerical tests presented in this section, the number of timesteps used to generate the full-order solutions appears to be abundant compared to those effectively required for accurately training the neural networks in the desired time range; indeed, a downsampling of the timesteps could be beneficial. Therefore, the reduced model does not seem to impose any limitation on the time evolution of the full-order model, as the settings used to properly solve the full-order model already provide sufficient data to accurately train the networks.

The construction of the reduced model proceeds with the training of the neural networks. The main factors influencing the training time are the desired accuracy, which is directly related to the number of epochs, and the network architecture (smaller networks generally leading to faster training), the optimization algorithm, and the chosen training strategy, namely semi-datadriven versus fully-datadriven. While the former offers the important advantage of accurately approximating $F_n$, enabling time super-resolution, it is intrinsically slower than the fully-datadriven approach. Indeed, in the semi-datadriven setting the autoencoder and the network $N_{F_n}$ must be trained in two distinct phases: first the autoencoder, and subsequently $N_{F_n}$, as detailed in \Cref{sec:train}. In practice, depending on the numerical cases presented, the semi-datadriven approach requires a computational time that is, roughly speaking, about twice that of the fully-datadriven one.

Once the offline phase is completed, the online phase involves comparatively inexpensive computations. In particular, it includes the time evolution of the \node, which is computationally cheap due to the small dimension $n$, even when small timesteps are employed. The time evolution starts from the initial condition, which is obtained by a feed-forward evaluation of the encoder network applied to $u_{N0}$. The high-dimensional solution is then reconstructed through feed-forward evaluations of the decoder network to obtain $\uut_N$ at all, or possibly only $\overline{m} \leq N_\text{times}$, timesteps. This task is trivially parallelizable, and its computational time is typically negligible when compared to the overall time.

Finally, we highlight that the computational cost of the networks is affected by the dimensions $N$ and $n$. In particular, the cost, $C$, of the feed-forward evaluation of a generic network scales as
$\textstyle C = \mathcal{O}\!\left(\sum_{\ell=1}^{L} n_{\ell} n_{\ell-1}\right)$.
For the encoder, we have $n_0 = N$ and, unless the network is linear, possibly some intermediate layers of comparable size. The final layer has dimension $n$. Similarly, for the decoder we have $n_L = N$, together with possibly other preceding layers of similar size. It follows that, for both the encoder and the decoder, the computational cost depends on both $N$ and $n$:
$C_\text{enc} = C_\text{enc}(N,n) = \mathcal{O}(N^2)$ (or $\mathcal{O}(Nn)$ in the linear case),
$C_\text{dec} = C_\text{dec}(N,n) = \mathcal{O}(N^2)$ (or $\mathcal{O}(Nn)$ in the linear case). \\
Conversely, the layers of $N_{F_n}$ depend only on the reduced dimension $n$ and the number of parameters, since $n_0 = n + a$ and $n_L = n$. Therefore, the cost associated with $N_{F_n}$ is
$C_{F_n} = C_{F_n}(n) = \mathcal{O}(n^2)$. \\
Since backpropagation has a computational cost of the same order of magnitude as the feed-forward evaluation, the training of the autoencoder is intrinsically limited by the problem size, scaling as $N^2$ (or $Nn$ in the linear case). This step therefore represents the most computationally expensive component of the proposed method. The reader is referred to Tab.~\ref{tab:costs} for an indication of the order of magnitude of the computational costs associated with each step involving the neural networks.
\begin{table}[H]
    \centering
    \begin{tabular}{l p{0.55\textwidth} p{0.2\textwidth} p{0.2\textwidth}}
        & \textbf{Common} & \textbf{Semi-datadriven} & \textbf{Fully-datadriven} \\
        \cline{2-4}
        \noalign{\vskip 3pt}
        %
        %
        \multirow{6}{*}{\rotatebox{90}{\textbf{offline}}}
        & $\bullet$ If PDE, space discretization and assembly of ODE.
        & & \\
        & $\bullet$ One timestep evolution of \Node\ using explicit or implicit method.
        & \multirow{5}{=}{\raggedright%
            $\bullet$ Training $N_{\Psi'}$, $N_{\Psi}$.
            Cost $=\mathcal{O}(n_e N^2)$. \\
            $\bullet$ Training $N_{F_n}$. Cost $=\mathcal{O}(n_e N^2)$.}
        & \multirow{4}{=}{\raggedright%
            $\bullet$ Training $N_{\Psi'}$, $N_{\Psi}$, and $N_{F_n}$ concurrently.
            Cost $=\mathcal{O}(n_e N^2)$.} \\
        & $\bullet$ Repeat the previous point for $N_\text{time}$ times.
        & & \\
        & $\bullet$ Repeat for every parameter. If the problem is non-affine, repeat assembly of the ODE and following steps for each parameter.
        & & \\
        & & & \\[4pt]
        %
        %
        \multirow{3}{*}{\rotatebox{90}{\textbf{online}}}
        & $\bullet$ Reduction of initial condition:
          $\uut_{n0} = N_{\Psi'}(u_{N0})$. Cost $=\mathcal{O}(N^2)$.
        & & \\
        & $\bullet$ Evolve in time \node. Cost $=\mathcal{O}(P N_\text{time} n^2)$.
        & & \\
        & $\bullet$ Reconstruction $\uut_N = N_{\Psi}(\uut_n)$
          for preselected $\overline{m}$ timesteps. Cost $=\mathcal{O}(\overline{m} N^2)$.
        & & \\
    \end{tabular}
    \caption{Main steps for the offline and online phases with associated
             computational costs. The number of training epochs is $n_e$. We recall that $P$ is the number of steps of the multistep methods utilized.}
    \label{tab:costs}
\end{table}
}

\section{Conclusions}
In this work, we have presented a data-driven reduced-order modeling method entirely based on neural networks. The methodology can be applied to generic ODEs, such as those resulting from the spatial discretization of PDEs \rev{as presented in the last test case in \Cref{sec:case_4_burgers}}. 

Several contributions have been made. First, we proved that the set of solutions corresponding to varying parameters forms a $m$-manifold, which can be parametrized using $n \geq m$ coordinates. We also established a theoretical link between the number of parameters and the minimum latent dimension required for an exact representation.

Second, we provided guidance on the construction of an autoencoder composed of fully connected neural networks, ensuring that no information is lost during the dimensionality reduction phase, i.e., achieving $\Varepsilon_1 = 0$. 

Third, we identified the \node\ that governs the reduced dynamics and proved \rev{its} well-posedness. This system can be solved rapidly in place of the original \Node, which is potentially expensive. The relationship between these two systems requires particular treatment, and, consequently, we focused on establishing conditions for stability and convergence such as $\Delta t_\text{train}, \Delta t \to 0$ and $|\myw| \to \infty$, accounting for both time discretization errors and neural network approximation errors.

Fourth, we proposed and analyzed two training strategies, each with distinct strengths and limitations. It is worth highlighting that the semi-data-driven approach enables fine temporal resolution after training, i.e., it allows for $\Delta t < \Delta t_\text{train}$.

The methodology was assessed in \rev{four} test cases: one illustrates stability properties, one shows convergence over time, one evaluates performance in a general application setting, \rev{and the last shows an application to a highly nonlinear PDE resulting in strongly nonlinear solution manifold}. The numerical results are in agreement with the theoretical findings. In particular, due to the nonlinear nature of both the encoder and decoder, the method is capable of accurately reconstructing the nonlinear $N$-solution by solving the smallest ODE that allows for a null representation error.

The solid theoretical foundation and promising numerical performance justify the continuation of the research of the presented framework. In the following, we mention some possible future research direction and limitations. We have shown that the injectivity of the solution map is crucial for the representation of $\id_\mathcal{M}$ and that, if absent, it can be recovered by augmenting the data. This observation motivates future investigation of complex manifolds arising from non-injective parameter-to-solution maps.

Moreover, given the analytical complexity of the full stability theory, we have restricted our analysis to relatively simple contexts. An extension of the stability results to more general settings, such as linear multistep or Runge–Kutta schemes, and the definition of the relation autoencoder vs. stability appears particularly interesting. Finally, given the effectiveness of the method, it is interesting to apply it to realistic and engineering-relevant scenarios (see, e.g., \cite{Ballini2025}) \rev{where it is possible to properly assess the effectiveness of the method in terms of computational cost and speedup}.

\paragraph{Declaration of competing interest}
The authors declare that they have no known competing financial interests or personal relationships that could have appeared to influence the work reported in this paper.

\paragraph{Data availability}
\rev{The code is publicly available at
\url{https://github.com/enricoballini/dynamical_reduced_embedding.git}}

\paragraph{Acknowledgements}
Enrico Ballini and Paolo Zunino acknowledge the partial support the European Union’s Euratom research and training programme 2021–2027 under grant agreement No. 101166699 \textit{Radiation safety through biological extended models and digital twins} (TETRIS).
This research is part of the activities of the \textit{Dipartimento di Eccellenza} 2023--2027, Department of Mathematics, Politecnico di Milano.  
Enrico Ballini, Alessio Fumagalli, Luca Formaggia, Anna Scotti, and Paolo Zunino are members of the INdAM Research group GNCS.


\bibliographystyle{abbrv}
\bibliography{all}

\appendix
\appendixsectionformat 


\section{Lipschitz constants and stability}\label{sec:computations}

\paragraph{Lipschitz constant $L_{F_n}$.}
Let $M_{y(x)} = \sup_x \|y(x)\|$ and $L_y$ the Lipschitz constant of the generic function $y$. We have
\begin{align*}
	&\|F_n(u_{n1}) - F_n(u_{n2})\| = \|J_{\Psi'}(\Psi(u_{n1}))F_N(\Psi(u_{n1})) - J_{\Psi'}(\Psi(u_{n2}))F_N(\Psi(u_{n2}))\| \leq \\
    & \leq  \|J_{\Psi'}(\Psi(u_{n1}))F_N(\Psi(u_{n1})) - J_{\Psi'}(\Psi(u_{n1}))F_N(\Psi(u_{n2})) + J_{\Psi'}(\Psi(u_{n1}))F_N(\Psi(u_{n2})) - J_{\Psi'}(\Psi(u_{n2}))F_N(\Psi(u_{n2}))\| \leq \\
    & \leq  \|J_{\Psi'}(\Psi(u_{n1}))[F_N(\Psi(u_{n1})) - F_N(\Psi(u_{n2}))] \| + \|F_N(\Psi(u_{n2}))[J_{\Psi'}(\Psi(u_{n1})) - J_{\Psi'}(\Psi(u_{n2}))]\| \leq \\
	&\leq (M_{J_{\Psi'}}L_{F_N} + M_{F_N}L_{J_{\Psi'}}) \| \Psi(u_{n1}) - \Psi(u_{n2})\| \leq \\
	& \leq \underbrace{(M_{J_{\Psi'}}L_{F_N} + M_{F_N}L_{J_{\Psi'}}) L_\Psi}_{= \overline{L}_{F_n,\mu}} \|u_{n1} - u_{n2}\|,
\end{align*}
where $\overline{L}_{F_n,\mu}$ is a upper bound for $L_{F_n,\mu}$.

\paragraph{Lipschitz constant $L_{F_n}$, scaled \node.}
We scale the equations by $K>0$:
\begin{align}\label{eq:scaled_node}
	\begin{aligned}
		&K\dot{u}_n = K J_{\Psi'} F_N(\Psi(u_n)), \\
		&Ku_n(0) = K \Psi'(u_{N0}).
	\end{aligned}
\end{align}
Replacing $w_n = K u_n$ into \eqref{eq:scaled_node} we have the dynamics of the scaled $n$-solution:
\begin{align}\label{eq:scaled_node_K}
	\begin{aligned}
		&\dot{w}_n = K J_{\Psi'} F_N(\Psi(1/K w_n)) \coloneqq \Tilde{F}_n(w_n), \\
		&w_n(0) = K \Psi'( u_{N0} ),	
	\end{aligned}
\end{align}
from which we derive:
\begin{align*}
	&\|\Tilde{F}_n(w_{n,1})-\Tilde{F}_n(w_{n,2})\| = K \|F_n(1/K w_{n,1})-F_n(1/K w_{n,1})\|\leq \\
	&\leq K L_{F_n} \|1/K w_{n,1}- 1/K w_{n,2}\| = L_{F_n} \|w_{n,1} - w_{n,2}\|.
\end{align*}
Therefore $L_{\Tilde{F}_n} = L_{F_n}$.

A modification of the Lipschitz constant can be made by using a nonlinear $g$, an invertible function with Jacobian $J_g$ and such that $\Psi\circ g^{-1} \circ g \circ\Psi'$ represents the autoencoder $\Psi\circ\Psi'$ but passing through different $n$-solutions. Let now $w_n = g(u_n)$, then: 
\begin{align*}
	&\dot w_n = J_g F_n(g^{-1}(w_n)) \coloneq \Tilde{F}_n(w_n), \\
	&w_n(0) = g(\Psi'(u_{N0})),
\end{align*}
where, again, an abuse of notation is done in the definition of $\Tilde{F}_n$. In general, depending on $g$, $L_{\Tilde{F}_n}$ differs from $L_{F_n}$.



\paragraph{Lyapunov stability.} 
First, we move the problem on the \Node\ to the \node: $\| w_N - u_N \| \leq L_\Psi \| w_n - u_n \|$.
Integrating the \node\ and computing the difference between the reference solution and the perturbed one at time $t$ reads:
\begin{equation*}
	w_n - u_n = \Psi'(u_{N0} + \Delta_0) + \delta_0 - \Psi'(u_{N0}) + \int_{t_0}^t \left(F_n(w_n) + J_{\Psi'}\Delta + \delta - F_n(u_n)\right).
\end{equation*}
Computing the norm and applying the triangular inequality:
\begin{align*}
	&\|w_n - u_n\| \leq \underbrace{\|\Psi'(u_{N0}+\Delta_0) - \Psi'(u_{N0}) + \delta_0\|}_{A} + \underbrace{\| \int_{0}^t J_{\Psi'}(u_N(t)) \Delta\|}_{B} + \underbrace{\|\int_{0}^t \delta\|}_{C} + \underbrace{\|\int_{0}^t F_n(w_n) - F_n(u_n)\|}_{D}.
\end{align*}
Each term can be bounded as follows:
\begin{align*}	
	& A \leq L_{\Psi'}\|\Delta_0\| + \|\delta_0\|, \\
	& B \leq M_{J_{\Psi'}}\|\Delta\| t, \\
	& C \leq \|\delta\| t, \\ 
	& D \leq \int_{0}^t L_{F_n}\|w_n-u_n\|. \\
\end{align*}
Applying the Gr\"onwall's lemma and the Lipschitz continuity of $\Psi$, we get the bound
\begin{align*}
	\|w_N - u_N\| \leq &L_\Psi\Big( L_{\Psi'}\|\Delta_0\| + \|\delta_0\| + \big(M_{J_{\Psi'}}\|\Delta\| + \|\delta\|\big)t \Big) e^{L_{F_n,\mu}t}. 
\end{align*}
Note that the autoencoder affects $L_{F_{n,\mu}}$, see \Cref{th:L_{F_n}}.

\paragraph{Lyapunov stability - Neural networks.}
Let us now consider an unperturbed \Node, so $\Delta = 0$ and $\Delta_0 = 0$. Interpreting $w_n$ and the reconstructed $w_N = N_\Psi(w_n)$ as perturbed solutions due to the use of neural networks, we have:
\begin{align*}
	&\|w_n - u_n\| \leq \underbrace{\|\Psi'(u_{N0}) - \kappa(u_{n0}) - \Psi'(u_{N0})\|}_{E} + \underbrace{\|\int_{0}^t \omega(w_n)\|}_{F} + \|\int_{0}^t F_n(w_n) - F_n(u_n)\|
\end{align*}
We can bound the first two terms with
\begin{align*}
	&E \leq \eps_{\Psi'}, \\
	&F \leq \eps_{F_n}t.
\end{align*}
Therefore, we can apply the Gr\"onwall's lemma. Using the Lipschitz continuity of $\Psi$ we have:
%
\begin{align*}
    &\|w_N - u_N\| \leq L_\Psi \|w_n - u_n\| + \eps_\Psi \leq \\
    &\leq L_\Psi\Big( \eps_{\Psi'} + \eps_{F_n}t \Big) e^{L_{F_n}t} + \eps_\Psi.
\end{align*}

\section{Neural networks' architectures}\label{sec:architectures}
We report in Tab.~\ref{tab:case_1_architecture}  the neural networks' architectures used in \Cref{sec:case_1_a_stability}. With ``Linear(Affine) $x \times y$'' we denote a linear(affine) transformation from $\mathbb{R}^x$ to $\mathbb{R}^y$. 
\begin{table}[H]
	\centering
	\begin{minipage}[t]{0.48\textwidth}
		\centering
		\vspace{-2cm}
		\begin{tabular}{llll}
			& \textbf{layer}        & \textbf{$|\myw|$}    & \textbf{$|\myw|_{\text{tot}}$} \\
			\hline \\
			\multirow{1}{*}{$N_{\Psi'}$} & Linear 3$\times$2  & \multirow{1}{*}{\numprint{6}} & \multirow{3}{*}{\numprint{12}}  \\
			& & & \\
			\multirow{1}{*}{$N_{\Psi}$} & Linear 2$\times$3   & \multirow{1}{*}{\numprint{6}} &                          \\
		\end{tabular}
	\end{minipage}
	\hfill
    \begin{minipage}[t]{0.48\textwidth}
		\centering
		\begin{tabular}{llll}
			& \textbf{layer}        & \textbf{$|\myw|$}    & \textbf{$|\myw|_{\text{tot}}$} \\
			\hline \\
			\multirow{1}{*}{$N_{\Psi'}$} & Linear 3$\times$2 & \multirow{1}{*}{\numprint{6}} & \multirow{6}{*}{\numprint{7678}}  \\
			& & & \\
			\multirow{5}{*}{$N_{\Psi}$} & Linear 2$\times$3, PReLU     & \multirow{5}{*}{\numprint{7672}} &                          \\
			& Affine 3$\times$30, PReLU  		&                     &                          \\
			& 8$\times$Affine 30$\times$30, PReLU  &                   &                          \\
			& Affine 30$\times$3, PReLU  	&                         &                          \\
			& Affine 3$\times$3  			&                         &                          \\
		\end{tabular}
	\end{minipage}
	\caption{Neural networks' architectures used in \Cref{sec:case_1_a_stability}. $|\myw|$ is the number of trainable weights for the single neural network, $|\myw|_{\text{tot}}$ the total number of trainable weights.}
	\label{tab:case_1_architecture}
\end{table}

\vspace{\baselineskip}
We report in Tab.~\ref{tab:case_2_convergence} the neural networks' architectures used in \Cref{sec:case_2_time_convergence}. $N_{F_n}$ is trained using two different strategies, as detailed in \Cref{sec:case_2_time_convergence}.
\begin{table}[H]
	\centering
    \begin{minipage}[t]{0.48\textwidth}
    \vspace{0pt}
	\begin{tabular}{llll}
		& \textbf{layer}        & \textbf{$|\myw|$}    & \textbf{$|\myw|_{\text{tot}}$} \\
		\hline \\
		\multirow{3}{*}{$N_{\Psi'}$} & Affine 3$\times$10, ELU & \multirow{3}{*}{\numprint{390}} & \multirow{13}{*}{\numprint{6974}}  \\
		& 3$\times$Affine 10$\times$10, ELU & & \\
        & Linear 10$\times$2 & & \\
        & & & \\
        & & & \\
		\multirow{5}{*}{$N_{\Psi}$} & Linear 2$\times$3, PReLU     & \multirow{5}{*}{\numprint{427}} &                          \\
		& Affine 3$\times$10, PReLU  		&                     &                          \\
		& 4$\times$Affine 10$\times$10, PReLU  &                   &                          \\
		& Affine 10$\times$3, PReLU  	&                         &                          \\
		& Affine 3$\times$3  			&                         &                          \\
		& & & \\
		\multirow{3}{*}{$N_{F_n}$} & Affine 3$\times$10, PReLU     & \multirow{3}{*}{\numprint{6157}} &                          \\
		& 10$\times$Affine 24$\times$24, PReLU  &                   &                          \\
		& Affine 24$\times$2  			&                         &                          
	\end{tabular}
    \end{minipage}
    \hfill
    \begin{minipage}[t]{0.48\textwidth}
    \vspace{0pt}
	\begin{tabular}{llll}
		& \textbf{layer}        & \textbf{$|\myw|$}    & \textbf{$|\myw|_{\text{tot}}$} \\
		\hline \\
		\multirow{1}{*}{$N_{\Psi'}$} & Linear 3$\times$2 & \multirow{1}{*}{\numprint{6}} & \multirow{12}{*}{\numprint{8463}}  \\
		& & & \\
		\multirow{5}{*}{$N_{\Psi}$} & Linear 2$\times$3, PReLU     & \multirow{5}{*}{\numprint{7672}} &                          \\
		& Affine 3$\times$30, PReLU  		&                     &                          \\
		& 8$\times$Affine 30$\times$30, PReLU  &                   &                          \\
		& Affine 30$\times$3, PReLU  	&                         &                          \\
		& Affine 3$\times$3  			&                         &                          \\
		& & & \\
		\multirow{4}{*}{$N_{F_n}$} & Affine 3$\times$3, PReLU     & \multirow{4}{*}{\numprint{785}} &                          \\
		& 8$\times$Affine 9$\times$9, PReLU  &                   &                          \\
		& Affine 9$\times$3, PReLU  	&                         &                          \\
		& Affine 3$\times$2  			&                         &                          \\
	\end{tabular}
    \end{minipage}
	\caption{Neural networks' architectures used in \Cref{sec:case_2_time_convergence}. $|\myw|$ is the number of trainable weights for the single neural network, $|\myw|_{\text{tot}}$ the total number of trainable weights. \rev{Left table: networks used in \Cref{sec:test_convergence_analytical}. Right table: networks used in \Cref{sec:test_convergence_generic}}.}
	\label{tab:case_2_convergence}
\end{table}

\vspace{\baselineskip}
We report in Tab.~\ref{tab:case_3_chemistry} the neural networks' architectures used in \Cref{sec:case_3_chemistry}. In this case, the activation functions used in the encoder are smooth to ensure a continuous $J_{\Psi'}$ and, consequently, a continuous $F_n$.
\begin{table}[H]
	\centering
	\begin{tabular}{llll}
		& \textbf{layer}        & \textbf{$|\myw|$}    & \textbf{$|\myw|_{\text{tot}}$} \\
		\hline \\
		\multirow{3}{*}{$N_{\Psi'}$} & Affine 20$\times$21, ELU & \multirow{3}{*}{\numprint{2814}} & \multirow{14}{*}{\numprint{9515}}  \\
		& 5$\times$Affine 21$\times$21, ELU  &                   &                          \\
		& Linear 21$\times$ 3, ELU  &                   &                          \\
		& & & \\
		\multirow{5}{*}{$N_{\Psi}$} & Linear 3$\times$20, PReLU     & \multirow{5}{*}{\numprint{4412}} &                          \\
		& Affine 20$\times$21, PReLU  		&                     &                          \\
		& 7$\times$Affine 21$\times$21, PReLU  &                   &                          \\
		& Affine 21$\times$20, PReLU  	&                         &                          \\
		& Affine 20$\times$20  			&                         &                          \\
		& & & \\
		\multirow{4}{*}{$N_{F_n}$} & Affine 5$\times$20, PReLU     & \multirow{4}{*}{\numprint{2289}} &                          \\
		& 5$\times$Affine 20$\times$20, PReLU  &                   &                          \\
		& Affine 9$\times$3, PReLU  	&                         &                          \\
		& Affine 3$\times$3  			&                         &                          \\
	\end{tabular}
	\caption{Neural networks' architectures used in \Cref{sec:case_3_chemistry}. $|\myw|$ is the number of trainable weights for the single neural network, $|\myw|_{\text{tot}}$ the total number of trainable weights.}
	\label{tab:case_3_chemistry}
\end{table}

\rev{
\vspace{\baselineskip}
We report in Tab.~\ref{tab:case_4_burgers} the neural networks' architectures used in \Cref{sec:case_4_burgers}. 
\begin{table}[H]
	\centering
	\begin{tabular}{llll}
		& \textbf{layer}        & \textbf{$|\myw|$}    & \textbf{$|\myw|_{\text{tot}}$} \\
		\hline \\
		\multirow{2}{*}{$N_{\Psi'}$} & 2$\times$Affine 50$\times$50, ELU & \multirow{2}{*}{\numprint{5250}} & \multirow{8}{*}{\numprint{28097}}  \\
		& Linear 50$\times$ 3      &                   &                          \\
		& & & \\
		\multirow{3}{*}{$N_{\Psi}$} & Linear 3$\times$50, PReLU     & \multirow{3}{*}{\numprint{20558}} &                          \\
		& 8$\times$Affine 50$\times$50, PReLU  &                   &                          \\
		& Affine 50$\times$50  			&                         &                          \\
		& & & \\
		\multirow{3}{*}{$N_{F_n}$} & Affine 5$\times$20, PReLU     & \multirow{3}{*}{\numprint{2289}} &                          \\
		& 5$\times$Affine 20$\times$20, PReLU  &                   &                          \\
		& Affine 20$\times$3  			&                         &                          \\
	\end{tabular}
	\caption{Neural networks' architectures used in \Cref{sec:case_4_burgers}. $|\myw|$ is the number of trainable weights for the single neural network, $|\myw|_{\text{tot}}$ the total number of trainable weights.}
	\label{tab:case_4_burgers}
\end{table}
}

\clearpage

\end{document}